\documentclass[10pt]{article}%

\usepackage[show]{ed}
\usepackage{tikz}
\usepackage{pgf}
\usetikzlibrary{arrows}
\usetikzlibrary{calc}
\usetikzlibrary{decorations.pathmorphing}
\usepackage[protrusion=true,expansion=true]{microtype}
\usepackage{xfrac}    
\usepackage[strict]{changepage}

\usepackage{faktor}
\usepackage{soul}
\setstcolor{red}
\setulcolor{red}

\usepackage{rotating} % needed for \begin{sidewaysfigure}\begin{sidewaystable}

\usepackage{mathrsfs}
\DeclareMathAlphabet{\mathpzc}{OT1}{pzc}{m}{it}

\usepackage{romannum}

\usepackage{color}
\usepackage{amsmath}
\usepackage{graphicx}
\usepackage{amsfonts}
\usepackage{amssymb}%
\usepackage{latexsym}
\usepackage{psfrag}
\usepackage{subfigure}
\usepackage{accents}

\newtheorem{theorem}{Theorem}[section]
\newtheorem{corollary}[theorem]{Corollary}
\newtheorem{definition}[theorem]{Definition}
\newenvironment{proof}[1][Proof]{\noindent \emph{#1.} }
{\hfill \ \rule{0.5em}{0.5em}}
\newtheorem{lemma}[theorem]{Lemma}
\newtheorem{proposition}[theorem]{Proposition}

\newtheorem{assumption}[theorem]{Assumption}

\numberwithin{equation}{section}
\numberwithin{table}{section}
\numberwithin{figure}{section}

\usepackage[a4paper]{geometry}
\newtheorem{remark}[theorem]{Remark}
\newtheorem{example}[theorem]{Example}

\newcommand{\cA}{{\cal A}}

\newcommand{\cE}{{\cal E}}
\newcommand{\cF}{{\cal F}}
\newcommand{\cH}{{\cal H}}

\newcommand{\cS}{{\cal S}}

\newcommand{\cD}{{\cal D}}

\newcommand{\bx}{x}%\mathbf{x}}

    \newcommand\quotient[2]{
        \mathchoice
            {% \displaystyle
                \text{\raise1ex\hbox{$#1$}\Big/\lower1ex\hbox{$#2$}}%
            }
            {% \textstyle
                #1\,/\,#2
            }
            {% \scriptstyle
                #1\,/\,#2
            }
            {% \scriptscriptstyle  
                #1\,/\,#2
            }
    }

\newcommand{\re}{{\rm e}}
\newcommand{\ri}{{\rm i}}
\newcommand{\rd}{{\rm d}}

\newcommand{\beq}{\begin{equation}}
\newcommand{\eeq}{\end{equation}}
\newcommand{\beqs}{\begin{equation*}}
\newcommand{\eeqs}{\end{equation*}}
\newcommand{\bit}{\begin{itemize}}
\newcommand{\eit}{\end{itemize}}
\newcommand{\ben}{\begin{enumerate}}
\newcommand{\een}{\end{enumerate}}
\newcommand{\bal}{\begin{align}}
\newcommand{\eal}{\end{align}}
\newcommand{\bals}{\begin{align*}}
\newcommand{\eals}{\end{align*}}
\newcommand{\bse}{\begin{subequations}}
\newcommand{\ese}{\end{subequations}}
\newcommand{\bpr}{\begin{proposition}}
\newcommand{\epr}{\end{proposition}}
\newcommand{\bre}{\begin{remark}}
\newcommand{\ere}{\end{remark}}
\newcommand{\bpf}{\begin{proof}}
\newcommand{\epf}{\end{proof}}
\newcommand{\ble}{\begin{lemma}}
\newcommand{\ele}{\end{lemma}}
\newcommand{\bco}{\begin{corollary}}
\newcommand{\eco}{\end{corollary}}
\newcommand{\bex}{\begin{example}}
\newcommand{\eex}{\end{example}}
\newcommand{\bth}{\begin{theorem}}
\newcommand{\enth}{\end{theorem}}

\newcommand{\Rea}{\mathbb{R}}

\newcommand{\eps}{\varepsilon}

\newcommand{\pdiff}[2]{\frac{\partial #1}{\partial #2}}

\newcommand{\tendi}{\rightarrow \infty}

\def\XXint#1#2#3{{\setbox0=\hbox{$#1{#2#3}{\int}$}
     \vcenter{\hbox{$#2#3$}}\kern-.5\wd0}}

\usepackage{hyperref}
\counterwithin{figure}{section}
\definecolor{myblue}{rgb}{0,0,0.6}
\hypersetup{colorlinks=true,linkcolor=myblue,citecolor=myblue,filecolor=myblue,urlcolor=myblue}
\newcommand*{\N}[1]{\left\|#1\right\|}

\allowdisplaybreaks[4]
\newcommand{\tfa}{\text{ for all }}
\newcommand{\tfor}{\text{ for }}

\newcommand{\tif}{\text{ if }}
\newcommand{\tin}{\text{ in }}
\newcommand{\ton}{\text{ on }}
\newcommand{\tas}{\text{ as }}
\newcommand{\tand}{\text{ and }}
\newcommand{\tst}{\text{ such that }}
\newcommand{\tfind}{\text{ find }}
\newcommand{\tthen}{\text{ then }}

\newcommand{\Hilb}{\cH}

\newcommand{\vertiii}[1]{{\left\vert\kern-0.25ex\left\vert\kern-0.25ex\left\vert #1
    \right\vert\kern-0.25ex\right\vert\kern-0.25ex\right\vert}}

\definecolor{jwcol}{RGB}{27, 137, 18}  %{rgb}{1,0.88,0.21} changed color for visibility (david)

\definecolor{dalcol}{rgb}{0.8,0,0}

\definecolor{escol}{rgb}{0,0,0.8}
\definecolor{estcol}{rgb}{0,0.5,0}
\definecolor{esnewcol}{rgb}{0,0.5,0}

\newcommand{\obstacle}{{\Omega}}

\newcommand{\Imag}{{\rm Im}}

\newcommand{\supp}{{\rm supp}}

\newcommand{\abs}[1]{{\left\lvert{#1}\right\rvert}}
\newcommand{\norm}[1]{{\left\lVert{#1}\right\rVert}}
\newcommand{\ang}[1]{{\left\langle{#1}\right\rangle}}

\newcommand{\RR}{\mathbb{R}}
\newcommand{\ZZ}{\mathbb{Z}}

\newcommand{\Cqo}{{C_{\rm qo}}}

\newcommand{\Ccont}{{C_{\rm cont}}}

\newcommand{\mythmname}[1]{\textbf{\emph{(#1)}}}

\newcommand{\uhigh}{u_{H^2}}
\newcommand{\ulow}{u_{\mathcal A}}
\newcommand{\vhigh}{v_{H^2}}
\newcommand{\vlow}{v_{\mathcal A}}

\newcommand{\Pilow}{\Pi_{\rm Low}}%\flat}}
\newcommand{\Pihigh}{\Pi_{\rm High}}%\sharp}}
\newcommand{\Pihightheta}{\Pi_{\rm High, \theta}}%\sharp}}
\newcommand{\Lap}{\Delta}
\newcommand{\hsc}{\hbar}

\newcommand{\WFh}{\operatorname{WF}_{\hsc}}
\newcommand{\pa}{\partial}

\newcommand{\Op}{{\rm Op}}

\newcommand{\hilbert}{\mathcal{H}}
\newcommand{\domain}{\mathcal{D}}
\newcommand{\torus}{\mathbb{T}}

\newcommand{\F}{\mathcal{F}}
\newcommand{\Tbar}{\overline{T}}
\DeclareMathOperator{\RC}{\mathsf{RC}}

\newcommand{\RlocA}{R_{_{\rm  \Romannum{4}}}}
\newcommand{\RlocB}{R_{_{\rm \Romannum{3}}}}
\newcommand{\RfarA}{R_{_{\rm \Romannum{1}}}}
\newcommand{\RfarB}{R_{_{\rm \Romannum{2}}}}

\newcommand{\residual}{O(\hsc^\infty)_{\mathcal D_{ \hsc}^{\sharp,-\infty} \rightarrow \mathcal D_{\hsc}^{\sharp,\infty}}}
\newcommand{\residualD}{O(\hsc^\infty)_{\cD^{\sharp, \infty}}}%\residual}

\newcommand{\vH}{v_{\rm High}}
\newcommand{\vHnear}{v_{\rm High, near}}
\newcommand{\vHfar}{v_{\rm High, far}}
\newcommand{\vL}{v_{\rm Low}}
\newcommand{\vP}{v_{{\rm PML}}}
\newcommand{\vLnear}{v_{\rm Low, near}}%\vL^{R_0}
\newcommand{\vLfar}{v_{\rm Low, far}}%\vLfar
\newcommand{\vAnear}{v_{{\cA}, {\rm near}}}
\newcommand{\vAfar}{v_{{\cA}, {\rm far}}}

\newcommand{\subsetH}{{\frak H}}
\newcommand{\Optorus}{\operatorname{Op}_\hsc^{\!\!\mathbb T^d_{R_{\sharp}}}}
\newcommand{\tr}{{\rm tr}}
\newcommand{\Helmsol}{u}
\newcommand{\PMLsol}{v}
\newcommand{\asol}{\mathcal S^*_k}

\newcommand{\Amin}{A_{\min}}
\newcommand{\Amax}{A_{\max}}
\newcommand{\cmin}{c_{\min}}
\newcommand{\cmax}{c_{\max}}

\newcommand{\EMZ}{\tau}

\newcommand{\scaled}{P_{\hsc,\theta}}

\newcommand{\refop}{P^\sharp_\hsc}

\newcommand{\Otr}{\Omega_\tr}

\newcommand{\tDelta}{\widetilde{\Delta}}

\newcommand{ \residualP}{O(\hsc^\infty)_{\Psi^{-\infty}_\hsc}}

\newcommand{\e}{\epsilon}
\newcommand{\vres}{v_{\rm residual}}
\makeatletter
\newcommand{\settheoremtag}[1]{% \settheoremtag{<tag>}
  \let\oldthetheorem\thetheorem% Store \thetheorem
  \renewcommand{\thetheorem}{#1}% Redefine it to a fixed value
  \g@addto@macro\endtheorem{% At \end{theorem}, ...
    \addtocounter{theorem}{-1}% ...restore theorem counter value and...
    \global\let\thetheorem\oldthetheorem}% ...restore \thetheorem
  }
\makeatother

\usetikzlibrary{positioning}
\usepackage{lscape}

\definecolor{jeffColor}{RGB}{102, 0, 204}

\begin{document}

\pagenumbering{arabic}
\title{%Wavenumber-explicit convergence of t
The $hp$-FEM applied to the Helmholtz equation with PML truncation does not suffer from the pollution effect}

\author{J.~Galkowski\footnotemark[1],\,\, D.~Lafontaine\footnotemark[2]\,\,, E.~A.~Spence\footnotemark[3]\,\,, J.~Wunsch\footnotemark[4]}

\date{\today}

\footnotetext[1]{Department of Mathematics, University College London, 25 Gordon Street, London, WC1H 0AY, UK,   \tt J.Galkowski@ucl.ac.uk}
\footnotetext[2]{CNRS and Institut de Math\'ematiques de Toulouse, UMR5219; Universit\'e de Toulouse, CNRS; UPS, F-31062 Toulouse Cedex 9, France; \tt david.lafontaine@math.univ-toulouse.fr}
\footnotetext[3]{Department of Mathematical Sciences, University of Bath, Bath, BA2 7AY, UK, \tt E.A.Spence@bath.ac.uk }
\footnotetext[4]{Department of Mathematics, Northwestern University, 2033 Sheridan Road, Evanston IL 60208-2730, US, \tt jwunsch@math.northwestern.edu}

\maketitle

%\begin{AMS}
%35J05, 65N15, 65N30, 78A45
%\end{AMS}

\begin{abstract}
We consider approximation of the variable-coefficient Helmholtz equation in the exterior of a Dirichlet obstacle using perfectly-matched-layer (PML) truncation; it is well known that this approximation is exponentially accurate in the PML width and the scaling angle, and the approximation was recently proved to be exponentially accurate in the wavenumber $k$ in \cite{GLS2}.

We show that the $hp$-FEM applied to this problem
 does not suffer from the pollution effect, in that there exist $C_1,C_2>0$ such that if $hk/p\leq C_1$ and $p \geq C_2 \log k$ then the Galerkin solutions are quasioptimal (with constant independent of $k$), under the following two conditions 
 (i) the solution operator of the original Helmholtz problem is polynomially bounded in $k$ (which occurs for ``most'' $k$ by \cite{LSW1}), and (ii) \emph{either} there is no obstacle and the coefficients are smooth \emph{or} the obstacle is analytic and the coefficients are analytic in a neighbourhood of the obstacle and smooth elsewhere.

This $hp$-FEM result is obtained via a decomposition of the PML solution into ``high-'' and ``low-frequency" components, analogous to the decomposition for the original Helmholtz solution recently proved in \cite{LSW4}. The decomposition is obtained using tools from semiclassical analysis (i.e., the PDE techniques specifically designed for studying Helmholtz problems with large $k$).

\paragraph{Keywords:}
Helmholtz equation, high frequency, perfectly-matched layer, pollution effect, finite element method, error estimate, semiclassical analysis.

\paragraph{AMS subject classifications:}
35J05, 65N12, 65N15, 65N30

\end{abstract}

\section{Introduction and statement of the main results}\label{sec:intro}

\subsection{Recap of the Helmholtz exterior Dirichlet problem and $k$-dependence of its solution operator}\label{sec:intro1}

This paper is primarily concerned with computing solutions of the Helmholtz exterior Dirichlet problem when the wavenumber $k$ is large.

\begin{definition}[Helmholtz Exterior Dirichlet problem]\label{def:EDP}
Let $\Omega_- \subset B_{R_0}:= \{ x : |x| < R_0\}\subset \Rea^d$, $d=2,3$, be a bounded open set with $C^\infty$ boundary $\Gamma_D$ such that $\Omega_+:=\Rea^d\setminus \overline{\Omega_-}$ is connected. 
Let  $A_{\rm scat} \in C^{\infty} (\obstacle_+ , \Rea^{d\times d})$ be symmetric positive definite, let $c_{\rm scat} \in C^\infty(\Omega_+;\Rea)$ be strictly positive and bounded, and let $A_{\rm scat}$ and $c_{\rm scat}$ be such that 
 there exists $R_{\rm scat}>R_0>0$ such that 
\beqs
\overline{\Omega_-} \cup {\rm supp}(I- A_{\rm scat}) \cup {\rm supp}(1-c_{\rm scat}) \Subset B_{R_{\rm scat}}.
\eeqs
Given $g \in L^2(\obstacle_+)$ with $\supp \,g \Subset \Rea^d$ and $k>0$, $u \in H^1_{\rm loc}(\obstacle_+)$ satisfies the exterior Dirichlet problem if 
\begin{align}\label{eq:HelmholtzD}
c_{\rm scat}^2\nabla \cdot (A_{\rm scat}\nabla u) + k^2u= -g\,\quad\text{ in } \obstacle_+,\qquad
 u = 0 \quad\text{ on }\Gamma_D,
\end{align}
and $u$ is \emph{outgoing} in the sense that $u$ satisfies the Sommerfeld radiation condition
\beq\label{eq:src}
\partial_r u
(\bx) - \ri k u(\bx) = o \big(r^{-(d-1)/2}\big)
\quad \tas r:= |\bx|\tendi, \text{ uniformly in } \widehat x:= \bx/r.
\eeq
\end{definition}

Although the exterior Dirichlet problem makes sense for non-smooth domains and coefficients, our results below require (at least) the smoothness in Definition \ref{def:EDP}, and so we assume this smoothness from the start for simplicity.
Let $\|\cdot\|_{H^1_k}$ be the standard weighted $H^1$ norm 
\beq
\label{eq:1knorm}
 \|w\|_{H^1_k}^2:= \|\nabla w \|^2_{L^2}+k^2 \| w \|^2_{L^2}.
\eeq

\begin{definition}[Polynomial-boundedness of the solution operator]\label{def:poly_bound}
Given $k_0>0$,  $K\subset [k_0,\infty)$, the solution operator of the Helmholtz exterior Dirichlet problem is \emph{polynomially bounded for $k\in K$} if there exists $M\geq0$ such that given $R>0$ there exists $C>0$ such that 
 given $g\in L^2(\Omega_+)$ with $\supp \, g\subset B_R$, the solution $u$ of the Helmholtz exterior Dirichlet problem satisfies
\beq\label{eq:poly_bound}
\N{u}_{H^1_k(B_R\cap \Omega_+)}
\leq C k^M \|g\|_{L^2(B_R\cap \Omega_+)}\quad\tfa k\in K.
\eeq
\end{definition}

There exist $C^\infty$ coefficients $A_{\rm scat}$ and $c_{\rm scat}$ and obstacles $\Omega_-$ such that
the solution operator is \emph{not} polynomially bounded for all $k$. E.g., \cite{Ra:71} gives an example of a $c_{\rm scat} \in C^\infty$ such that 
the solution operator with this $c_{\rm scat}$ and $A_{\rm scat} \equiv I$ grows exponentially through a sequence 
$0<k_1<k_2<\ldots$ with $k_j\tendi$ as $j\tendi$. Note that this exponential growth is the worst-possible growth of the solution operator by \cite[Theorem 2]{Bu:98}.

\begin{theorem}\mythmname{Conditions under which the solution operator is polynomially bounded}\label{thm:polyboundD}
Suppose $\Omega_-, A_{\rm scat}$, and $c_{\rm scat}$ are as in Definition \ref{def:EDP}.

(i) If $\Omega_-$, $A_{\rm scat}$, and $c_{\rm scat}$ are additionally nontrapping (i.e.~all the trajectories of the generalised bicharacteristic flow defined by the semiclassical principal symbol of \eqref{eq:HelmholtzD} starting in $B_R$ leave $B_R$ after a uniform time), then 
given $k_0>0$,
\eqref{eq:poly_bound} holds with $M=0$ and $K=[k_0,\infty)$.

(ii) Given $k_0,\delta,\eps>0$ there exists a set $J\subset [k_0,\infty)$ with $|J|\leq \delta$ such that  
\eqref{eq:poly_bound} holds with $M=5d/2+\eps$ and $K= [k_0,\infty)\setminus J$.
\end{theorem}
\bpf[References for the proof]
(i) follows from \emph{either} the results of \cite{MeSj:82} combined with either \cite[Theorem 3]{Va:75} or 
\cite{LaxPhi}, \emph{or} \cite[Theorem 1.3 and \S3]{Bu:02}.
(ii) is proved for $c=1$ in \cite[Theorem 1.1 and Corollary 3.6]{LSW1} and the proof for more-general $c$ follows from combining the results of \cite{LSW1} with \cite[Lemma 2.3]{LSW4};
we highlight that, under an additional assumption about the location of resonances, a similar result with a larger $M$ can also be extracted from \cite[Proposition 3]{St:01} by using the Markov inequality.
\epf

\subsection{Truncation of the exterior domain $\Omega_+$ using the exact Dirichlet-to-Neumann map and solution via the $hp$-FEM}

A popular way of solving boundary value problems involving variable-coefficient PDEs, such as the Helmholtz exterior Dirichlet problem of Definition \ref{def:EDP}, is the finite-element method (FEM).
When the FEM is used with standard piecewise-polynomial subspaces (i.e., piecewise polynomials of degree $p$ on a mesh with meshwidth $h$), the exterior domain $\Omega_+$ must be truncated before the FEM can be used.

One truncation option is to introduce $R> R_{\rm scat}$ such that $\supp \, g \subset B_R$, and then replace $\Omega_+$ by $\Omega_+ \cap B_R$, using as a boundary condition on $\partial B_R$ the exact Dirichlet-to-Neumann (DtN) map for the Helmholtz equation $\Delta u+k^2 u=0$ in the exterior of $B_R$ with the radiation condition \eqref{eq:src} (with this map given explicitly, by separation of variables, in terms of Fourier series and Hankel functions). The solution of this truncated problem is then the restriction of the solution of the exterior Dirichlet problem to $B_R$.

For the exterior Dirichlet problem with exact-DtN-map truncation, there has been a relatively large amount of analysis of the associated FEMs since the initial work of 
\cite{Ma:87, KeGi:89}.
In particular, for the $hp$-version of the FEM, where accuracy is increased by \emph{both} decreasing $h$ \emph{and} increasing $p$, the results of \cite[Theorem 5.8]{MeSa:11} (when $\Gamma_D$ is analytic, $A=I$, and $c=1$) and \cite[Theorem B1]{LSW4} (when $\Gamma_D$ is analytic and $A, c$ are analytic near $\Gamma_D$ -- see 
Assumption \ref{ass:LSW4} below) show that
if the solution operator is polynomially bounded in $k$ as $k\to \infty$ (in the sense of Definition \ref{def:poly_bound})
then there exist $C_1, C_2,$ and $\Cqo$ (independent of $k,h$, and $p$) such that if 
\beqs
\frac{hk}{p}\leq C_1 \quad\tand\quad p\geq C_2 \log k
\eeqs
then the Galerkin solution $u_N$ exists, is unique, and satisfies
\beqs
\N{u-u_N}_{H^1_k(B_R\cap \Omega_+)}\leq \Cqo \min_{v_N\in V_N} \N{u-v_N}_{H^1_k(B_R\cap \Omega_+)},
\eeqs
where $V_N$ is the $hp$ approximation space.

Since the total number of degrees of freedom of the approximation space is proportional to $(p/h)^d$, 
these results show there is a choice of $h$ and $p$ such that the Galerkin solution is quasioptimal, 
with quasioptimality constant (i.e.~$\Cqo$) independent of $k$, and
with the total number of degrees of freedom proportional to $k^d$.
The significance of this result is that it is well-known that the $h$-FEM (where accuracy is increased by decreasing $h$ with $p$ fixed) is \emph{not} quasioptimal with $\Cqo$ independent of $k$ when the total number of degrees of freedom $\sim k^d$ (i.e., when $h\sim k^{-1}$); see \cite{BaSa:00}. This feature is known as ``the pollution effect'' (with the term coined in \cite{IhBa:95a}), and the results of 
\cite{MeSa:11, LSW4} quoted above  therefore show that the $hp$-FEM applied to the exterior Dirichlet problem with exact-DtN-map truncation does not suffer from it.

\subsection{Truncation of $\Omega_+$ using a PML}\label{sec:PML_intro}

Although the solution of the problem truncated with the exact DtN map is the restriction of the solution of the true problem to $\Omega_+\cap B_R$, the exact DtN map is a non-local operator, and hence expensive to compute. 
A popular way of truncating in a less-computationally-expensive way is to use a \emph{perfectly-matched layer (PML)}, introduced by \cite{Be:94} (in cartesian coordinates) and \cite{CoMo:98a} (in spherical coordinates). 
In this paper we consider the following radial PMLs.

\paragraph{Radial PML definition.}
Let $R_{\tr}>R_1>R_{\rm scat}$ and let $\Omega_{\tr}\subset \mathbb{R}^d$ be a bounded Lipschitz open set with $B_{R_{\tr}}\subset \Omega_{\tr}$.
Let $\Omega:=\Omega_{\tr}\cap \Omega_+$, $\Gamma_{\tr}:=\partial\Omega_{\tr}$, and $0\leq \theta<\pi/2$. 
Let 
\beqs
P:= - c_{\rm scat}^2 \nabla\cdot (A_{\rm scat}\nabla).
\eeqs
so that the Helmholtz equation in \eqref{eq:HelmholtzD} is $(P-k^2)u=g$.
The PML method replaces~\eqref{eq:HelmholtzD}-\eqref{eq:src} by 
\beq\label{eq:PML}
(P_\theta -k^2) v=g \quad\tin \Omega, \quad v= 0 \ton \Gamma_D, \quad\tand\quad v=0 \ton \Gamma_{\tr},
\eeq
where
\beq\label{eq:P_theta_informal}
P_\theta := 
\begin{cases}
P, & r\leq R_1,\\
-\Delta_\theta, & r >R_1,
\end{cases}
\eeq
where $-\Delta_\theta$ is a second order differential operator that is given in spherical coordinates $(r,\omega)\in [0,\infty)\times S^{d-1}$  by
\begin{align}
\label{e:deltaTheta}
\Delta_\theta&= \Big(\frac{1}{1+\ri f_\theta'(r)}\pdiff{}{r}\Big)^2+\frac{d-1}{(r+\ri f_\theta(r))(1+\ri f'_\theta(r))}\pdiff{}{r} +\frac{1}{(r+\ri f_\theta(r))^2}\Delta_\omega,\\
&= \frac{1}{(1+\ri f_\theta'(r))(r+\ri f_\theta(r))^{d-1}}\frac{\partial}{\partial r}
\left( \frac{ (r+\ri f_\theta(r))^{d-1}}{1+\ri f_\theta'(r)}\frac{\partial}{\partial r}
\right)+\frac{1}{(r+\ri f_\theta(r))^2}\Delta_\omega,
\nonumber
\end{align}
with $\Delta_\omega$ the surface Laplacian on $S^{d-1}$ and $f_\theta(r)\in C^{\infty}([0,\infty);\mathbb{R})$  given by $f_\theta(r):=f(r)\tan\theta$ for some $f$ satisfying
\begin{equation}
\label{e:fProp}
\begin{gathered}
\big\{f(r)=0\big\}=\big\{f'(r)=0\big\}=\big\{r\leq R_1\big\},\qquad f'(r)\geq 0,\qquad f(r)\equiv r \text{ on }r\geq R_2;
\end{gathered}
\end{equation}
i.e., the scaling ``turns on'' at $r=R_1$, and is linear when $r\geq R_2$. We emphasize that $R_{\tr}$ can be $<R_2$, i.e., we allow truncation before linear scaling is reached. Indeed,
$R_2>R_1$ can be arbitrarily large and therefore, given any bounded interval $[0,R]$ and any function $\widetilde{f}\in C^{\infty}([0,R])$ satisfying 
\beqs
\begin{gathered}
\big\{\widetilde{f}(r)=0\big\}=\big\{\widetilde{f}'(r)=0\big\}=\big\{r\leq R_1\big\},\qquad \widetilde{f}'(r)\geq 0,
\end{gathered}
\eeqs
our results hold for an $f$ with $f|_{[0,R]}=\widetilde{f}$. 

\begin{remark}[Link with other notation used in the literature]\label{rem:notation}In \eqref{eq:PML}-\eqref{e:fProp} the PML problem is written using notation from the method of complex scaling (see, e.g., \cite[\S4.5]{DyZw:19}).
In the numerical-analysis literature on PML, the scaled variable is often written as $r(1+ i \widetilde{\sigma}(r))$ with $\widetilde{\sigma}(r)= \sigma_0$ for $r$ sufficiently large, see, e.g., \cite[\S4]{HoScZs:03}, \cite[\S2]{BrPa:07}. To convert from our notation, set $\widetilde{\sigma}(r)= f_\theta(r)/r$ and $\sigma_0= \tan\theta$.
\end{remark}

\begin{remark}[Smoothness of the PML scaling function $f_\theta$]
We assume that $f_\theta\in C^\infty$ because we  need the differential operator $-\Delta_\theta$ 
to be a semiclassical pseudodifferential operator (with the definition of these recapped in \S\ref{app:sct}).
More precisely, we need the operator $\widetilde{Q}_{\hsc,\theta}$, defined by \eqref{eq:widetildeQ} in terms of $-\Delta_\theta$, to be a semiclassical pseudodifferential operator.
While we could work with pseudodifferential operators with non-smooth symbols, and thus cover $f_\theta$ with lower regularity, this would be more technical.
\end{remark}

\paragraph{Accuracy of PML truncation.}

It is well-known that, for fixed $k$, the error $\|u-v\|_{H^1_k(B_{R_1}\setminus \Omega)}$ decays exponentially in $R_{\rm tr}-R_1$ and $\tan\theta$ -- see 
\cite[Theorem 2.1]{LaSo:98}, \cite[Theorem A]{LaSo:01}, \cite[Theorem 5.8]{HoScZs:03} (with analogous results for cartesian PMLs in  \cite[Theorem 5.5]{KiPa:10}, \cite[Theorem 5.8]{BrPa:13}). 

It was recently proved in \cite{GLS2} that the error $\|u-v\|_{H^1_k(B_{R_1}\setminus \Omega)}$ also decreases exponentially in $k$; indeed, 
the following theorem is a simplified version of \cite[Theorems 1.2 and 1.5]{GLS2}.

\begin{theorem}[Radial PMLs are exponentially accurate for $k$ large]\label{thm:GLS}
Suppose that $f_\theta \in C^3(0,\infty)$ and the solution operator of exterior Dirichlet problem is polynomially bounded in $k$ (in the sense of Definition \ref{def:poly_bound}). Given $\epsilon>0$, there exist $C_1, C_2, k_0>0$ such that for all $\theta\geq \epsilon$, $R_{\rm tr}\geq R_1(1+\epsilon)$, and $k\geq k_0$ the following is true.

Given $g\in L^2(\Omega_+)$ with $\supp\,g\subset B_{R_1}$, the solution $\PMLsol$ to \eqref{eq:PML} exists, is unique, and satisfies 
\beq\label{eq:PML_error}
\|\Helmsol-\PMLsol\|_{H^1_k(B_{R_1}\setminus\Omega_-)} \leq C_1 \exp \Big( - C_2 k\big(R_{\rm tr}-R_1(1+\epsilon)\big)\tan\theta\Big) \|g\|_{L^2(\Omega_+)},
\eeq
where $\Helmsol$ is the solution to the exterior Dirichlet problem of Definition \ref{def:EDP}.
\end{theorem}
We make four remarks regarding Theorem \ref{thm:GLS}.
\bit
\item The order of the quantifiers in Theorem \ref{thm:GLS} (and also later results in the paper) dictates what the constants depend on; e.g., in Theorem \ref{thm:GLS}, $C_1, C_2,$ and $k_0$ depend on $\epsilon$, but are independent of $R_{\tr}, R_1$, and $\theta$. 

\item A similar bound on the error holds even when the solution operator is \emph{not} polynomially bounded and grows exponentially in $k$; see \cite[Theorems 1.2 and 1.5]{GLS2}.

\item Results showing exponential decay in $k$ (similar to in \eqref{eq:PML_error}) for the model problem of $A_{\rm scat}\equiv I, c_{\rm scat}\equiv 1$, and $\Omega_-=\emptyset$ (i.e., no scatterer) were given in  \cite[Lemma 3.4]{ChXi:13} for $d=2$ and \cite[Theorem 3.7]{LiWu:19} for $d=2,3$, using 
 the fact that the solution of this problem can be written explicitly via the fundamental solution or separation of variables.

\item The exponential decay of the error \eqref{eq:PML_error} in $k$ is in contrast to truncation with local absorbing boundary conditions (introduced in 
\cite{Li:75, EnMa:77, 
EnMa:79, BaTu:80, BaGuTu:82}) which give $O(1)$ relative errors as $k\to \infty$ when approximating the solutions of scattering problems; see \cite{GLS1}.
\eit

\paragraph{The variational formulation of the PML problem.}

Given $f_\theta(r)$, let 
\beq\label{eq:alpha_beta}
\alpha(r) := 1 + \ri f_\theta'(r) \quad \tand\quad \beta(r) := 1 + \ri f_\theta(r)/r.
\eeq
Let 
\beq\label{eq:Ac}
A := 
\begin{cases}
A_{\rm scat} & \tfor r < R_1 \\
HDH^T &\tfor r\geq R_1,
\end{cases}
\quad\tand\quad
\frac{1}{c^2} := 
\begin{cases}
c_{\rm scat}^{-2} & \tfor r < R_1 \\
\alpha(r) \beta(r)^{d-1} &\tfor r\geq R_1,
\end{cases}
\eeq
where, in polar coordinates $(r,\varphi)$,
\beqs
D =
\left(
\begin{array}{cc}
\beta(r)\alpha(r)^{-1} &0 \\
0 & \alpha(r) \beta(r)^{-1}
\end{array}
\right) 
\quad\tand\quad
H =
\left(
\begin{array}{cc}
\cos \varphi & - \sin\varphi \\
\sin \varphi & \cos\varphi
\end{array}
\right) 
\tfor d=2,
\eeqs
and, in spherical polar coordinates $(r,\varphi,\phi)$,
\beqs
D =
\left(
\begin{array}{ccc}
\beta(r)^2\alpha(r)^{-1} &0 &0\\
0 & \alpha(r) &0 \\
0 & 0 &\alpha(r)
\end{array}
\right) 
\,\,\tand\,\,
H =
\left(
\begin{array}{ccc}
\sin \varphi \cos\phi & \cos \varphi \cos\phi & - \sin \phi \\
\sin \varphi \sin\phi & \cos \varphi \sin\phi & \cos \phi \\
\cos \varphi & - \sin \varphi & 0 
\end{array}
\right) 
\tfor d=3.
\eeqs
(since $A_{\rm scat}=I$ and $c_{\rm scat}^{-2}=1$ when $r=R_1$, $A$ and $c^{-2}$ are continuous at $r=R_1$).

\ble[Variational formulation of the PML problem \eqref{eq:PML}]
Given $g\in L^2(\Omega_+)$ with $\supp \, g\subset B_{R_1}$, the variational formulation of the PML problem \eqref{eq:PML} is
\beq\label{eq:PML_vf}
\text{ find } v \in H^1_0(\Omega) \,\tst\, a(v,w) = G(w) \,\tfa w \in H^1_0(\Omega),
\eeq
where 
\beqs
a(v,w) := \int_\Omega A \nabla v \cdot\overline{\nabla w} - \frac{k^2}{c^2} v \overline{w}
\quad\tand\quad
G(w):= \int_{B_{R_1}} \frac{ g}{c^2} \overline{w}.
\eeqs
\ele

\bpf
With $\alpha$ and $\beta$ defined by \eqref{eq:alpha_beta} (with this notation used by \cite{LaSo:98, LiWu:19}), $\Delta_\theta$ defined by \eqref{e:deltaTheta} becomes
\beqs
\Delta_\theta = \frac{1}{\alpha(r\beta)^{d-1}} \pdiff{}{r} \left(\frac{(\beta r)^{d-1}}{\alpha}\pdiff{}{r} \right) + \frac{1}{(r\beta)^2} \Delta_\omega.
\eeqs
Multiplying the PDE in \eqref{eq:PML} by $c_{\rm scat}^{-2} \alpha \beta^{d-1}$, using that $c_{\rm scat}\equiv 1$ for $r\geq R_1$, $\alpha\equiv \beta\equiv 1$ for $r\leq R_1$, and $\supp\, g\subset B_{R_1}$, and then changing variables to cartesian coordinates, we find that
$\nabla\cdot(A\nabla u) + (k^2/c^2) u = -g/c^2;$
the variational formulation \eqref{eq:PML_vf} follows.
\epf

\begin{remark}[Plane-wave scattering]\label{rem:pw}
The exterior Dirichlet problem of Definition \ref{def:EDP} considers the Helmholtz equation with right-hand side $g$. Another important Helmholtz problem is that of plane-wave scattering; that is, with $\Omega_-, A_{\rm scat}$, and $c_{\rm scat}$ as above, 
given $a \in \Rea^d$ with $|a|=1$, 
let $\Helmsol^I(x):= \exp (\ri k x\cdot a)$ and find $\Helmsol\in H^1_{\rm loc}(\Omega_+)$ such that 
\begin{align*}
&c_{\rm scat}^2\nabla\cdot( A_{\rm scat} \nabla \Helmsol ) + k^2 \Helmsol = 0\quad\tin \Omega_+,\qquad\Helmsol = 0  \,\,\ton \Gamma_D,
\end{align*}
and $\Helmsol^S:= \Helmsol - \Helmsol^I$ is outgoing (i.e., satisfies \eqref{eq:src}).
Since $\Helmsol$ itself is not outgoing, it cannot be directly approximated by the solution of a problem with PML truncation.
However, let $\chi \in C^{\infty}_{\rm comp}(\Rea^d;[0,1])$ be such that $\chi \equiv 1$ for $r\leq R_{\rm scat}$ and $\chi \equiv 0$ for $r\geq R_1$, and let 
\beqs
\widetilde{\Helmsol}: =\chi \Helmsol^I  + \Helmsol^S = \Helmsol- (1-\chi)\Helmsol^I \quad\tand\quad g:= 2 \nabla \chi\cdot \nabla \Helmsol^I + \Helmsol^I \Delta \chi.
\eeqs
Observe that $\widetilde{\Helmsol}$ then satisfies the PDE in \eqref{eq:HelmholtzD}, and that the right-hand side $g$ is 
supported in $R_{\rm scat}\leq r\leq R_1$. Therefore  PML truncation can be used to approximate $\widetilde{\Helmsol}$. Observe further that 
 $\widetilde{\Helmsol}\equiv \Helmsol$ for $r\leq R_{\rm scat}$, with this usually the region where one is interested in finding the solution $\Helmsol$.
\end{remark}

\begin{assumption}\label{ass:PML}
When $d=3$, $f_\theta(r)/r$ is nondecreasing.
\end{assumption}

Assumption \ref{ass:PML} is standard in the literature (in the notation described in Remark \ref{rem:notation} it is that $\widetilde{\sigma}$ is non-decreasing; see, e.g., \cite[\S2]{BrPa:07}) and ensures that the matrix $A$ \eqref{eq:Ac} satisfies $\Re A>0$ (in the sense of quadratic forms) for all $\theta$; see Lemma \ref{lem:strong_elliptic} and Remark \ref{rem:strong_elliptic} below.

\subsection{The main result:~accuracy of the $hp$-FEM applied to the Helmholtz exterior Dirichlet problem with PML truncation}

\paragraph{Existing results on the accuracy of the FEM applied to Helmholtz problems with PML truncation.} Although the FEM with PML truncation is widely used to compute solutions of the Helmholtz exterior Dirichlet problem (and other boundary value problems involving the Helmholtz or Maxwell equations), until now there have been no rigorous $k$-explicit results guaranteeing the accuracy of the computed solutions of the Helmholtz exterior Dirichlet problem with PML truncation as described in \S\ref{sec:PML_intro}.

Indeed, the only existing $k$-explicit results on the accuracy of the FEM applied to Helmholtz problems with PML truncation are the following.
\bit
\item  The result \cite[Theorem 4.4]{LiWu:19} concerns the model problem of $A_{\rm scat}\equiv I, c_{\rm scat}\equiv 1$, and $\Omega_-=\emptyset$ (i.e., no scatterer), and shows that $\N{v-v_N}_{H^1_k(\Omega)}$ is bounded (independently of $k$) in terms of the data if $hk^{3/2}$ is sufficiently small; this threshold is observed empirically to be sharp when $p=1$ and is the same threshold that appears for the problem with DtN truncation \cite{LSW2} or a first-order absorbing boundary condition \cite{Wu:14}. \footnote{Since the preprint of the present paper appeared, the preprint \cite{GS3} generalised the result of \cite[Theorem 4.4]{LiWu:19} to general scattering problems with PML truncation and $h$-FEM spaces of arbitrary polynomial degree, proving quasi-optimality if $(hk)^{p} k^{1+M}$ is sufficiently small, where $M$ is as in \eqref{eq:poly_bound}, and a bound on the relative error if $(hk)^{2p}k^{1+M}$ is sufficiently small.}
\item
The result \cite{ChGaNiTo:22} considers $\Omega_-$ starshaped, $A_{\rm scat}\equiv I$, and $c_{\rm scat}\equiv 1$, and obtains the same thresholds for quasioptimality (for arbitrary fixed $p>0$) as for both the problem with DtN truncation or a first-order absorbing boundary condition \cite{MeSa:11}. However, \cite{ChGaNiTo:22} considers scaling functions of the form $f_\theta(r) = r \widetilde{\sigma}/k$ (with $\widetilde{\sigma}$ independent of $k$), and with such scaling the PML error is not exponentially small in $k$.
\item
The result \cite[Theorem 6.6.7]{Be:21}/\cite[Theorem 5.5]{BeChMe:22} covers the exterior Dirichlet problem with PML truncation, with a Robin boundary condition on $\Gamma_{\tr}$, under the assumptions that (i) the PML scaling angle, $\theta$, is sufficiently small and (ii) the solution operator 
for this problem is polynomially bounded (in the sense of Definition \ref{def:poly_bound});
we discuss the results of \cite{Be:21, BeChMe:22} further in \S\ref{sec:idea2} below.
\eit

\paragraph{Statement of the main result.}
We consider the exterior Dirichlet problem with domain and coefficients satisfying one of the following two assumptions.

\begin{assumption}\label{ass:LSW3}
(i)  $\Omega_-=\emptyset$. 

(ii) $A_{\rm scat}$ and $c_{\rm scat}$ are as in Definition \ref{def:EDP}.

(iii) 
$\Gamma_{\tr}$ is $C^{1,1}$. 
\end{assumption}

\begin{assumption}\label{ass:LSW4}
(i) $\Omega_-, A_{\rm scat},$ and $c_{\rm scat}$ are as in Definition \ref{def:EDP}.  

(ii) $\obstacle_-$ is analytic, and both $A_{\rm scat}$ and $c_{\rm scat}$ are analytic in $B_{R_*}$ for some $R_0<R_*<R_1$.

(iii) 
$\Gamma_{\tr}$ is $C^{1,1}$. 
\end{assumption}

The reasons we consider these classes of domain and coefficients is explained in \S\ref{sec:idea2}/\S\ref{sec:mainbb_discuss} below. We note here that 
the assumption that $\Gamma_{\tr}$ is $C^{1,1}$ ensures 
that the PML solution is in $H^2(\Otr)$.

\begin{theorem}\mythmname{Quasioptimality of $hp$-FEM for the exterior Dirichlet problem with PML truncation}
\label{thm:FEM1}
Suppose that  $\Omega_-, A_{\rm scat}, $ $c_{\rm scat}$, and $\Omega_\tr$ satisfy \emph{either} Assumption \ref{ass:LSW3} \emph{or} Assumption \ref{ass:LSW4}.
Suppose further that  $\Omega_-, A_{\rm scat}, $ $c_{\rm scat}$, and $K\subset[k_0,\infty)$ are such that the solution operator of the exterior Dirichlet problem is polynomially bounded (in the sense of Definition \ref{def:poly_bound}).
Suppose that the PML scaling function $f_\theta\in C^\infty$ and satisfies Assumption \ref{ass:PML}.
Let $(V_N)_{N=0}^\infty$ be the piecewise-polynomial approximation spaces described in \cite[\S5]{MeSa:10}, \cite[\S5.1.1]{MeSa:11} (where, in particular, the triangulations are quasi-uniform and allow curved elements). 

Given $\epsilon>0$, there exist $k_1, C_1,C_2, \Cqo>0$ such that the following is true.
Given $G\in (H^1_k(\Omega))^*$, for all $k\in K\cap[k_1,\infty)$, $\epsilon\leq \theta \leq \pi/2-\epsilon$, and $R_{\tr}\geq R_1(1 +\epsilon)$,
the solution $\PMLsol$ to the PML problem \eqref{eq:PML}/ \eqref{eq:PML_vf} exists and is unique. Furthermore, 
if 
\beq\label{eq:thresholdD}
\frac{hk}{p} \leq C_1 \quad\tand\quad p \geq C_2 \log k, 
\eeq 
then 
the Galerkin solution $v_N$ of the PML problem \eqref{eq:PML_vf}, satisfying 
\beq\label{eq:FEM}
a(v_N, w_N) = G(w_N) \quad\tfa w_N \in V_N,
\eeq
 exists, is unique, and satisfies the quasioptimal error bound 
\beq\label{eq:qo}
\N{v-v_N}_{H^1_k(\Omega)}\leq \Cqo \min_{w_N\in V_N} \N{v-w_N}_{H^1_k(\Omega)}.
\eeq
\end{theorem}

The error on $B_R\cap \Omega_+$ between the true solution $u$ and the Galerkin approximation to the PML solution $v_N$ is then controlled by combining \eqref{eq:qo} with \eqref{eq:PML_error}.

\bre[Non-conforming error]
Theorem \ref{thm:FEM1} assumes that the domain $\Omega$ is triangulated exactly. 
In practical applications, however, exact triangulations are seldom used, and some analysis of the geometric error is therefore necessary. We ignore this issue here (just as in the previous work on the $hp$-FEM in \cite{MeSa:10, MeSa:11, EsMe:12, MePaSa:13, LSW3, LSW4}), but note that, empirically,  at least for the $h$-FEM, 
the geometric error caused by using simplicial elements 
is smaller than the pollution error. 
\ere

\subsection{The idea behind the $hp$-FEM result of Theorem \ref{thm:FEM1}:~decompositions of high-frequency Helmholtz solutions}\label{sec:idea1}

\paragraph{Decomposition of constant-coefficient Helmholtz solutions in \cite{MeSa:10, MeSa:11, EsMe:12}.}

The celebrated papers \cite{MeSa:10, MeSa:11, EsMe:12, MePaSa:13} established a $k$-explicit convergence theory for the $hp$-FEM applied to the constant-coefficient Helmholtz equation 
$\Delta u +k^2 u =-f.$
This theory is based on decomposing solutions of this equation as 
\beq\label{eq:decomp_H}
u = \ulow + \uhigh,
\eeq
where
\bit
\item[(i)] $\ulow$ is analytic, and satisfies bounds with the same
  $k$-dependence as those satisfied by the full Helmholtz solution,
  but with explicit $k$-dependence built into the Cauchy estimates, and 
\item[(ii)] $\uhigh$ has finite regularity (normally $H^2$), and satisfies bounds with improved $k$-dependence compared to those satisfied by the full Helmholtz solution.
\eit
The papers \cite{MeSa:10, MeSa:11, EsMe:12} obtained such a decomposition for a variety of constant-coefficient Helmholtz problems, 
with the idea of the decomposition
that $\ulow$ corresponds to the low-frequency components of the solution $u$ (i.e., components with frequencies $\lesssim k$) 
$\uhigh$ corresponds to the high-frequency components of solution (i.e., components with frequencies $\gtrsim k$)
-- we discuss this ``decomposing-via-frequencies'' idea further in \S\ref{sec:idea2}.

\paragraph{How the decomposition shows that the $hp$-FEM does not suffer from the pollution effect under the conditions \eqref{eq:thresholdD}.}
The classic duality argument (originating from ideas introduced in \cite{Sc:74} and then refined by \cite{Sa:06}) 
gives a condition for the Galerkin solutions to be quasioptimal in terms of how well solutions of the adjoint problem 
are approximated by the finite-element space (see \S\ref{sec:FEM_overview} below and the discussion/references therein).
Note that solutions of the adjoint problem for the Helmholtz equation are just complex-conjugates of Helmholtz solutions (see Lemma \ref{lem:asol_sol} below), so in this argument one only needs to consider approximation of Helmholtz solutions.

When applying the classic duality argument to the Helmholtz equation, approximating the Helmholtz solution directly (without any decomposing) and using the sharp bound (in terms of $k$-dependence) on its $H^2$ norm results in the condition ``$hk^2/p$ sufficiently small'' for quasioptimality; this is the sharp condition when $p=1$ -- see, e.g., \cite[Figure 8]{IhBa:95a}. 

The fact that $\uhigh$ satisfies a bound one power of $k$ better than that satisfied by $u$ means that the analogue of the condition 
``$hk^2/p$ sufficiently small'' with $u$ replaced by $\uhigh$ is the improved ``$hk/p$ sufficiently small''; i.e., the first condition in \eqref{eq:thresholdD}.
Provided that the solution operator is polynomially bounded in the sense of \eqref{eq:poly_bound}, 
the analogue of the condition ``$hk^2/p$ sufficiently small'' with $u$ replaced by $\ulow$ (and using the first $p+1$ derivatives of $\ulow$) is essentially
\beq\label{eq:explain1}
k^{1+M} \left( \frac{hk}{\sigma p}\right)^p
\eeq
sufficiently small 
(with $\sigma$ constant); see \eqref{eq:etabound_LSW4} below.
With $hk/p$ sufficiently small, \eqref{eq:explain1} can be made arbitrarily small if $p/\log k$ is sufficiently large, leading to the second condition in \eqref{eq:thresholdD}; note that the analyticity of $\ulow$ is crucial here, since it allows us to take $p$ arbitrarily large.

\paragraph{The recent paper \cite{LSW4}:~analogous decompositions for very general Helmholtz scattering problems.}

The recent paper \cite{LSW4} (following \cite{LSW3}) showed that similar decompositions can be obtained for very general Helmholtz scattering problems, namely, those fitting into the so-called ``black-box'' framework of Sj\"ostrand--Zworski \cite{SjZw:91}, with this framework including problems where the scattering is caused by variable coefficients, penetrable obstacles, or
impenetrable obstacles.
For these general Helmholtz solutions, $\ulow$ is not necessarily analytic, but the regularity is determined by properties of the scatterer. 
The paper \cite{LSW4} then showed that, if the domain and coefficients satisfy either Assumptions \ref{ass:LSW3} or \ref{ass:LSW4}, then $\ulow$ is analytic (possibly modulo a remainder that is super-algebraically small in $k$), and then
the arguments of \cite{MeSa:10, MeSa:11} can be used to show that
 the $hp$-FEM applied to these Helmholtz problems does not suffer from the pollution effect.

\paragraph{The main contribution of the present paper.}

The main contribution of the present paper is \emph{showing that the decompositions of outgoing Helmholtz solutions obtained in \cite{LSW4} also hold for the corresponding Helmholtz solutions with PML truncation.} Indeed, our main decomposition result for PML solutions, stated informally in the next subsection as Theorem \ref{thm:informal}, and then rigorously in Theorem \ref{thm:mainbb}, is the exact analogue of the corresponding decomposition result in \cite{LSW4} for outgoing Helmholtz solutions. 

The results in \cite{LSW4} that show that $\ulow$ is analytic if the domain and coefficients satisfy either Assumption \ref{ass:LSW3} or \ref{ass:LSW4}, then show the corresponding result for the low-frequency components of the PML solution. Thus, exactly as in \cite{LSW4}, the  arguments of \cite{MeSa:10, MeSa:11} can be used to show that the $hp$-FEM applied to these PML problems does not suffer from the pollution effect, i.e., Theorem \ref{thm:FEM1}.

We emphasise that the proof of Theorem \ref{thm:mainbb} involves several new technical ideas compared to the proof of the analogous result in \cite{LSW4} for outgoing Helmholtz solutions. These differences arise from the fact that in \cite{LSW4} the notion of ``high-frequency'' and ``low-frequency" components of the solution is defined via the functional calculus for self-adjoint operators (see \S\ref{sec:idea2} below)
but the PML operator is not self-adjoint.
To overcome this obstacle, we use (i) the ellipticity of the PML operator in the scaling region and the recent results of \cite{GLS2}, (ii) the fact that the functional calculus is pseudolocal (see Lemma \ref{lem:funcloc1} below), and (iii) the fact that, away from the scatterer and the PML truncation boundary, the functional calculus is pseudodifferential (see Lemma \ref{lem:funcloc2} below).

\paragraph{Recap of $k$-explicit analyticity.}

Before stating informally the main decomposition result for PML solutions (Theorem \ref{thm:informal}), we record the following lemma about 
how the bound an analytic function depending on $k$ satisfies dictates the $k$-dependence of the region of analyticity; we use this lemma below to understand the properties of the $\vlow$s in Theorems \ref{thm:informal}, \ref{thm:LSW3}, 
and \ref{thm:LSW4}.

\ble[$k$-explicit analyticity]\label{lem:analytic}
With $D$ a bounded open set, let $u\in C^\infty(D)$ be a family of functions depending on $k$.

(i) If 
there exist $C, C_u>0$ such that, for all multiindices $\alpha$,
\beqs
\N{\partial^\alpha u}_{L^2(D)}\leq C_u (C k)^{|\alpha|}.
\eeqs
then $u$ is real analytic in $D$ with infinite radius of convergence, i.e., $u$ is entire.

(ii) If 
there exist $C, C_u>0$ such that, for all multiindices $\alpha$,
\beqs
\N{\partial^\alpha u}_{L^2(D)}\leq C_u (Ck)^{|\alpha|} |\alpha|!,
\eeqs
then $u$ is real analytic in $D$ with radius of convergence proportional to $(C k)^{-1}$. 

(iii) If 
there exist $C, C_u>0$ such that, for all multiindices $\alpha$,
\beqs
\N{\partial^\alpha u}_{L^2(D)}\leq C_u C^{|\alpha|}\max \big\{ |\alpha|, k \big\}^{|\alpha|},
\eeqs
then $u$ is real analytic in $D$ with radius of convergence independent of $k$.
\ele

\bpf
In each part, we use the Sobolev embedding theorem to obtain a bound on $\|\partial^\alpha u \|_{L^\infty(D)}$, and then sum the 
remainder in the truncated Taylor series. For this procedure carried out in Part (iii), see, e.g., \cite[Proof of Lemma C.2]{MeSa:10}; the proofs for the other cases are similar.
\epf

\subsection{Informal statement of the main decomposition result for Helmholtz problems with PML truncation}\label{sec:informal}

\begin{theorem}[Informal statement of the main decomposition result]\label{thm:informal}
Let $P$ be a 
formally self-adjoint operator with $P = -\Delta$ outside a sufficiently-large ball (``the black box''). Suppose that  $P -k^2$ is well defined and that
\begin{itemize}
    \item[(H1)] the solution operator associated with $P-k^2$ is polynomially bounded:~there exists $M\geq 0$ so that for any $\chi \in C^\infty_{\rm comp}$ and any compactly-supported $g\in L^2$, the outgoing solution of $(P-k^2)u = g$ satisfies
    $$
    \Vert \chi u \Vert_{L^2} \lesssim  k^{-1+M} \Vert g\Vert_{L^2},
    $$
    \item[(H2)] one has an estimate quantifying the regularity of $P$ 
    inside the black box.
\end{itemize}

Let $P_\theta$ be defined by \eqref{eq:P_theta_informal}, and let $\Omega_{\tr}$ and $\Omega$ be as in \S\ref{sec:intro1}.
Then any solution of $(P_\theta-k^2) v = g$ in $\Omega$ can be written as
$$
v = \vhigh + \vlow + \vres
$$
where 
\begin{itemize}
    \item[(i)] $\vhigh$ satisfies the same boundary conditions as $v$ and the bound
    $$\Vert \vhigh\Vert_{L^2(\Omega)} + k^{-2}\Vert P_\theta\vhigh\Vert_{L^2(\Omega)} \lesssim k^{-2}\Vert g \Vert_{L^2(\Omega)},$$
    \item[(ii)] $\vlow$ is \emph{regular}, with an estimate depending on \emph{both} the regularity of the underlying problem (as measured by (H2)) \emph{and} $M$. In addition, the part of $\vlow$ away from the black box is entire (in the sense of Lemma \ref{lem:analytic}(i)).
    \item[(iii)] $\vres$ is negligible, in the sense that all of its norms are smaller than any algebraic power of $k$.
\end{itemize}
Finally, given $\epsilon>0$, the constants in the bounds on $\vhigh,\vlow$, and $\vres$ are uniform in $\theta$ for $\epsilon\leq \theta\leq \pi/2-\epsilon$.
\end{theorem}
We make the following immediate remarks:
\bit
\item The assumptions in Theorem \ref{thm:informal} (involving the unscaled operator $P$) are exactly the same as in the analogue of Theorem \ref{thm:informal} for outgoing Helmholtz solutions; see \cite[Theorem A$^\prime$]{LSW4}. The conclusions of Theorem \ref{thm:informal} are essentially the same as those \cite[Theorem A$^\prime$]{LSW4}, except with $u$ replaced by $v$, $P$ replaced by $P_\theta$, and the addition of the ``residual'' term $\vres$ (the reason why this residual term appears here, but not in \cite[Theorem A$^\prime$]{LSW4}, is 
to make $\vhigh$ satisfy the zero Dirichlet boundary condition on $\Gamma_\tr$ -- see the 
discussion after Theorem \ref{thm:mainbb}).
\item If $P$ is the Dirichlet Laplacian with both $\Gamma_D$ and $\Gamma_\tr \in C^{1,1}$ 
then $\|P_\theta\vhigh\|_{L^2}$ controls $\|\vhigh\|_{H^2}$ up to $\|\vhigh\|_{L^2}$ by elliptic regularity, and thus the bound in (i) is a bound on $\|\vhigh\|_{H^2}$ -- hence the notation $\vhigh$. (Assumptions \ref{ass:LSW3} and \ref{ass:LSW4} contain these assumptions on $\Gamma_D$ and $\Gamma_\tr$ 
precisely to ensure this $H^2$ regularity of $\vhigh$.)

\item The paper \cite{LSW1} shows that the assumption (H1) holds in the black-box framework for ``most'' frequencies (see Part (i) of Theorem \ref{thm:polyboundD} for a more precise statement of this). Therefore, to apply this result to specific situations,
the key point is to check that an estimate of the type (H2) holds; we discuss this further in \S\ref{sec:mainbb_discuss}. 
\eit

\paragraph{Transferring the results in \cite{LSW4} for particular Helmholtz solutions to the corresponding Helmholtz solutions with PML truncation.}

Since (i) the assumptions of Theorem \ref{thm:informal} (and its precise version Theorem \ref{thm:mainbb}) are exactly the same (by design) as the assumptions of  \cite[Theorem A$^\prime$/Theorem A]{LSW4}, and (ii) these assumptions are checked in \cite{LSW4} for the particular Helmholtz problems we are interested in here,
analogous decompositions to those in \cite{LSW4} for outgoing Helmholtz solutions then immediately hold for the analogous PML problems. 
Indeed, \cite{LSW4} proves the decomposition $u=\ulow + \uhigh$ \eqref{eq:decomp_H} with $\ulow$ analytic under Assumptions \ref{ass:LSW3} and \ref{ass:LSW4}, with (H2) corresponding to, respectively, an explicit estimate on the eigenfunctions of the Laplace operator on the torus and an analytic estimate for solutions of the heat equation. The PML analogues of these results then follow immediately 
and are stated in Theorems \ref{thm:LSW3} and 
\ref{thm:LSW4} in the next section.

We highlight that \cite{LSW4} also decomposes the solution of the Helmholtz transmission problem, and thus an analogous result holds for the corresponding PML problem. This result shows only finite-regularity of $\vlow$ (as opposed to analyticity), and so gives a (sharp) result about quasioptimality of the $h$-FEM, but not the $hp$-FEM. Since we focus on the $hp$-FEM in the present paper, we do not state this decomposition for the transmission problem with PML truncation (but highlight here that it exists).

\subsection{The main decomposition result applied to the Helmholtz exterior Dirichlet problem with PML truncation under Assumptions \ref{ass:LSW3} or \ref{ass:LSW4}}

\begin{theorem}[Decomposition of the PML solution under Assumption \ref{ass:LSW3}]
\label{thm:LSW3}
Suppose that $\Omega_-, A_{\rm scat}, $ $c_{\rm scat}$, and $\Omega_\tr$ satisfy  Assumption \ref{ass:LSW3}.
Suppose further that $A_{\rm scat}, $ $c_{\rm scat}$, and $K\subset[k_0,\infty)$ are such that the solution operator is polynomially bounded (in the sense of Definition \ref{def:poly_bound}).
 
 Given $\epsilon>0$, there exist $C_j, j=1,2,3,$ and $k_1>0$ such that the following is true.
For all $R_{\tr}>R_1+\e$, $B_{R_{\tr}}\subset  \Omega_{\tr}\Subset \mathbb{R}^d$, 
and $\e<\theta<\pi/2-\e$, 
given $g\in L^2(\Omega)$, the solution $v$ of the PML problem \eqref{eq:PML} exists, is unique, and is such that
\beqs
v =  \vhigh+\vlow + \vres,
\eeqs
where $\vlow, \vhigh$, and $\vres$ satisfy the following. 
$\vhigh\in H^2(\Omega)\cap H^1_0(\Omega)$ with 
\begin{equation} \label{eq:decHFLSW3}
\Vert \partial^\alpha \vhigh \Vert_{L^2(\Omega)} \leq C_1 k^{|\alpha|-2}\Vert g \Vert_{L^2(\Omega)} \quad \tfa  k \in K\cap [k_1,\infty) \text{ and for all } |\alpha|\leq 2.
\end{equation}
$\vlow$ satisfies
\begin{equation} \label{eq:decLFLSW3}
\Vert \partial^\alpha \vlow \Vert_{L^2(\Omega) } \leq C_2 (C_3)^{|\alpha|} k ^{ |\alpha| -1+M} \,\Vert g \Vert_{L^2(\Omega)} \quad\tfa  k \in K\cap [k_1,\infty) \text{ and for all }   \alpha
\end{equation}
and is negligible in the scaling region in the sense that for any $N,m>0$ there exists $C_{N,m}>0$ (independent of $\theta$) such 
\beqs
\Vert \vlow \Vert_{\mathcal H^m((B_{R_1(1+\epsilon)})^c)} \leq C_{N,m}k^{-N} \Vert g \Vert_{\mathcal H(\Otr)}   \quad \tfa  k \in K \cap [k_1,\infty).
\eeqs
Finally $\vres$ is negligible in the sense that for any $N,m>0$ there exists $C_{N,m}>0$ (independent of $\theta$) so that
\begin{equation} \label{eq:residual_neg}
\N{\vres}_{H^m(\Omega)}
\leq C_{N,m}k^{-N}  \,\Vert g \Vert_{L^2(\Omega)}    \tfa  k \in K\cap [k_1,\infty).
\end{equation}
\end{theorem}

By Part (i) of Lemma \ref{lem:analytic}, $v_\cA$ in Theorem \ref{thm:LSW3} is entire.

\begin{figure}[h]
\begin{center}
    \includegraphics[scale=0.8]{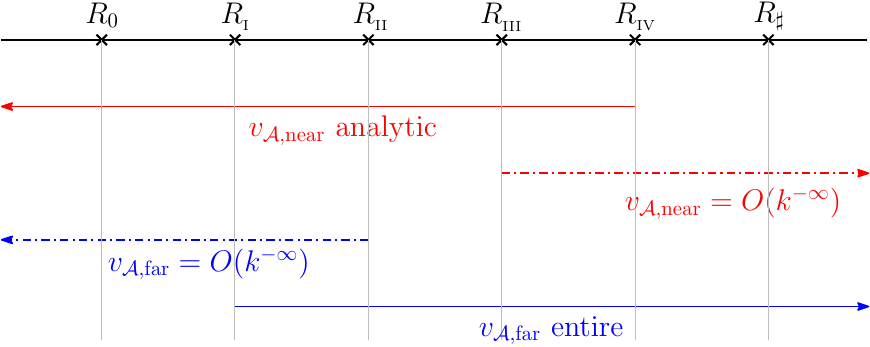}
  \end{center} 
      \caption{The regions where $\vAnear$ and $\vAfar$ appearing in Theorem \ref{thm:LSW4} are analytic, entire, or $O(k^{-\infty})$.}\label{fig:line2}
\end{figure}

\begin{theorem}[Decomposition of the PML solution under 
Assumption 
\ref{ass:LSW4}]
\label{thm:LSW4}
Suppose that $\Omega_-, A_{\rm scat}, $ $c_{\rm scat}$, and $\Omega_\tr$
satisfy 
Assumption \ref{ass:LSW4}.
Suppose further that $\Omega_-,A_{\rm scat}, $ $c_{\rm scat}$, and $K\subset[k_0,\infty)$ are such that the solution operator is polynomially bounded (in the sense of Definition \ref{def:poly_bound}).

Given $\epsilon>0$, there exist 
$C_j>0, j=1,\ldots,5,$ and $R_0<\RfarA<\RfarB<\RlocB<\RlocA<R_1$
such that the following is true.
For all $R_{\tr}>R_1+\e$, $B_{R_{\tr}}\subset  \Omega_{\tr}\Subset \mathbb{R}^d$, 
and $\e<\theta<\pi/2-\e$, 
given $g\in L^2(\Omega)$, the solution $v$ of the PML problem \eqref{eq:PML} exists, is unique, and is such that
\beqs
v = \vhigh + \vlow + \vres
\eeqs
where $\vlow, \vhigh$, and $\vres$ satisfy the following. $\vhigh\in H^2(\Omega)\cap H^1_0(\Omega)$ with 
\begin{equation} \label{eq:decHFDir}
\Vert \partial^\alpha \uhigh \Vert_{L^2(\obstacle)} \leq C_1 k^{|\alpha|-2}\Vert g \Vert_{L^2(\obstacle)} \quad \tfa  k \in K\cap [k_1,\infty) \text{ and for all } |\alpha|\leq 2.
\end{equation}
$\vlow = \vAnear + \vAfar$, 
where $ \vAnear $ 
has zero Dirichlet trace on $\Gamma_D$, $\vAfar$ 
has 
zero Dirichlet trace on $\Gamma_{\tr}$, and, for all $k\in K\cap [k_1,\infty)$ and all $\alpha$,  
\begin{equation} \label{eq:decLFDir1}
\Vert \partial^\alpha \vAnear \Vert_{L^2(B_{\RlocA}\cap \obstacle) } \leq C_2  (C_3)^{|\alpha|}
\max \big\{ |\alpha|^{|\alpha|}, k^{|\alpha|}\big\} k ^{ -1+M} \,\Vert g \Vert_{L^2(\obstacle)},
\end{equation}
\begin{equation*}
\Vert \partial^\alpha \vAfar \Vert_{L^2((B_{\RfarA})^c\cap \obstacle) } \leq C_4 (C_5)^{|\alpha|} k ^{ |\alpha| -1+M} \,\Vert g \Vert_{L^2( \obstacle)}, 
\end{equation*}
and, for any $N,m>0$ there exists $C_{N,m}>0$ (independent of $\theta$) so that
\begin{equation*}
\Vert \vAfar \Vert_{H^{m}(B_{\RfarB}\cap \obstacle) } + \Vert \vAnear \Vert_{H^{m}((B_{\RlocB})^c\cap \obstacle) }  \leq C_{N,m}k^{-N}  \,\Vert g \Vert_{L^2(\Omega)}    \tfa  k \in K\cap [k_1,\infty).
\end{equation*}
Finally $\vres$ is negligible in the sense that for any $N,m>0$ there exists $C_{N,m}>0$ (independent of $\theta$) so that \eqref{eq:residual_neg} holds.
\end{theorem}

By Parts (iii) and (i) of Lemma \ref{lem:analytic}, $\vAnear$ is analytic in $B_{\RlocA}$ with $k$-independent radius of convergence, and $\vAfar$ is entire in $(B_{\RfarA})^c$; see Figure \ref{fig:line2}.

\subsection{The ideas behind the decomposition result of Theorems \ref{thm:informal}, \ref{thm:LSW3}, and \ref{thm:LSW4} and previous decomposition results for Helmholtz problems}
\label{sec:idea2}

Table \ref{tab:summary} summarises the problems considered and approaches to the decompositions in the papers \cite{MeSa:10, MeSa:11, EsMe:12, LSW3, LSW4}, and the present paper. 
We now discuss the six main ideas/ingredients used in the proof of Theorem \ref{thm:informal} (and its precise statement in Theorem \ref{thm:mainbb}).

\begin{landscape}
\global\pdfpageattr\expandafter{\the\pdfpageattr/Rotate 90}
\begin{table}
\begin{tabular}{|c|c|c|c|c|c|c|}
\hline
Paper & Helmholtz equation & Problem & Freq.~cut-offs  & Freq.~cut-offs   & Proof of bound  & Proof of bound \\
&&& defined by & applied to & on HF part & on LF part \\
\hline\hline
\cite{MeSa:10} & $\Delta u + k^2 u=-f$ & in $\Rea^d$ with SRC & Fourier transform on $\Rea^d$ & data & 
asymptotics of & asymptotics of \\
& & &
with sharp cut-off & &  
Bessel/Hankel 
&Bessel/Hankel \\
&&&&&functions &functions\\
\hline
\cite{MeSa:11} & $\Delta u + k^2 u=-f$ &EDP obstacle analytic  & as in \cite{MeSa:10} plus  & data & bounds on  & analytic estimate\\
&&  IIP convex polygon & extension operators && cut-offs from \cite{MeSa:10} & on Helm.~solutions \\
&& \quad or smooth &&& 
& with analytic data\\
\hline 
\cite{EsMe:12} & $\Delta u + k^2 u=-f$ & IIP convex polygon & as in \cite{MeSa:10} plus & data & bounds on & analytic estimate \\
&&&extension operators && cut-offs from \cite{MeSa:10}&on Helm.~solutions\\
&&&&&
&with analytic data\\
\hline
\cite{LSW3} & $\nabla\cdot(A\nabla u ) + k^2 c u=-f$ & in $\Rea^d$ with SRC & Fourier transform on $\Rea^d$ & solution  & semiclassical ellip. & immediate  \\
 && $A, c$ smooth & smooth cut-off & $\times$ spatial cut-off & of Helmholtz on HF &from FT\\
 \hline
 \cite{LSW4} & equations that are   & any problem & functional calculus  & solution  & semiclassical ellip. & abstract  regularity \\
(general & $\Delta u +k^2u=0$ & fitting in framework & (i.e., eigenfunction  & $\times$ spatial cut-off& pseudo.~prop. & estimate in \\
\, result) & outside large ball &of black-box scattering& \quad expansion), && of func.~calc. & black box\\
&& & smooth cut-off &&&\\
\hline
 \cite{LSW4} & $\nabla\cdot(A\nabla u ) + k^2 c u=-f$ & EDP obstacle analytic & functional calculus, & 
 solution  & semiclassical ellip. & heat-flow  \\
(specific & & $A, c$ analytic near obstacle &  smooth cut-off &$\times$ spatial cut-off& pseudo.~prop. & estimate \\
\, result) &&&&& of func.~calc. &\\
\hline 
this & $\nabla\cdot(A\nabla u ) + k^2 c u=-f$ & \emph{either} $A$, $c$ smooth, no obs.  & functional calculus, & 
 solution  & semiclassical ellip. & heat-flow \\
paper  &  + PML truncation & \emph{or} EDP obstacle analytic  &  smooth cut-off &$\times$ spatial cut-off&pseudo.~prop. &estimate \\
&&$A, c$ analytic near obstacle && supported & of func.~calc. &\\
&&&& into PML region &&\\
\hline
\end{tabular}
\caption{Summary of the decomposition results in the papers \cite{MeSa:10, MeSa:11, EsMe:12, LSW3, LSW4} and the present paper.
``SRC'' stands for ``Sommerfeld radiation condition'', ``EDP'' stands for ``exterior Dirichlet problem'', ``IIP'' stands for ``interior impedance problem'', ``HF'' stands for ``high-frequency'', and ``LF'' stands for ``low-frequency''. To keep the notation concise, we abbreviate $A_{\rm scat}$ and $c_{\rm scat}$ by $A$ and $c$, respectively.
}\label{tab:summary}
\end{table}
\end{landscape}
\global\pdfpageattr\expandafter{\the\pdfpageattr/Rotate 0}

\paragraph{Ingredient 1:~semiclassical ellipticity of the Helmholtz operator on high frequencies.}

The reason the high-frequency component $\vhigh$ satisfies a bound with better $k$-dependence than the solution $v$ is because 
\emph{the Helmholtz operator is semiclassically elliptic on frequencies with modulus $>k$}. While this feature was observed in \cite{LSW3} in the variable-coefficient setting, its essence is most easily illustrated in the constant-coefficient setting. 
With the Fourier transform defined by 
\beq\label{eq:FT}
\mathcal F_k\phi(\xi) := \int_{\mathbb R^d} \exp\big( -\ri k x \cdot \xi\big)
\phi(x) \, \rd x
\eeq
(i.e., the standard Fourier transform with the Fourier variable scaled by $k$), the constant-coefficient Helmholtz operator is Fourier multiplier with Fourier symbol $|\xi|^2-1$; i.e.,
\beq\label{eq:FM}
\big((-k^{-2}\Delta -1 )v\big)(x) = \cF_k^{-1}\Big( (|\xi|^2-1)\cF_k v(\xi)\Big) (x).
\eeq
If $\lambda 
>1$ then there exists $C>0$ such that 
\beqs
\big||\xi|^2-1\big|\geq C \langle\xi\rangle^2 \quad \tfor |\xi|\geq  \lambda;
\eeqs
i.e., the Fourier symbol of the constant-coefficient Helmholtz operator is elliptic on $|\xi|>1$, 
with this range of $\xi$ corresponding to the standard Fourier variable (i.e., with no scaling by $k$ in \eqref{eq:FT}) having modulus $>k$. 
The ``high-frequency'' components of the solution are then defined as those with frequency $>k$, and the ``low-frequency'' ones defined as those with frequencies $\lesssim k$.

\paragraph{Ingredient 2:~semiclassical pseudodifferential operators.}

The variable-coefficient Helmholtz operator $\nabla\cdot (A_{\rm scat}\nabla) + k^2 c_{\rm scat}$ is no longer a Fourier multiplier (i.e., it cannot be written in the form \eqref{eq:FM}). It is, however, a pseudodifferential operator; indeed, recall that part of the motivation for the development of pseudodifferential operators was to extend Fourier analysis to study variable-coefficient (as opposed to constant-coefficient) PDEs. 
\emph{Semiclassical} pseudodifferential operators are those defined with Fourier transform defined by \eqref{eq:FT}, i.e., with the large 
parameter $k$ (or small parameter $k^{-1}$) built in; thus semiclassical pseudodifferential operators are precisely the pseudodifferential operators tailor-made to study problems with a large/small parameter.

The paper \cite{LSW3} uses the ``nice'' behaviour of elliptic semiclassical pseudodifferential operators (namely, they are invertible up to a small error) to prove the required bound on the high-frequency components of the decomposition for the (non-truncated) Helmholtz equation in $\Rea^d$ (i.e., $\Omega_-= \emptyset$) with smooth $A_{\rm scat}$ and $c_{\rm scat}$. Note that (i) the polynomial boundedness condition of Definition \ref{def:poly_bound} is needed to show that the $O(k^{-\infty})$ error terms in the pseudodifferential calculus acting on the solution are indeed small (which is not guaranteed if the solution operator grows exponentially in $k$), and (ii) 
the theory of pseudodifferential operators is the least technical when the symbols are smooth, thus \cite{LSW3} used smooth frequency cut-offs (as opposed to those defined by an indicator function in \cite{MeSa:10, MeSa:11}). \footnote{The expository paper \cite{S1} shows that a frequency cut-off defined by an indicator function can nevertheless be used in the constant-coefficient case; this is because Fourier multipliers can be formulated without any differentiability requirements on the symbols. The paper \cite{S1} gives an alternative proof of the decomposition result in \cite{MeSa:10} using just elementary properties of the Fourier transform and integration by parts (in particular, without any of the Bessel/Hankel-function asymptotics used in \cite{MeSa:10}).
}

\paragraph{Ingredient 3:~frequency cut-offs defined as functions of the operator (i.e., eigenfunction expansion).}

For problems posed in domains other than $\Rea^d$, it is difficult to use the Fourier transform to define frequency cut-offs. 
The papers \cite{MeSa:11, EsMe:12} tackle this issue by using the composition of the frequency cut-offs on $\Rea^d$ and a suitable extension operator from the domain to $\Rea^d$. Here, following \cite{LSW4}, we instead define frequency cut-offs using the eigenfunctions of the Helmholtz operator considered on a large torus including $\Otr$ (and the black box inside it); this approach has the advantage that the frequency cut-offs then commute with the Helmholtz operator used to define them.

More precisely, recall that the functional calculus defines functions of a self-adjoint elliptic operator in terms of eigenfunction expansions.
Here we choose the operator to be the so-called \emph{reference operator} in the framework of black-box scattering; this is just the operator $P_\hsc^\sharp:= -k^{-2}c_{\rm scat}^2 \nabla\cdot(A_{\rm scat}\nabla)$ considered on the torus $\mathbb{T}^d_{R_\sharp}$ with $R_{\sharp}$ sufficiently large 
so that the torus contains $\Omega_\tr$ (see \S\ref{subsec:bb} below). Then, with $\lambda_j^\sharp$ and $\phi_j^\sharp$ the eigenvalues and eigenfunctions of $P_\hsc^\sharp$ and $f$ a real-valued Borel function,
\beqs
f(P_\hsc^\sharp)v =\sum_j a_j f(\lambda_j^\sharp) \phi_j^\sharp \quad\tfor \quad v=\sum_j a_j \phi_j^\sharp
\eeqs
(see \S\ref{sec:BBFC} below). Given 
$\psi \in C^\infty_{\rm comp}(\mathbb R;[0,1])$ with $\supp \,\psi \subset [-2,2]$ and $\psi\equiv 1$ on $[-1,1]$, we define
$\psi_\mu := \psi(\cdot/\mu)$ and let 
\beqs
\Pilow:= \psi_\mu (P_\hsc^\sharp) \quad\tand \quad\Pihigh:= (1-\psi_\mu) (P_\hsc^\sharp)= I-\Pilow;
\eeqs
see \eqref{eq:PiL} and \eqref{eq:Pihigh} below.
As mentioned above, a crucial fact about these frequency cut-offs is that they commute with $\refop$.

\paragraph{Ingredient 4:~introduce a spatial cut-off and use
ellipticity of the PML operator in
  the scaling region.}

We choose $\varphi_\tr \in C^\infty_{\rm comp}(\Rea^d; [0,1])$ such that $\varphi_\tr \equiv 1$ on $B_{R_1(1+\delta)}$ and $\supp\, \varphi_\tr \subset B_{R_1(1+2\delta)}$ for a suitably chosen $\delta>0$. 
We then decompose $v$ as 
\beq\label{eq:v_split_explain}
 v = \underbrace{\Pihigh(\varphi_\tr v)}_{=: \vH} +  \underbrace{\Pilow(\varphi_\tr v)}_{=: \vL} + \underbrace{(1-\varphi_\tr)v}_{=:\vP}.
\eeq
We then use results from the recent paper \cite{GLS2} to bound $\vP$ in terms of the data with one power better $k$-dependence than the bound on the solution $v$; thus $\vP$ can be included in the component $\vhigh$ (note that the conditions on $\Gamma_\tr$ in Assumptions \ref{ass:LSW3} and \ref{ass:LSW4} ensure that the PML solution is $H^2$ up to the boundary $\Gamma_\tr$).

The ingredients used to bound $\vP$ are (i) the fact that, at highest order, the imaginary part of $-k^{-2}\Delta_\theta-1$ has a sign in the scaling region (see, e.g., \cite[Equation 4.22]{GLS2}, with this behind Lemma \ref{l:nearBoundary} below) 
and (ii) a Carleman estimate describing how $v$ propagates in the scaling region (see Lemma \ref{l:carleman} below).

In bounding $\vP$, it is crucial that 
$(1-\varphi_\tr)$ (and hence also $(1-\varphi_\tr)v$) is supported only in the PML scaling region $(B_{R_1})^c$. However, the fact that $\supp \, \varphi_\tr$ enters the scaling region causes the following issue. 
When bounding $\vH$, we consider 
\begin{align}
(\refop-I)\Pihigh({\varphi}_\tr v) &= \Pihigh (\refop-I)({\varphi}_\tr v) =
 \Pihigh \Big(
[\refop, \varphi_\tr ] v  + \varphi_\tr (\refop-I) v\Big).\label{eq:temp1}
\end{align}
We would now like to say that $(\refop-I)v$ equals the data $(P_{\hsc,\theta} -I)v$, but this is not the case since $\refop \neq P_{\hsc,\theta}$ on $\supp \,\varphi_\tr$ (which enters the scaling region).

The solution is twofold:~we first split $\vH = \Pihigh (\varphi_0 v) + \Pihigh(1-\varphi_0)\varphi_\tr v$ (see \eqref{eq:vH_decomp1} below), where $\varphi_0\in C^\infty_{\rm comp}(\Rea^d; [0,1])$ such that $\varphi_0 \equiv 1$ on $B_{R_0}$ and $\supp\, \varphi_0 \subset B_{R_1}$, and thus $\refop = P_{\hsc,\theta}$ on $\supp\, \varphi_0$. 
We argue as above for $\Pihigh (\varphi_0 v)$ and then deal with the component $\Pihigh(1-\varphi_0)\varphi_\tr v$, as well as the commutator term in \eqref{eq:temp1}, using the next ingredient.

\paragraph{Ingredient 5:~away from the black box, functions of $\refop$ are semiclassical pseudodifferential operators.}

When bounding $\vH$ and $\vL$, we use repeatedly the result that, when $f$ is sufficiently well-behaved and $\chi \in C^\infty(\Rea^d;[0,1])$ is zero in a neighbourhood of the black box, $\chi f(\refop)\chi$ is a pseudodifferential operator (up to a negligible error term); see Lemma \ref{lem:funcloc2} below. In particular, this result allows us to treat $\Pihigh$ and $\Pilow$ as pseudodifferential operators away from the black box.

The context of this result, due to Sj\"ostrand \cite{Sj:97}, 
 is the following:~in the setting of the homogeneous
pseudodifferential calculus, 
Strichartz \cite{St:72} proved  that a well-behaved function 
of a self-adjoint
elliptic differential operator is a pseudodifferential
operator. Helffer--Robert \cite{HeRo:83} proved the corresponding result in the semiclassical setting (see, e.g., the account
\cite[Chapter 8]{DiSj:99}), with this result using the 
Helffer--Sj\"ostrand approach to the functional calculus
\cite{HeSj:89}.
In the setting of black-box scattering, we cannot expect such a result to hold everywhere, because we don't know what's inside the black box. However, thanks to 
Sj\"ostrand \cite{Sj:97} this pseudodifferential property holds when localised away from the black box.

\paragraph{Ingredient 6:~regularity estimates inside the black box.}

While the analysis of $\vH$ is insensitive to the contents of the
black-box (see Ingredient~3)
understanding the properties of the low-frequency piece $\vL$ necessarily involves
``opening'' the black box.
 Intuitively, the fact that the spectral parameter in  $\Pilow( \varphi_\tr v)$ is compactly supported indicates that strong elliptic
estimates should hold, but knowing that $\vL$ is analytic 
is dependent on the coefficients and domain inside the black box.

The abstract result Theorem \ref{thm:mainbb} contains 
the abstract regularity hypothesis \eqref{eq:lowenest}.
The choices of this hypothesis to prove Theorems \ref{thm:LSW3} and \ref{thm:LSW4} are discussed 
in \S\ref{sec:mainbb_discuss} (after the statement of Theorem \ref{thm:mainbb}), but we highlight here that bound \eqref{eq:decLFLSW3} on $\vL$ in Theorem \ref{thm:LSW3} is proved using explicit calculation involving the eigenvalues of $-\Delta$ on the torus, and the bound \eqref{eq:decLFDir1} on $\vL$ in Theorem \ref{thm:LSW4} is proved using heat equation bounds from \cite{EsMoZh:17}. Indeed, for the latter, because of the compact support of the spectral parameter in $\Pilow$, 
we can run the \emph{backward heat equation} on $\Pilow (\varphi_\tr v)$ for
as long as we like and obtain $L^2$ estimates on the result.  If the
boundary and coefficients are analytic then known heat kernel estimates yield the necessary Cauchy-type estimates on $\pa^\alpha \Pilow( \varphi_\tr v)$; see Corollary \ref{cor:heat} and Theorem \ref{thm:heat} below.

\paragraph{Discussion of the recent results \cite{Be:21, BeChMe:22} that extend the approach of \cite{MeSa:10, MeSa:11, EsMe:12} to variable-coefficient problems.}

The recent thesis \cite{Be:21} is an extension of the approach of \cite{MeSa:10, MeSa:11, EsMe:12} to variable-coefficient Helmholtz problems. 
Since the preprint of the present paper appeared, the results of \cite{Be:21} appeared as the preprint \cite{BeChMe:22}. We make the following three remarks comparing and contrasting the approach of \cite{Be:21,BeChMe:22} (following \cite{MeSa:10, MeSa:11, EsMe:12}) and the approach of \cite{LSW3}/\cite{LSW4}/the present paper.
\ben
\item \emph{(Boundary conditions.)} The approach of \cite{Be:21,BeChMe:22}
  in principle covers a variety of boundary conditions. For example, 
    \cite[Theorem 5.5]{BeChMe:22} proves an analogous result to Theorem \ref{thm:FEM1} 
    for the PML problem with an impedance boundary condition on $\Gamma_{\tr}$ under (i) assumptions about the coefficients and domain discussed in Point 2 below, (ii) the assumption that the solution operator of the PML problem is polynomially bounded in $k$, and (iii) the assumption that the PML scaling angle $\theta$ is sufficiently small.
Theorem \ref{t:scaledResolve} below (from \cite{GLS2}) verifies the assumption (ii) for truncation with a Dirichlet boundary condition
(under the assumptions on the scaling function in \S\ref{sec:PML_intro}) and this result also holds for truncation with an impedance boundary condition (indeed, the boundary condition on $\Gamma_{\tr}$ enters the analysis in \cite{GLS2} via \cite[Lemma 4.4]{GLS2}, and this lemma -- relying on integration by parts near $\Gamma_{\tr}$ -- goes through as before provided the impedance parameter has the correct sign).

We note that truncation via the exact DtN map, which is the easiest boundary condition to deal with in the approach of 
\cite{LSW3}/\cite{LSW4}/the present paper, is the most difficult boundary condition to deal with in the approach of \cite{Be:21}. Indeed, the decomposition for the DtN map required in the latter approach is proved using results about boundary integral operators from \cite{Me:12} (see \cite[Lemma 6.5.12 and its proof in \S6.9]{Be:21}).

\item \emph{(Assumptions on the coefficients/domain.)} As in \cite{MeSa:10, MeSa:11, EsMe:12}, the frequency cut-offs in \cite{Be:21, BeChMe:22} are applied to the data; $\vlow$ is then the solution of a Helmholtz problem with (piecewise) analytic data, and one needs (piecewise) analytic coefficients (where the pieces are separated by analytic surfaces) and an analytic domain to get that $\vlow$ is analytic \cite[Lemma 6.5.8]{Be:21}. 
In contrast, the approach in \cite{LSW3}/\cite{LSW4}/the present paper 
 can deal with smooth coefficients (everywhere when $\Omega_-=\emptyset$, and away from the obstacle in the general case) as a result of applying the cut-offs to the solution itself.

 \item \emph{(Bound on the high-frequency part.)}
In \cite{Be:21, BeChMe:22}, the semiclassical ellipticity of the Helmholtz operator on high frequencies -- although not explicitly mentioned -- is again behind the improved bound on $\vhigh$ compared to $v$. Indeed, with $S_k^-$ the solution operator to the Helmholtz equation $(\Delta +k^2)v=-f$ and $S_k^+$ the solution operator to $(\Delta -k^2)v=f$, \cite[Page 98]{Be:21} writes ``we will later see that $S_k^-$ and $S_k^+$ act very similar on high-frequency data'' (with ``later'' referring to \cite[Remark 6.3.7]{Be:21}).
\een

\subsection{Outline of the rest of the paper}

Section \ref{sec:FEM} proves the $hp$-FEM convergence result of Theorem \ref{thm:FEM1} using Theorems \ref{thm:LSW3} and 
\ref{thm:LSW4}, as discussed in \S\ref{sec:idea1}, this follows closely the arguments in \cite{MeSa:10, MeSa:11, LSW3, LSW4} and so, for brevity, quotes several results from these papers without proof.

Section~\ref{sec:blackbox} recalls the framework of black-box scattering, and sets up the associated functional calculus; this section is similar to \cite[\S2]{LSW4} (and refers to that for some of the proofs) except that it now has to deal with both the (unscaled) operator $P$ and the scaled operator $P_\theta$, whereas \cite[\S2]{LSW4} only deals with $P$.

Section \ref{sec:mainresult} states the main decomposition result for Helmholtz solutions in the black-box framework with PML truncation
(Theorem \ref{thm:mainbb}), with this result then proved in Section~\ref{sec:blackboxresult}.

Section \ref{sec:LSW34} shows how Theorems \ref{thm:LSW3} and \ref{thm:LSW4} follow from Theorem \ref{thm:mainbb} -- by design, these proofs are essentially identical to the proofs  in \cite{LSW4} of the analogous results for outgoing Helmholtz solutions; we therefore 
give a sketch of the main steps.

Appendix~\ref{app:sct} recalls results about semiclassical pseudodifferential operators on the torus.

\section{Proof of Theorem \ref{thm:FEM1} using Theorems \ref{thm:LSW3} and 
\ref{thm:LSW4}}\label{sec:FEM}

\subsection{Overview}\label{sec:FEM_overview}

The two ingredients for the proof of Theorem \ref{thm:FEM1} are
\bit
\item Lemma \ref{lem:Schatz}, which is the classic duality argument giving a condition for quasioptimality to hold in terms of how well the solution of the adjoint problem 
is approximated by the finite-element space (measured by the quantity  $\eta(V_N)$ defined by \eqref{eq:etadef}), and  
\item Lemma 
\ref{lem:MS_LSW4} that bounds $\eta(V_N)$ using the decomposition from Theorem \ref{thm:LSW4}.
\eit 
Regarding Lemma \ref{lem:Schatz}:~this argument came out of ideas introduced in \cite{Sc:74}, was formalised in its present form in \cite{Sa:06}, and has been used extensively in the analysis of the Helmholtz FEM; see, e.g., \cite{AzKeSt:88, DoSaShBe:93, Me:95, Ih:98, Sa:06, MeSa:10, MeSa:11, ZhWu:13, Wu:14, DuWu:15, ChNi:18, LiWu:19, ChNi:20, ChGaNiTo:22, GrSa:20, GaSpWu:20, LSW3}.

Regarding Lemma 
\ref{lem:MS_LSW4}:~given the decomposition in Theorem \ref{thm:LSW4}, the bound on $\eta(V_N)$ when Assumption \ref{ass:LSW4} is satisfied is identical to the corresponding proof of \cite[Lemma 5.5]{LSW4} (which is also very similar to the proof of \cite[Theorem 5.5]{MeSa:10}). 

The main work in this section is therefore recalling that the PML variational formulation \eqref{eq:PML_vf} satisfies a G\aa rding inequality and therefore fits in the framework of Lemma \ref{lem:Schatz}.

\subsection{The sesquilinear form $a(\cdot,\cdot)$ is continuous and satisfies a G\aa rding inequality}

In the following lemma $(\cdot,\cdot)_2$ and $\|\cdot\|_2$ denote, respectively, the
Euclidean inner product and associated norm on $\mathbb{C}^d$.

\ble
Given $A_{\rm scat}$ and $c_{\rm scat}$ as in Definition \ref{def:EDP}, a scaling function $f(r)$ satisfying \eqref{e:fProp}, and $\epsilon>0$ there exist $\Amax$ and $\cmax$ such that, for all $\epsilon \leq \theta \leq \pi/2-\epsilon$, $x \in \Omega$, and $\xi,\zeta \in \mathbb{C}^d$,
\beqs
|(A(x)\xi,\zeta)_2|\leq \Amax \|\xi\|_2 \|\zeta\|_2
\quad\tand\quad
\frac{1}{|c(x)|^2} \geq \frac{1}{(\cmax)^2}.
\eeqs
\ele

\bpf
This follows from the definitions of $A$ and $c$ in \eqref{eq:Ac}, the definitions of $\alpha$ and $\beta$ in \eqref{eq:alpha_beta}, and the fact that $f_\theta(r):= f(r)\tan\theta$.
\epf

\begin{corollary}[Continuity of $a(\cdot,\cdot)$]
If
$\Ccont := \max\{ \Amax, \cmin^{-2} \}$,
then
\beqs
|a(v,w)|\leq \Ccont \N{v}_{H^1_k(\Omega)} \N{w}_{H^1_k(\Omega)} \quad\tfa v,w \in H^1_0(\Omega).
\eeqs
\end{corollary}

\bpf This follows by the  Cauchy-Schwarz inequality and the definition of $\|\cdot\|_{H^1_k(\Omega)}$ \eqref{eq:1knorm}.
\epf

\begin{lemma}\label{lem:strong_elliptic}

Suppose that Assumption \ref{ass:PML} holds when $d=3$.
With $A$ defined by \eqref{eq:Ac}, given $\epsilon>0$
there exists $\Amin>0$ such that, 
for all $\epsilon \leq \theta\leq \pi/2-\epsilon$,
\beqs
\Re \big( A(x) \xi, \xi\big)_2 \geq \Amin \|\xi\|_2^2 \quad\tfa \xi \in \mathbb{C}^d \tand x \in \Omega_+.
\eeqs
\end{lemma}

\begin{corollary}[$a(\cdot,\cdot)$ satisfies a G\aa rding inequality]
\beq\label{eq:Garding}
\Re a(w,w) \geq \Amin \|w\|^2_{H^1_k(\Omega)} - \big(\Amin + (\cmin)^{-2}\big)k^2 \|w\|^2_{L^2(\Omega)} \quad\tfa w\in H^1_0(\Omega).
\eeq
\end{corollary}

\bpf[Proof of Lemma \ref{lem:strong_elliptic}]
By assumption, $A_{\rm scat}(x)$ is symmetric positive definite for all $x\in \Omega$ with $r\leq R_1$. We therefore only need to consider the region $r\geq R_1$

Let $\eta:= H^T \xi$; since $H$ is orthogonal, $\|\eta\|_2 = \|\xi\|_2$. Then $\Re(A\xi,\xi)_2 = \Re(D\eta,\eta)_2$. Explicit calculation from the definition of $D$ shows that 
\beq\label{eq_ReD}
\Re D =
\left(
\begin{array}{cc}
\frac{1+r^{-1}f_\theta f_\theta'}{1+ (f_\theta')^2} &0 \\
0 & \frac{1+ r^{-1} f_\theta f_\theta'}{1+r^{-2} f_\theta^2}
\end{array}
\right), 
\,d=2, 
\,\,\tand\,\,
\Re D =
\left(
\begin{array}{ccc}
\frac{1- r^{-2} f_\theta^2 + 2 r^{-1} f_\theta' f_\theta}{1+ f_\theta^2} &0 &0\\
0 & 1 &0 \\
0 & 0 &1
\end{array}
\right),
\,d=3.
\eeq

We now claim that there exists $C>0$ (depending on $\tan \theta$) such that 
\beq\label{eq_Dpos}
\Re (D(x)\eta,\eta)_2 = (\Re D(x) \eta,\eta)_2 \geq C \|\eta\|_2^2 \quad\tfa \eta \in \mathbb{C}^d \tand r\geq R_1;
\eeq
the result then follows since $\tan \theta$ depends continuously on $\theta$ and is bounded above and below (with bounds depending on $\epsilon$) for $\epsilon \leq \theta\leq \pi/2-\epsilon$.

When $d=2$, \eqref{eq_Dpos} follows immediately from \eqref{eq_ReD} and the fact that both $r^{-1} f_\theta$ and $f_\theta'$ are non-negative. When $d=3$, \eqref{eq_Dpos} follows if we can show that 
$r^{-1} f_\theta(r) \leq f_\theta'(r)$ 
for all $r\geq R_1,$
which in turn follows from Assumption \ref{ass:PML} since $f_\theta'(r) =f_\theta(r)/r + r (f_\theta(r)/r)'$.
\epf

\bre[Assumption \ref{ass:PML} and Lemma \ref{lem:strong_elliptic}]\label{rem:strong_elliptic}
Without assumptions on $f_\theta(r)$ additional to \eqref{e:fProp} (such as Assumption \ref{ass:PML}) the eigenvalues of the matrix $D$ will not all lie in a half plane. Indeed, $\alpha$ (defined in \eqref{eq:alpha_beta}) lies in the first quadrant of the complex plane for all $\theta\in [0,\pi/2]$. Explicit calculation shows that 
\beqs
\frac{\beta^2}{\alpha} = 
\big(1 + (f'_\theta)^2\big)^{-1}\left[
1- \left(\frac{f_\theta}{r}\right)^2 + \frac{2f_\theta f'_\theta}{r} + \ri \left( \frac{2f_\theta}{r} + f'_\theta \left( \left(\frac{f_\theta}{r}\right)^2 -1 \right)\right)
\right].
\eeqs

If $f_\theta(r)/r$ is small compared to both $1$ and $f_\theta'(r)$ (which can occur when the scaling ``turns on'' sufficiently quickly at a large $R_1$)
\beqs
\frac{\beta^2}{\alpha} \approx 
\big(1 + (f'_\theta)^2\big)^{-1}\left[
1+\frac{2f_\theta f'_\theta}{r} - \ri  f'_\theta 
\right]
\eeqs
and so is in the fourth quadrant of the complex plane. If, in addition, $f_\theta'(r)$ is large compared to $1$, then $3\pi/2\leq \arg (\beta^2/\alpha)\leq 7\pi/8$. 

If there exists $r^*\in (R_1,R_2)$ such that $f_\theta'(r^*)$ is small 
then 
\beqs
\frac{\beta^2}{\alpha} \approx
1-\left(\frac{f_\theta}{r}\right)^2 + \ri  \frac{2 f_\theta}{r}.
\eeqs
Suppose, furthermore, that $f_\theta(r^*)> r^* \tan\theta$.
Then if $\tan \theta>1$ (i.e., $\theta>\pi/4$), then when $r=r^*$, $\beta^2/\alpha$ lies in the second quadrant of the complex plane. Furthermore, as $\theta\to \pi/2$, the argument of $\beta^2/\alpha$ tends to $\pi$. 

Therefore, for an $f_\theta(r)$ combining the two types of behaviour above, $\beta^2/\alpha$ and $\alpha$ are not contained in the same half plane for all $R_1\leq r\leq R_2$ and $\epsilon\leq \theta\leq \pi/2-\epsilon$.
\ere

\subsection{The standard duality argument}\label{sec:Schatz}

\begin{definition}[Adjoint solution operator $\cS^*$]\label{def:asol}
Given $f\in L^2(\Omega)$, let $\cS^*f$ be defined as the solution of the variational problem
\beq\label{eq:S*vp}
\tfind\,\, \cS^*f \in H^1_0(\Omega) \,\,\tst\,\, a(v,\cS^*f) = \int_{\Omega}v\,\overline{f}\,\,\tfa v\in H^1_0(\Omega).
\eeq
\end{definition}

The conditions for quasioptimality below are formulated in terms of $\cS^*$. However, we record immediately 
in the following lemma that $\cS^* f$ is just the complex-conjugate of a solution of the PML variational problem \eqref{eq:PML_vf}. 

\begin{lemma}[The adjoint solution is the complex conjugate of a Helmholtz solution]\label{lem:asol_sol}
With $\cS^*$ is defined by \eqref{eq:S*vp},
\beqs
a(\overline{\asol f}, w) = \int_\Omega\overline{f}\,\overline{w} \quad \tfa w \in H^1_0(\Omega).
\eeqs
\end{lemma}

\bpf
By the definitions of $a(\cdot,\cdot)$ and the coefficients $A$ and $c^{-2}$ \eqref{eq:Ac}, and the facts that $H$ is real and $D$ is diagonal (and hence symmetric), 
$a(\overline{\PMLsol},w) = a(\overline{w},\PMLsol)$ 
for all $\PMLsol,w \in H^1_0(\Omega);$
the result then follows from the definition of $\asol$ \eqref{eq:S*vp}.
\end{proof}

\begin{definition}[$\eta(V_N)$]
Given a sequence $(V_N)_{N=0}^\infty$ of finite-dimensional subspaces of $H_0^1(\Omega)$,
let
\beq\label{eq:etadef}
\eta(V_N):= \sup_{0\neq f\in L^2(\Omega)}
\min_{w_N\in V_N} \frac{\N{S^*f-w_N}_{H^1_k(\Omega)}}{\big\|
f\big\|_{L^2(\Omega)}}.
\eeq
\end{definition} 

\ble[Conditions for quasioptimality]\label{lem:Schatz}
If 
\beqs
k\,\eta(V_N) \leq \frac{1}{\Ccont} \sqrt{ \frac{\Amax}{2\big( \Amin+\cmin^{-2}\big)}},
\eeqs
then the Galerkin equations \eqref{eq:FEM} have a unique solution which satisfies
\beqs
\N{v-v_N}_{H^1_k(\Omega)}
 \leq \frac{2\Ccont}{\Amin}\left(\min_{w_N\in V_N} \N{v-w_N}_{H^1_k(\Omega)}\right).
\eeqs
\ele

\bpf[References for the proof] 
This is based on the G\aa rding inequality \eqref{eq:Garding}; see, e.g., \cite[Theorem 4.3]{MeSa:10} (when $A\equiv I$ and $c\equiv 1$) or \cite[Lemma 6.4]{LSW3} (for general $A$ and $c$). 
\epf

\subsection{The bound on $\eta(V_N)$ obtained using Theorems \ref{thm:LSW3} and 
\ref{thm:LSW4}}

\ble[Bound on $\eta(V_N)$ under Assumption \ref{ass:LSW3} or 
\ref{ass:LSW4}]\label{lem:MS_LSW4}
Suppose that $\obstacle_-, A_{\rm scat},$ and $c_{\rm scat}$ satisfy \emph{either} Assumption \ref{ass:LSW3} \emph{or} \ref{ass:LSW4}.
Suppose further that $\Omega_-, A_{\rm scat}, $ $c_{\rm scat}$, and $K\subset[k_0,\infty)$ are such that the solution operator of the exterior Dirichlet problem is polynomially bounded (in the sense of Definition \ref{def:poly_bound}).

Given $N>0$ there exist 
\bit
\item
$k_1, C_1, C_2, \sigma>0$, all  independent of $k$, $h$, $p$, and $N$, and  
\item $C_N>0$, independent of $k$, $h$, $p$, 
\eit
such that, for $k\in K\cap [k_1,\infty)$,
\beq\label{eq:etabound_LSW4}
k\,\eta(V_N) \leq C_1 \frac{hk}{p}\left(1+ \frac{hk}{p}\right) + C_2
k^{M}
 \left(
\left(\frac{h}{h+\sigma}\right)^p  
+k\left(
 \frac{hk}{\sigma p}\right)^p\right)
+C_N k^{1-N}.
\eeq
\ele

\bpf
The proof of the bound \eqref{eq:etabound_LSW4} using
Theorems \ref{thm:LSW3}/\ref{thm:LSW4} is 
 identical to the proof of \cite[Lemma 5.5]{LSW4}, which uses the results  
 \cite[Theorem 5.5]{MeSa:10} and \cite[Proposition 5.3]{MeSa:11}.
 The only difference between the present set up and \cite[Lemma 5.5]{LSW4} is that here we have $v= \vhigh + \vlow +\vres$, whereas \cite[Lemma 5.5]{LSW4} only has $v= \vhigh+ \vlow$. The term $\vres$, however, can be approximated by zero giving a term of the form $C_N k^{1-N}$ (other terms of this form arise, exactly as in the proof of \cite[Lemma 5.5]{LSW4}, from approximating in the regions where they are negligible \emph{either} $\vAfar$ and $\vAnear$ in Theorem \ref{thm:LSW4} \emph{or} $\vlow$ in Theorem \ref{thm:LSW3}).
  \epf

\subsection{Proof of Theorem \ref{thm:FEM1} from the bound on $\eta(V_N)$}

The existence of the solution $\PMLsol$ to the variational problem \eqref{eq:PML_vf} follows from \cite[Theorem 1.6]{GLS2}. Indeed, this result proves existence and uniqueness of the PML solution 
for $k$ is sufficiently large when $G(w)= \int_\Omega g \, \overline{w}$ for $g\in L^2(\Omega)$. Existence and uniqueness of the PML solution for $G\in (H^1_k(\Omega))^*$ follows from existence and uniqueness for $L^2$ right-hand sides since the problem is Fredholm (via the G\aa rding inequality \eqref{eq:Garding}).

To prove existence of the Galerkin solution $\PMLsol_N$ to \eqref{eq:FEM} under the conditions \eqref{eq:thresholdD}, we combine Lemmas \ref{lem:Schatz} and \ref{lem:MS_LSW4}
Indeed, the bound on $k\,\eta(V_N)$ \eqref{eq:etabound_LSW4} holds by Lemma \ref{lem:MS_LSW4}.
We choose $N>1$, and then increase $k_1>0$ (if necessary) so that
\beqs
C_N k^{1-N} \leq 
 \frac{1}{2\Ccont} \sqrt{ \frac{\Amin}{2\big( \Amin+\cmin^{-2}\big)}} \quad\tfa k \geq k_1.
\eeqs
After using this bound in \eqref{eq:etabound_LSW4}, we see that the conditions \eqref{eq:thresholdD} with $C_1$ sufficiently small and $C_2$ sufficiently large then ensure that $k\,\eta(V_N)$ is sufficiently small (independent of $k$), and the result follows from Lemma \ref{lem:Schatz}.

\section{The black-box framework and functional calculus}\label{sec:blackbox}

\subsection{Recap of the black-box framework} \label{subsec:bb}

Let $\hsc:= k^{-1}$ be the semiclassical parameter; in the literature, the semiclassical parameter is often denoted by $h$, but we use $\hsc$ to avoid a notational clash with the meshwidth of the FEM 
appearing in \S\ref{sec:intro} and \S\ref{sec:FEM}.

In this subsection, we briefly recap the abstract framework of \emph{black-box scattering} introduced in \cite{SjZw:91};
for more details, see the comprehensive presentation in \cite[Chapter 4]{DyZw:19}. 
In fact, we use the approach of \cite[\S2]{Sj:97}, where the black-box operator is a variable-coefficient Laplacian (with smooth coefficients) outside the black box, and not the Laplacian $-\hsc^2\Delta$ itself as in \cite[Chapter 4]{DyZw:19}
(although the operator still agrees with $-\hsc^2 \Delta$ outside a
  sufficiently large ball).

\paragraph{The operator $P_\hsc$.}
Let $\mathcal{H}$ be a Hilbert space with an orthogonal decomposition
\begin{equation} \label{eq:bbdec}\tag{BB1}
\mathcal{H}=\mathcal{H}_{R_{0}}\oplus L^{2}(\mathbb{R}^{d}\backslash B_{R_0}, \omega(x)\rd x),
\end{equation}
where the weight-function $\omega : \mathbb R^d \rightarrow \mathbb R$ is measurable and $\supp(1-\omega)$ is compact in $\Rea^d$.
We call $\mathcal{H}_{R_{0}}$ the ``black box''. We emphasise that, although standard examples of the subspace $\cH_{R_0}$ are $L^2(B_{R_0})$ or $L^2(B_{R_0} \cap \Omega_+)$
(see \S\ref{sec:bbexamples} below), $\cH_{R_0}$ need only be an abstract Hilbert space; see the discussion at the end of \cite[\S4.1]{DyZw:19}.
Let $\mathbf 1_{B_{R_0}}$ and $\mathbf 1_{\mathbb{R}^{d}\backslash B_{R_0}}$ denote the corresponding orthogonal projections. 
Let $P_{\hsc}$ be a family in $\hsc$ of self adjoint operators $\mathcal{H}\rightarrow\mathcal{H}$
with domain $\mathcal{D}\subset\mathcal{H}$ independent of $\hsc$ (so that, in particular, $\mathcal{D}$ is dense in $\mathcal{H}$). 
Outside the black box $\mathcal H_{R_0}$,
we assume that $P_\hsc$ equals $Q_\hsc$ defined as follows.
We assume that, for any multi-index $|\alpha|\leq 2$, there exist functions $a_{\hsc, \alpha} \in C^\infty(\mathbb R^d)$, uniformly bounded with respect to $\hsc$, independent of $\hsc$ for $|\alpha|=2$, and such that (i) for some $C_1>0$

\begin{equation} \label{eq:propq_ell}
\sum_{|\alpha|=2} a_{\hsc, \alpha}(x)\xi^\alpha \geq C_1|\xi|^2\quad \tfa x\in \Rea^d
\end{equation}
(where $\xi^\alpha:= \xi_1^{\alpha_1}  \ldots \xi_d^{\alpha_d}$),
(ii) for some $R_{\rm scat}>R_0$ 
\begin{equation*} 
\sum_{|\alpha|\leq 2} a_{\hsc, \alpha}(x)\xi^\alpha = |\xi|^2
\hspace{0.3cm}\text{for }|x|\geq R_{\rm scat},
\end{equation*}
and (iii) the operator $Q_\hsc$ defined by
\beq\label{eq:Qdef}
Q_\hsc := \sum_{|\alpha|\leq2}a_{\hsc,\alpha}(x)(\hsc D_x)^\alpha
\eeq
(where $D:= -\ri \partial$) is formally self-adjoint on $L^2(\mathbb R^d, \omega(x) \rd x)$.

We require the operator $P_{\hsc}$ to be equal to $Q_\hsc$ outside 
the black box $\mathcal{H}_{R_{0}}$ in the sense that 
\beq\label{eq:bbreq1}\tag{BB2}
 \boldsymbol{1}_{\mathbb{R}^{d}\backslash B_{R_0}}(P_{\hsc}u)=Q_\hsc (u\vert_{\mathbb{R}^{d}\backslash B_{R_0}}) \quad\tfor u\in \cD, \quad\tand\quad
\boldsymbol{1}_{\mathbb{R}^{d}\backslash B_{R_0}}\mathcal{D} \subset H^{2}(\mathbb{R}^{d}\backslash B_{R_0}).
\eeq
We further assume that  if, for some $\eps>0$,
\beq\label{eq:bbreq2a}\tag{BB3}
v\in H^{2}(\mathbb{R}^{d})\quad\tand \quad v\vert_{B_{R_{0}+\eps}}=0,\quad \text{ then }\quad v\in\mathcal{D},
\eeq
(with the restriction to $B_{R_0+\epsilon}$ defined in terms
  of the projections in \eqref{eq:bbreq1}; see also \eqref{eq:restriction} below)
and that
\beq\label{eq:bbreq2}\tag{BB4}
\boldsymbol{1}_{ B_{R_0}}(P_{\hsc}+\ri)^{-1}\text{ is compact from } \cH \rightarrow \cH.
\eeq
Under these assumptions, the semiclassical resolvent 
$
R(z,\hsc):=(P_{\hsc}-z)^{-1}:\mathcal{H}\rightarrow\mathcal{D}
$
is meromorphic for $\Imag \,z>0$ and extends to a meromorphic
family of operators of $\mathcal{H_{\rm comp}}\rightarrow\mathcal{D}_{\rm loc}$
in the whole complex plane when $d$ is odd and in the logarithmic
plane when $d$ is even \cite[Theorem 4.4]{DyZw:19};
where $\mathcal H_{\rm comp}$ and $\mathcal{D}_{\rm loc}$ are defined by
$$
\mathcal H_{\rm comp} := \Big\{ u\in \mathcal H \; : \; \mathbf 1_{\mathbb R^d\backslash B_{R_0}}u \in L^2_{\rm comp}(\mathbb R^d \backslash B_{R_0}) \Big\},
$$
(where $L^2_{\rm comp}$ denotes compactly-supported $L^2$ functions) and
\begin{align*}\nonumber
\mathcal D_{\rm loc} := \Big\{ u\in \mathcal{H}_{R_{0}}\oplus L_{\rm loc}^{2}(\mathbb{R}^{d}\backslash B_{R_0}) :  \tif\,\chi \in C^\infty_{\rm comp}(\mathbb R^d), \; \chi|_{B_{R_0}} \equiv 1   
\text{ then }  (\mathbf 1_{B_{R_0}}u,
\chi \mathbf 1_{\mathbb R^d\backslash B_{R_0}} u) \in \mathcal D \Big\}.
\end{align*}

\paragraph{The reference operator $P^{\text{\ensuremath{\sharp}}}_\hsc$.}

We now define the so-called \emph{reference operator} using the torus $ \mathbb T^d_{R_{\sharp}} := \quotient{\mathbb{R}^d}{(2R_{\sharp}\mathbb{Z})^d}$ for some $R_\sharp>0$ such that $\supp(1-\omega)\subset B_{R_\sharp}$.
We work with $[-R_{\sharp}, R_{\sharp}]^d$ as a fundamental domain for this torus. 
The black-box framework by itself requires that $R_{\sharp}>R_{\rm scat}$; for simplicity we take $R_{\sharp}> {\rm diam}(\Otr)$, so that $\Otr \subset [-R_{\sharp}, R_{\sharp}]^d$ 
(where we assume, without loss of generality, the origin is inside $\Otr$).
\footnote{In fact, we could modify the arguments below to work for $R_{\sharp}> R_1$ only, since we just need $\supp\,\varphi_\tr$ contained inside $B_{R_\sharp}$.}

Let
\beqs
\mathcal H^{\sharp} := \mathcal H_{R_0} \oplus  L^{2}(\mathbb T^d_{R_{\sharp}}\backslash B_{R_0},\omega(x)\rd x),
\eeqs
and let $\mathbf 1_{B_{R_0}}$ and  $\mathbf 1_{\mathbb T^d_{R_{\sharp}}\backslash B_{R_0}}$ denote the corresponding orthogonal
projections. We
 define
\begin{gather} 
\mathcal D^{\sharp} := \Big\{ u \in \mathcal H^{\sharp}: \; \tif\,\chi \in C^\infty_{\rm comp}(B_{R_{\sharp}}), \; \,\chi = 1 \text{ near } B_{R_0}, \tthen (\mathbf 1_{B_{R_0}}u,
\chi \mathbf 1_{\mathbb T^d_{R_{\sharp}}\backslash B_{R_0}} u) \in \mathcal D, \nonumber \\
 \tand\, (1-\chi)\mathbf 1_{\mathbb T^d_{R_{\sharp}}\backslash B_{R_0}}u \in H^2(\mathbb T^d_{R_{\sharp}}) \Big\}, \label{eq:defdomsharp}
\end{gather}
and, for any $\chi$ as in \eqref{eq:defdomsharp} and $u \in \mathcal D^{\sharp}$,
\begin{gather} 
\refop u := P_{\hsc}(\mathbf 1_{B_{R_0}}u, \chi \mathbf 1_{\mathbb T^d_{R_{\sharp}}\backslash B_{R_0}} u) 
+ Q_\hsc\big((1-\chi)\mathbf 1_{\mathbb T^d_{R_{\sharp}}\backslash B_{R_0}}u\big), \label{eq:defref}
\end{gather}
where we have identified functions supported in $B(0,
  R_{\sharp})\backslash B(0, R_0)\subset
\mathbb{T}^d_{R_\sharp}\backslash B(0, R_0)$ with the corresponding functions on $\mathbb R^d\backslash
B(0, R_0)$ -- see the paragraph on notation below.

Let $q_\hsc \in S^2(\mathbb T^d_{R_\sharp})$ denote the principal symbol of $Q_\hsc$ as a semiclassical pseudodifferential operator acting on the torus $\mathbb T_{R_\sharp}^d$ (see Appendix \ref{app:sct} for a review of semiclassical pseudodifferential operators on $\mathbb
T^d_{R_{\sharp}}$); i.e., 
\beqs
q_\hsc(x,\xi)= \sum_{|\alpha|\leq2}a_{\hsc,\alpha}(x)\xi^\alpha.
\eeqs
We record for later the fact that \eqref{eq:propq_ell}, \eqref{eq:Qdef}, and the uniform boundedness of $a_{\hsc,\alpha}(x)$ with respect to $\hsc$ imply that there exist $C_1, C_2>0$ such that 
\beq\label{eq:Qnew}
C_1|\xi|^2 \leq q_\hsc(x,\xi) \leq C_2|\xi|^2 \quad\text{ for sufficiently large $\xi$ and all $x$}.
\eeq
The idea behind these definitions is that we have glued our black box
into a torus instead of $\mathbb{R}^d$, and then defined on the torus an operator $\refop$ that can be thought of as $P_{\hsc}$ in $\mathcal{H}_{R_{0}}$ and $Q_\hsc$ in
$(\mathbb{R}/2R_{\sharp}\mathbb{Z})^{d}\backslash B_{R_0}$;
see Figure \ref{fig:bb}.
The resolvent $(\refop + i)^{-1}$ is compact (see \cite[Lemma 4.11]{DyZw:19}), and hence the spectrum of $\refop$, denoted by $\operatorname{Sp}\refop$, is discrete (i.e., countable and with no accumulation point).

We assume that 
the eigenvalues of $\refop$ satisfy the \emph{polynomial
growth of eigenvalues condition }
\begin{equation}\label{eq:gro}\tag{BB5}
N\big(\refop,[-C,\lambda]\big)=O(\hsc^{-d^{\sharp}}\lambda^{d^{\sharp}/2}),
\end{equation}
for some $d^{\sharp}\geq d$, where $N(\refop,I)$ is the number of eigenvalues of
$\refop$ in the interval $I$, counted with their multiplicity. 
When $d^{\sharp}=d$, the asymptotics \eqref{eq:gro} correspond to a Weyl-type upper bound, 
and thus \eqref{eq:gro} can be thought of as a weak Weyl law. 

We summarise with the following definition.

\begin{definition} \mythmname{Semiclassical black-box operator} \label{def:bb}
We say that a family of self-adjoint operators $P_{\hsc}$ on a Hilbert space $\mathcal H$, with dense domain $\cD$, independent of $\hsc$, is a semiclassical black-box operator if $(P_{\hsc}, \mathcal H)$ satisfies (\ref{eq:bbdec}),  (\ref{eq:bbreq1}),  (\ref{eq:bbreq2a}), (\ref{eq:bbreq2}), (\ref{eq:gro}).
\end{definition}

\begin{figure}[h!]
\begin{center}
    \includegraphics[scale=0.4]{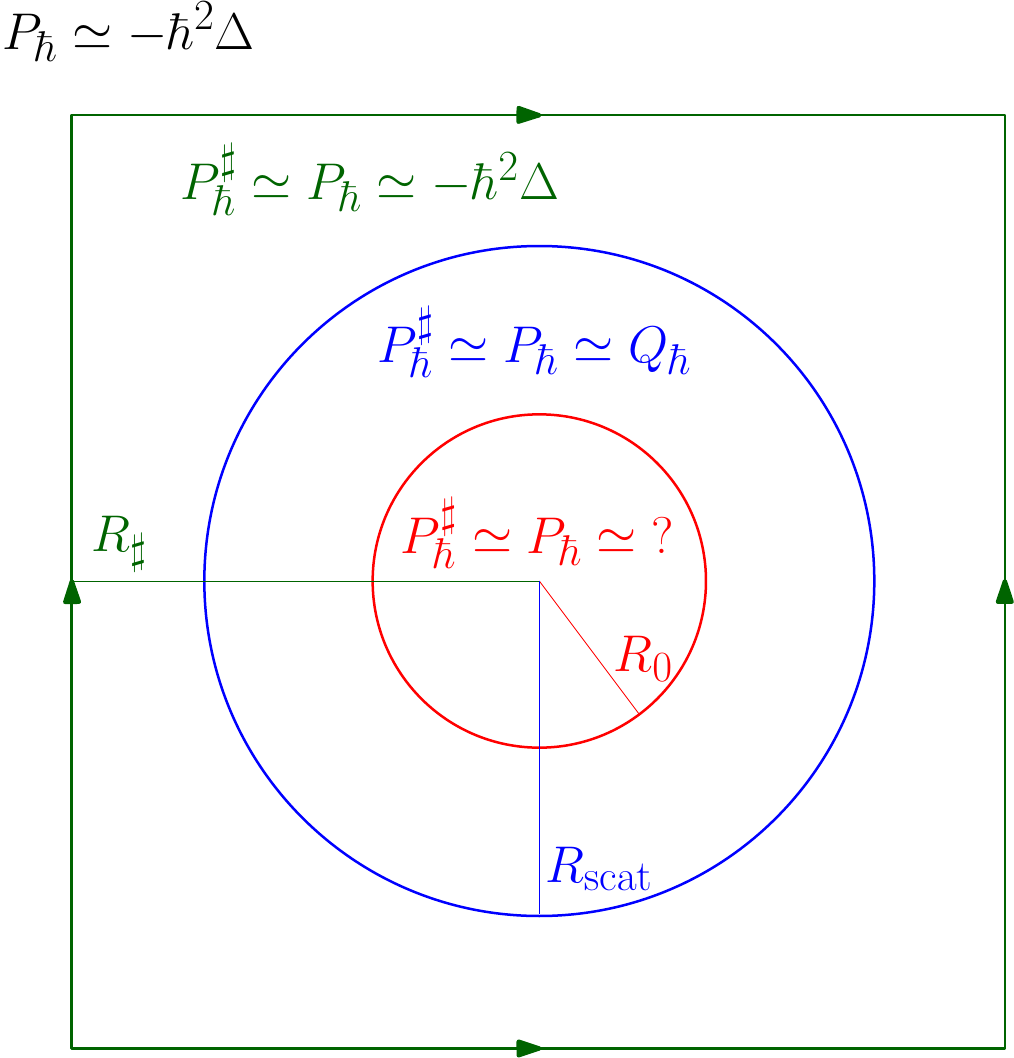}
  \end{center}
    \caption{The black-box setting. The symbol $\simeq$ is used to denote equality in the sense of \eqref{eq:bbreq1} and \eqref{eq:defref}.}
       \label{fig:bb}
\end{figure}

\paragraph{Notation.} We identify in the natural way:
\begin{itemize}
\item the elements of $\{ 0 \} \oplus L^2(\mathbb T_{R_{\sharp}}^d \backslash B_{R_0}) \subset \mathcal H^{\sharp}$,
\item the elements of $L^2(\mathbb T_{R_{\sharp}}^d \backslash B_{R_0})$, 
\item the elements of $L^2(\mathbb T_{R_{\sharp}}^d)$ 
  supported outside $B_{R_0}$,
\item the elements of $L^2(\mathbb R^d)$ supported in $[-R_{\sharp},R_{\sharp}]^d  \backslash B_{R_0}$,
\item and the elements of $\{ 0 \} \oplus L^2(\mathbb R^d \backslash B_{R_0}) \subset \mathcal H$ whose orthogonal projection onto 
$L^2(\mathbb R^d \backslash B_{R_0})$ is supported in $[-R_{\sharp},R_{\sharp}]^d \backslash B_{R_0}$.
\end{itemize}
If $v \in \mathcal H$ and $\chi \in C^\infty_{\rm comp}(\mathbb R^d)$ is equal to some constant $\alpha$ on a neighbourhood of $B_{R_0}$, we \emph{define}
\beq\label{eq:mult}
\chi v := (\alpha\mathbf 1_{B_{R_0}} v,  \chi \mathbf 1_{\mathbb R^d \backslash B_{R_0}} v) \in \mathcal H.
\eeq
(for example, using this notation, the requirements on $u$ in the definition of $\cD^\sharp$ \eqref{eq:defdomsharp} are $\chi u \in \cD$ and $(1-\chi)u \in H^2(\mathbb{T}^d_{R_{\sharp}})$ for $\chi$ equal to $1$ near $B_{R_0}$).
If $v\in \mathcal H$ and $R>R_0$, we define
\beq\label{eq:restriction}
v|_{B_{R}} := \big(\mathbf 1_{B_{R_0}} v, (\mathbf 1_{\mathbb R^d \backslash B_{R_0}} v )|_{B_{R}\setminus B_{R_0}}\big) \in \mathcal{H}_{R_{0}}\oplus L^{2}(B_{R}\backslash B_{R_0}),
\eeq
(where the restriction of $\mathbf 1_{\mathbb R^d \backslash B_{R_0}} v$ to $B_{R}\setminus B_{R_0}$ is restriction in the standard sense, since $\mathbf 1_{\mathbb R^d \backslash B_{R_0}} v \in L^2 (\mathbb R^d\setminus B_{R_0})$) and, if  $v\in \mathcal H^{\sharp}$,
$$
v|_{B_{R}} := \big(\mathbf 1_{B_{R_0}} v, (\mathbf 1_{\mathbb T_{R_{\sharp}}^d \backslash B_{R_0}} v )|_{B_{R}\setminus B_{R_0}}\big) \in \mathcal{H}_{R_{0}}\oplus L^{2}(B_{R}\backslash B_{R_0}).
$$
Finally, if 
$R_0\leq r \leq R_\sharp$, we define the partial norms 
$$
\Vert u \Vert_{\mathcal H^\sharp(B_r)} =\Vert u \Vert_{\mathcal
  H(B_r)}:= \Vert u \Vert_{\mathcal H_{R_0} \oplus L^2(B_r\backslash B_{R_0})}, 
\qquad\Vert u \Vert_{\mathcal H^\sharp(B^c_r)} := \Vert \mathbf 1_{\mathbb{T}^d_{R_{\sharp}}\backslash B_{R_0}} u \Vert_{ L^2(\mathbb T^d_{R_\sharp}\backslash B_r)}
$$
and
$$
\Vert u \Vert_{\mathcal H(B^c_r)} := \Vert \mathbf 1_{\mathbb R^d \backslash B_{R_0}} u \Vert_{ L^2(\mathbb{R}^d\backslash B_r)}.
$$

\subsection{Scattering problems fitting in the black-box framework}\label{sec:bbexamples}

A wide variety of scattering problems fit in the black-box framework; see  \cite[\S4.1]{DyZw:19}, \cite[\S2.2]{LSW4}.
The present paper only uses that the exterior Dirichlet problem of Definition \ref{def:EDP} fits in this framework.

\begin{lemma}[Scattering by a Dirichlet obstacle in the black-box framework]\label{lem:obstacle}
Let $\obstacle_-, A_{\rm scat}, c_{\rm scat}, R_0,$ and $R_{\rm scat}$ be as in Definition \ref{def:EDP}. 
Then 
the family of operators
\beqs
P_\hsc v := - \hsc^{2} c_{\rm scat}^2 \nabla \cdot \big(A_{\rm scat}\nabla v)
\quad\text{ with the domain }\quad
\cD:= H^2(\obstacle_+)\cap H^1_0(\obstacle_+)
\eeqs
is a semiclassical black-box operator (in the sense of Definition \ref{def:bb}) with $\omega=c_{\rm scat}^{-2}$,
$Q_\hsc= -\hsc^2 c_{\rm scat}^2\nabla\cdot(A_{\rm scat}\nabla)$, 
and
\beqs 
\cH_{R_0} = L^2\big( B_{R_0} \cap \obstacle_+; c_{\rm scat}^{-2}(x)\rd x\big)
\quad \text{ so that } \quad \cH = L^2\big(\obstacle_+; c_{\rm scat}^{-2}(x)\rd x\big).
\eeqs
Furthermore the corresponding reference operator $\refop$ satisfies \eqref{eq:gro} with $d^{\sharp}=d$.
\end{lemma}

\begin{proof}
In \cite[Lemma 2.3]{LSW4} the result 
 is proved for Lipschitz $\Omega_-$ and $A_{\rm scat}$ and $c\in L^\infty$ with domain 
\beqs
\Big\{ v\in H^1(\obstacle_+), \, \nabla\cdot \big(A_{\rm scat}\nabla v\big) \in L^2(\obstacle_+), \, v=0 \ton \partial \obstacle_+\Big\};
\eeqs
by elliptic regularity, this domain equals $H^2(\obstacle_+)\cap H^1_0(\obstacle_+)$ when $\Omega_-$ and $A_{\rm scat}$ are smooth.
 \end{proof}

\subsection{The scaled operator $\scaled$ and its truncation}\label{sec:scaled}

\paragraph{The scaled operator $\scaled$.}
With $\chi \in C_{\rm comp}^\infty(B_{R_1})$ equal to 1 on $B_{R_0}$, we define the scaled operator
\begin{equation}
\label{e:defP}
\begin{gathered}
\scaled u:= P_\hsc(\chi u)+ (-\hsc^2\Delta_\theta)((1-\chi )u),
\end{gathered}
\end{equation}
where $\Delta_\theta$ is defined by \eqref{e:deltaTheta}
Although the domain and range of $P_{\hsc,\theta}$ strictly involve the scaled manifold (see \cite[Definition 4.31]{DyZw:19}, \cite[Equation A.3]{GLS1}), they can be naturally identified with $\cD$ and $\cH$, respectively. 

\paragraph{Truncation of the scaled operator (i.e., PML truncation).}
For the PML truncation, just as in \S\ref{sec:PML_intro}, we let $\Omega_{\tr}\subset \mathbb{R}^d$ be a bounded Lipschitz open set with $B_{R_{\tr}}\subset \Omega_{\tr}$. 
Just as for $P_{\hsc,\theta}$ on the whole exterior domain, the domain and range of $P_{\hsc,\theta}$ on the truncated domain strictly involve the scaled manifold (see \cite[\S A.3]{GLS2}). However, we can naturally identify them with the following:
\begin{align*}
\cH(\Omega_\tr)&:=\cH_{R_0}\oplus L^2(\Omega_\tr \setminus B_{R_0}),\\
\cD(\Omega_\tr)&:=\Big\{u\in \cH(\Otr) \,:\, \tif \chi \in C_{\rm comp}^\infty(B_{R_1}) \text{ with } \chi \equiv 1 \text{ near } B_{R_0} \text{ then } 
\\
&\hspace{2cm}
 \chi u\in \cD,\, (1-\chi)u\in H_0^1(\Omega_\tr),\, -\Delta_\theta ((1-\chi)u)\in L^2(\Omega_\tr)\Big\}.
\end{align*}

\bre[A different choice of reference operator]\label{rem:trunc_bb}
Instead of defining the reference operator $\refop$ using the torus $\mathbb T^d_{R_{\sharp}}$, we could instead define $\refop$ using a large ball or hypercube with zero Dirichlet boundary conditions; see \cite[Remark on Page 236]{DyZw:19}. We could therefore define the reference operator $\refop$ on the domain $\Omega_{\tr}$ used for the PML truncation, which would have the advantage that the domain of $\refop$ could be naturally identified with the domain of $\scaled$. 
We choose not to do this, however, since
our arguments extensively use pseudodifferential operators defined on the torus $\mathbb T^d_{R_{\sharp}}$, and part of our proof of the decomposition of Theorem \ref{thm:informal}/\ref{thm:mainbb} involve explicit computation with the eigenvalues of the Laplacian on $\mathbb T^d_{R_{\sharp}}$; see \S\ref{sec:334}. 
\ere

\paragraph{Definition of a suitable scaled operator on the torus.}
Fix $\delta >0$ so that $R_1 + 4 \delta < R_{\rm tr}$. In the course of the proof of the main result, we need a operator defined on $\mathbb{T}^d_{R_\sharp}$ and equal to $P_{\hsc,\theta}$ on $B_{R_1(1+3\delta)}\setminus B_{R_0}$. 
We therefore let $-\tDelta_\theta$ be defined by \eqref{e:deltaTheta} with $f_\theta$ replaced by a non-negative function $\widetilde{f}_\theta \in C^{\infty}([0,\infty);\Rea)$ such that 
\beq\label{eq:widetildef}
\widetilde{f}_\theta(r) = f_\theta(r) \quad\tfor r\leq R_1 + 3\delta \quad\tand \quad \widetilde{f}(r)= 0 \quad\tfor r\geq R_1+ 4 \delta;
\eeq
i.e., $-\tDelta_\theta = -\Delta_\theta$ for $r\leq R_1 + 3\delta$ and $-\tDelta_\theta = -\Delta$ for $r\geq R_1+ 4\delta$ (so that the coefficients of $-\tDelta_\theta$ are periodic on the torus $\mathbb{T}^d_{R_\sharp}$).
Define the operator $\widetilde{Q}_{\hsc,\theta}$ on $H^2(\mathbb{T}^d_{R_\sharp})$ by 
\beq\label{eq:widetildeQ}
\widetilde{Q}_{\hsc,\theta} u = Q_\hsc \big(\psi u\big) + ( -\hsc^2 \tDelta_\theta) \big( (1-\psi)u\big),
\eeq
where $\psi \in C^{\infty}_{\rm comp}(B_{R_1})$ with $\psi \equiv 1$ on $B_{R_{\rm scat}}$
(we use a tilde in the notation to denote that $\widetilde{Q}_{\hsc,\theta}$ is not just the natural scaling of $Q_\hsc$).
Let $\widetilde{q}_{\hsc,\theta}\in S^2(\mathbb{T}^d_{R_\sharp})$ denote the principal symbol of $\widetilde{Q}_{\hsc,\theta}$ as a semiclassical pseudodifferential operator acting on the torus $\mathbb{T}^d_{R_\sharp}$ (see \S\ref{app:sct}). 

\subsection{A black-box functional calculus for $\refop$}\label{sec:BBFC}

\paragraph{The Borel functional calculus.}

The operator $\refop$ on the torus with domain
$\mathcal{D}^\sharp$ is self-adjoint with compact resolvent
\cite[Lemma 4.11]{DyZw:19}, hence we can describe the Borel functional
calculus \cite[Theorem VIII.6]{ReeSim72} for this operator explicitly in terms of the orthonormal
basis of eigenfunctions $\phi^\sharp_j \in \hilbert^\sharp$ 
(with
eigenvalues $\lambda_j^\sharp$, appearing with multiplicity and depending on $\hsc$):~for $f$ a real-valued Borel function on $\RR,$
$f(\refop)$ is self-adjoint with domain
\begin{equation*}
\domain_f:=\bigg \{\sum a_j \phi_j^\sharp \in \cH^\sharp
\,\,: \,\,
\sum \big \lvert f(\lambda^\sharp_j)
a_j \big \rvert^2 <\infty\bigg\},
\end{equation*}
and if $v =\sum a_j \phi_j^\sharp\in \domain_f$ then
\beqs
f(\refop)(v):=\sum a_j f(\lambda^\sharp_j) \phi^\sharp_j.
\eeqs
For $f$ a bounded Borel function, $f(P^\sharp_\hsc)$ is a bounded
operator, hence in this case we can dispense with the definition of
the domain and allow $f$ to be complex-valued.

For $m \geq 1$, we then define
$\mathcal D_\hsc^{\sharp,m}$ as the domain of $(\refop)^m$,
i.e.,
\beqs
\mathcal D_\hsc^{\sharp,m}:= \Big\{
v \in \cH^\sharp \,:\, (\refop)^{\ell} v \in \cD^\sharp, \,\ell=0,\ldots,m-1
\Big\},
\eeqs
equipped with the norm 
\begin{equation} \label{eq:BB:norms}
\Vert v \Vert_{\mathcal D_\hsc^{\sharp, m}} :=  \Vert  v\Vert_{\mathcal H^{\sharp}} + \Vert (\refop)^m  v\Vert_{\mathcal H^{\sharp}},
\end{equation}
and $\mathcal D_\hsc^{\sharp,-m}$ as its dual (note that, in the exterior of the black box, the regularity imposed in the definition of $\mathcal D_\hsc^{\sharp,m}$ is that of periodic functions on the torus with $2m$ derivatives in $L^2$).
 We also  define the partial norms, for $m>0$, 
\beqs
\Vert v \Vert_{\mathcal D_\hsc^{\sharp, m}(B)} :=  \Vert  v\Vert_{\mathcal H^{\sharp}(B)} + \Vert (\refop)^m  v\Vert_{\mathcal H^{\sharp}(B)},
\eeqs
 where $B$ equals one of $B_r$ or $(B_r)^c$ (with $R_0 \leq r \leq R_\sharp$) or $\Otr$. In addition, we let
\beq\label{eq:Dinfty}
\mathcal D_\hsc^{\sharp,\infty} := \bigcap_{m\geq0} \mathcal D_\hsc^{\sharp,m},
\eeq
so that $v \in \mathcal{D}_\hsc^{\sharp,\infty}$ iff $(\refop)^m v \in \mathcal D_\hsc^\sharp$ for all $m\in \mathbb{Z}^+$.

The following theorem is proved in \cite[Pages 23 and 24]{Da:96}; see also \cite[Theorem VIII.5]{ReeSim72}.

\begin{theorem} \label{thm:fundfc}
The Borel functional calculus enjoys the following properties.
  \begin{enumerate}
\item $f\rightarrow f(\refop)$ is a $\star$-algebra homomorphism. 
\label{it:fc3}
\item for $z \notin \mathbb R$, 
if $r_z(w):=(w-z)^{-1}$ then $r_z(P^\sharp)= (\refop - z)^{-1}$. \label{it:fc4}
\item If $f$ is bounded, $f(\refop)$ is a bounded operator
  for all $\hsc$, with $\Vert f(\refop) \Vert_{\mathcal L(\mathcal H^{\sharp})} \leq \sup_{\lambda \in \mathbb R}|f(\lambda)|$. \label{it:fc5}
\item If $f$ has disjoint support from $\operatorname{Sp}\refop$, then $f(\refop) = 0$. \label{it:fc6}
\end{enumerate}
\end{theorem}

\paragraph{The Helffer--Sj\"ostrand construction.}
In describing the \emph{structure} of the operators produced by the
functional calculus, at least for well-behaved functions $f,$ it is
useful to recall the Helffer--Sj\"ostrand construction of the
functional calculus \cite{HeSj:89}, \cite[\S2.2]{Da:96} (which can also be used to prove the spectral
theorem to begin with; see \cite{Da:95}).

We say that $f \in \mathcal A$ if $f \in C^\infty(\mathbb R)$ and
there exists $\beta < 0$, such that, for all $r>0$, there exists  $C_r>0$ such that $|f^{(r)}(x)| \leq C_r (1+|x|^2)^{(\beta - r)/2}$.

Let $\tau \in C^\infty(\mathbb R)$ be such that $\tau(s) = 1$ for $|s|\leq 1$ and $\tau(s) = 0$ for $|s|\geq 2$. 
Finally, let $\frak n \geq 1$.
We define an 
$\frak n$-almost-analytic
extension of $f$, denoted by $\widetilde f$, by
$$
\widetilde f(z) :=\left( \sum_{m=0}^{\frak n} \frac{1}{m!} \big(\partial^m f(\operatorname{Re}z)\big)\,(\ri\operatorname{Im}z)^m \right) \tau\left(\frac{\operatorname{Im}z}{\langle\operatorname{Re}z\rangle}\right)
$$
(observe that $\widetilde{f}(z)=f(z)$ if $z$ is real).
For $f\in \mathcal A$, we define
\beq\label{eq:HS1}
f(\refop) := - \frac 1 \pi \int_{\mathbb C} \frac{\partial \widetilde f}{\partial \bar z}(\refop - z)^{-1} \; \rd x\rd y,
\eeq
where $\rd x\rd y$ is the Lebesgue measure on $\mathbb C$. The integral on
the right-hand side of \eqref{eq:HS1} converges; see, e.g., \cite[Lemma 1]{Da:95}, \cite[Lemma 2.2.1]{Da:96}.  This definition can be shown to be independent of the choices of
$\frak n$ and $\tau,$ and to agree with the operators defined by the
Borel functional calculus for $f \in \mathcal{A}$; see \cite[Theorems 2-5]{Da:95}, \cite[Lemmas 2.2.4-2.2.7]{Da:96}. 

\paragraph{Pseudodifferential properties of the functional calculus.}

We say that $E_\infty \in \mathcal L(\mathcal H^{\sharp})$ is 
$\residual$
if, for any
$N>0$ and any $m>0$, there exists $C_{N,m}>0$ such that
\begin{equation*}
\Vert E_\infty \Vert _{\mathcal D_\hsc^{\sharp,-m} \rightarrow \mathcal D_\hsc^{\sharp,m}} \leq C_{N,m} \hsc^N
\end{equation*}
(compare to \eqref{eq:residual} below).
Operators in the functional calculus are pseudo-local in the following sense.

\begin{lemma}[Pseudolocality]
\label{lem:funcloc1}
Suppose $f\in \mathcal{A}$ is independent of $\hsc$, and $\psi_1, \psi_2 \in C^\infty(\mathbb T_{R_{\sharp}}^d)$ are constant
near $B_{R_0}$. If $\psi_1$ and $\psi_2$ have disjoint supports, then
\beq\label{eq:pseudolocal1}
\psi_1 f(\refop) \psi_2= O(\hsc^\infty)_{\mathcal D_\hsc^{\sharp,-\infty} \rightarrow \mathcal D_\hsc^{\sharp,\infty}}.
\eeq
\end{lemma}

\bpf
On a smooth manifold with boundary, this result
follows from the fact that $f(\refop)$ is a
pseudodifferential operator, and hence pseudo-local.  Here, it follows from combining the corresponding result about the
resolvent \cite[Lemma 4.1]{Sj:97} (i.e., \eqref{eq:pseudolocal1} with
$f(w):= (w-z)^{-1})$) with \eqref{eq:HS1} and then integrating (as
discussing in a slightly different context in \cite[Paragraph after
proof of Lemma 4.2]{Sj:97}).
\epf

\

Furthermore, we can show from \cite[\S4]{Sj:97} that, modulo a negligible term,  away from the black box the functional calculus is given by the semiclassical pseudodifferential calculus in the sense of our next lemma.
The following lemma uses the notion of semiclassical pseudodifferential operators on $\mathbb
T^d_{R_{\sharp}}$ (including the concept of the \emph{operator wavefront set} $\WFh $), recapped in Appendix \ref{app:sct}.

\begin{lemma}[Pseudodifferential properties away from the black box]
 \label{lem:funcloc2}

If $f\in C^\infty_{\rm comp}(\mathbb R)$ is independent of $\hsc$ and $\chi\in C^\infty(\mathbb T_{R_{\sharp}}^d)$ is equal to zero near $B_{R_0}$, then $f(Q_\hsc)\in \Psi^{-\infty}_\hsc(\mathbb T^d_{R_\sharp})$ with
\beq\label{eq:Sj1}
\chi f(\refop) \chi = \chi f(Q_\hsc) \chi + O(\hsc^\infty)_{\mathcal D_\hsc^{\sharp,-\infty} \rightarrow \mathcal D_\hsc^{\sharp,\infty}}.
\eeq
\end{lemma}
\begin{proof}[References for the proof]
The relation \eqref{eq:Sj1} follows from \cite[Lemma 4.2 and the subsequent two paragraphs]{Sj:97} 
(similar to in the proof of Lemma \ref{lem:funcloc2}). 
The result \cite[Th\'eor\`eme 4.1]{HeRo:83} (see also \cite[Th\'eor\`eme III-11]{Rob87}, \cite[Theorem 8.7]{DiSj:99}) imply that $f(Q_\hsc)$ is a pseudodifferential operator on $\mathbb T^d_{R_\sharp}$. 
\end{proof}

\subsection{Black-box differentiation operator}

Finally, we define the (non-standard) notion of a family of black-box differentiation operators as a family of operators 
agreeing with differentiation outside the black box (note that there
is no a priori notion of derivative inside the black box itself).

\begin{definition}[Black-box differentiation operator]\label{def:BBdiff}
$(D(\alpha))_{\alpha \in \frak A}$ is a family of black-box
differentiation operators on $\mathcal D^{\sharp, \infty}_{\hsc}$ (defined by  \eqref{eq:Dinfty})
 if $\frak A$ is a family of $d$--multi-indices, and 
for any $\alpha$ and any $v \in C^\infty_{\rm comp}(\mathbb T^d_{R_{\sharp}}\backslash \overline{B_{R_0}})$,
$
D(\alpha)v 
=\partial^\alpha v.
$
\end{definition}

\section{The main decomposition result in the black-box setting}\label{sec:mainresult}

\subsection{The precise statement of Theorem \ref{thm:informal}}

In addition to the black-box notation introduced in \S\ref{sec:blackbox}, 
we use the notation that 
\beq\label{eq:C0}
C_0(\RR) :=\Big\{f \in C(\RR)\colon \lim_{\lambda \to
  \pm \infty} f(\lambda)=0\Big\}.
\eeq

\begin{theorem} \mythmname{The decomposition in the black-box setting} \label{thm:mainbb}
Let $P_{\hsc}$ be a semiclassical black-box operator on $\mathcal H$ (in the sense of Definition \ref{def:bb}). 
There exists $\Lambda>0$ 
such that the following holds.
Suppose that, for some $\hsc_0>0$,
there exists 
$\subsetH \subset (0,\hsc_0]$ such that the following two assumptions hold. 
\begin{enumerate}
\item There exists $M\geq0$ such that 
for any $\chi \in C^\infty_{\rm comp}(\mathbb R^d)$ equal to one near
$B_{R_0}$, there exists $C>0$ such that if $u \in \mathcal D$
  is an outgoing solution to $(P_\hsc -I)u=\chi g$, then
\begin{equation} \label{eq:res}
\Vert \chi u \Vert_{\mathcal H} \leq C \hsc^{-1-M}\Vert \chi g\Vert_{\mathcal H} \quad \tfa  \hsc \in \subsetH.
\end{equation} 
\item 
There exists $\mathcal E \in C_0(\mathbb R)$ that is nowhere zero on $[-\Lambda, \Lambda]$ such that 
\beq\label{eq:Friday1}
\mathcal E(\refop) = E + \residual,
\eeq
where $E$ has the following property:~there exists $\rho \in C^\infty(\mathbb T^d_{R_{\sharp}})$ equal to one near $B_{R_0}$, such that, for some $\alpha$-family of black-box differentiation operators $(D(\alpha))_{\alpha \in \frak A}$ and for some $C_{\mathcal E}(\alpha, \hsc)>0$,
\begin{equation} \label{eq:lowenest}
\N{\rho D(\alpha) E w}_{\mathcal H^{\sharp}} \leq C_{\mathcal E}(\alpha, \hsc)\Vert w \Vert_{\mathcal H^\sharp} \quad \tfa  w \in \mathcal D_\hsc^{\sharp, \infty} \tand \hsc \in \subsetH.
\end{equation}
\end{enumerate}

Given $\e>0$, there exist $\hsc_1>0$, $C_j>0,\, j=1,2,3,$ and $\lambda>1$ such that 
for all $R_{\tr}>(1+\e)R_1$, $B_{R_{\tr}}\subset  \Omega_{\tr}\Subset \mathbb{R}^d$, 
$\e<\theta<\pi/2-\e$, 
all $g\in \cH(\Otr)$, 
and all $\hsc \in \subsetH \cap (0,\hsc_1]$, the following holds.
The solution $v \in \mathcal D(\Otr)$ to 
\begin{equation} \label{eq:pde}
(P_{\hsc,\theta} -I)v = g \,\,\ton \Otr\quad \tand \quad v= 0 \,\,\ton \Gamma_\tr
\end{equation}
exists and is unique and there exists $\vhigh\in \cD(\Otr),\vlow\in \cD^{\sharp,\infty}_\hsc$, and $\vres\in \cD^{\sharp,\infty}_\hsc$ such that
\begin{equation} \label{eq:maindec}
v= \vhigh + \vlow + \vres
\end{equation}
and $\vhigh,\vlow,$ and $\vres$ satisfy the following properties.
The component $\vhigh\in \cD(\Otr)$ 
satisfies
\begin{equation} \label{eq:decHF}
\Vert \vhigh \Vert_{\mathcal H(\Otr) } + \big\| P_{\hsc,\theta}\vhigh\big\|_{\mathcal H(\Otr)}  \leq C_1 \Vert g \Vert_{\mathcal H(\Otr)} \quad \tfa  \hsc \in \subsetH \cap (0,\hsc_1].
\end{equation}
There exist
$\RfarA, \RfarB, \RlocB, \RlocA$ with
$R_0<\RfarA<\RfarB<\RlocB<\RlocA<R_1$ 
such that $\vlow\in \cD^{\sharp,\infty}_\hsc$ decomposes as
\begin{equation} \label{eq:decLF0}
\vlow = \vAnear + \vAfar,
\end{equation}
where $ \vAnear\in \mathcal{D}^\sharp$ is regular near the black box and negligible away from it, in the sense that
\begin{equation} \label{eq:decLF1}
\Vert D(\alpha) \vAnear \Vert_{\mathcal H^{\sharp}(B_{\RlocA}) } \leq C_2 C_{\mathcal E}(\alpha, \hsc) \sup_{\lambda\in[-\Lambda, \Lambda]} \big|\mathcal E(\lambda)^{-1}\big|\; \hsc^{-1-M} \Vert g \Vert_{\mathcal H(\Otr)} 
\end{equation}
for all $\hsc \in \subsetH \cap (0,\hsc_1], \alpha \in \mathfrak A$, 
and, for any $N,m>0$ there exists $C_{N,m}>0$ (independent of $\theta$) such that
\begin{equation} \label{eq:decLF2}
\Vert \vAnear \Vert_{\mathcal D_\hsc^{\sharp,m}((B_{\RlocB})^c) } \leq C_{N,m}\hsc^N \Vert g \Vert_{\mathcal H(\Otr)}   \quad \tfa  \hsc \in \subsetH \cap (0,\hsc_1] 
\end{equation}
and  $ \vAfar\in \cD(\Otr)$ is entire away from the black box and negligible near it, in the sense that
\begin{equation} \label{eq:decLF3}
\Vert \partial^\alpha \vAfar \Vert_{\mathcal H^{\sharp}((B_{\RfarA})^c) } \leq C_3 \lambda^{|\alpha|} \hsc^{-|\alpha|-1-M} \Vert g \Vert_{\mathcal H(\Otr)} \quad \tfa  \hsc \in \subsetH \cap (0,\hsc_1] \tand \alpha \in \mathfrak A,
\end{equation}
and, for any $N,m>0$ there exists $C_{N,m}>0$ (independent of $\theta$) such that
\begin{equation} \label{eq:decLF4}
\Vert \vAfar \Vert_{\mathcal D_\hsc^{\sharp,m}(B_{\RfarB}) } \leq C_{N,m}\hsc^N \Vert g \Vert_{\mathcal H(\Otr)}   \quad \tfa  \hsc \in \subsetH \cap (0,\hsc_1]. 
\end{equation}
Finally, $\vres\in \cD^{\sharp,\infty}_\hsc$ is negligible in the sense that for any $N,m>0$ there exists $C_{N,m}>0$ (independent of $\theta$) such that
\begin{equation} \label{eq:vresidual}
\Vert \vres \Vert_{\mathcal D_\hsc^{\sharp,m}(\Otr) } \leq C_{N,m}\hsc^N \Vert g \Vert_{\mathcal H(\Otr)}   \quad \tfa  \hsc \in \subsetH \cap (0,\hsc_1]. 
\end{equation}
In addition, if $\rho = 1$ in \eqref{eq:lowenest}, then the decomposition (\ref{eq:maindec}) can be constructed in such a way that instead of (\ref{eq:decLF0})--(\ref{eq:decLF4}), $\vlow\in\cD^{\sharp,\infty}_\hsc$ satisfies the global regularity estimate
\begin{equation} \label{eq:decLF5}
\Vert D(\alpha) \vlow \Vert_{\mathcal H^{\sharp} } \lesssim C_{\mathcal E}(\alpha, \hsc)  \sup_{\lambda\in[-\Lambda, \Lambda]} \big|\mathcal E(\lambda)^{-1}\big| \; \hsc^{-1-M} \Vert g \Vert_{\mathcal H(\Otr)} \quad \tfa  \hsc \in \subsetH\tand \alpha \in \mathfrak A
\end{equation}
and is negligible in the scaling region in the sense that for any $N,m>0$ there exists $C_{N,m}>0$ (independent of $\theta$) such 
\beq\label{eq:decLF5a}
\Vert \vlow \Vert_{\mathcal D_\hsc^{\sharp,m}((B_{R_1(1+\epsilon)})^c)} \leq C_{N,m}\hsc^N \Vert g \Vert_{\mathcal H(\Otr)}   \quad \tfa  \hsc \in \subsetH \cap (0,\hsc_1].
\eeq
Finally, 
If $\mathcal E(P^\sharp_\hbar) = E$ (i.e., with no $\residual$ remainder in \eqref{eq:Friday1}), then the functions 
$\vhigh$, $\vlow$, $\vAnear$, and $\vAfar$ are all independent of $\mathcal E$, and all the constants in the bounds above are independent of $\mathcal E$ as well.
\end{theorem}

\subsection{Discussion of Theorem \ref{thm:mainbb}}\label{sec:mainbb_discuss}

\paragraph{The first assumption (involving \eqref{eq:res}).}
This assumption is 
that the solution operator is polynomially bounded in $\hsc$. 
In the black-box setting, \cite{LSW1} proved that this assumption always holds with $M>5d/2$ and 
$\{\hbar^{-1}: \hbar \in \subsetH\}^c$ having arbitrarily small measure in $\Rea^+$ (see Part (ii) of Theorem \ref{thm:polyboundD}). The solution operator is then polynomially bounded because $\subsetH$ excludes (inverse) frequencies close to resonances.
(Under an additional assumption about the location of resonances, a similar result with a larger $M$ can also be extracted from \cite[Proposition 3]{St:01} by using the Markov inequality.)  For nontrapping problems, one expects 
(\ref{eq:res}) to hold with $M=0$ and $\subsetH = (0, h_0]$ (see Theorem \ref{thm:polyboundD} and the references therein).

\paragraph{The second assumption (involving \eqref{eq:Friday1} and \eqref{eq:lowenest}).}
This assumption is a regularity assumption that depends on the contents of the black box.
We later refer to \eqref{eq:lowenest} as the ``low-frequency estimate'', since the fact that $\mathcal E$ is nowhere zero on $[-\Lambda,\Lambda]$ means that it bounds low-frequency components.
The cut-off $\rho$ in \eqref{eq:lowenest} is needed when the black box
  contains, e.g., an analytic obstacle and the operator inside has
  analytic coefficients and we want to show that $Ew$ is analytic inside the black box.
  
To prove Theorem \ref{thm:LSW3}, we choose $\mathcal E \in C^\infty_{\rm comp}(\Rea^d)$ with $\mathcal E \equiv 1$ on $[-\Lambda, \Lambda]$, and $\rho\equiv 1$;
the low-frequency estimate \eqref{eq:lowenest} then corresponds to a bound on the eigenfunctions of $P_\hsc^\sharp$.
By using the flexibility to write $\mathcal E(P^\sharp_\hbar)$ as $E + \residual$, we can actually obtain the low-frequency estimate \eqref{eq:lowenest} from a bound on the eigenfunctions of $-\Delta$ on the torus, instead of those of the variable-coefficient operator; see 
\S\ref{sec:LSW3proof}. 

To prove Theorem \ref{thm:LSW4}, we choose $\mathcal E(\lambda) = \re^{-t|\lambda|}$, corresponding to a heat-flow estimate; see \S\ref{sec:LSW4proof}.
Since $E_\infty=0$, the decomposition is independent of $\mathcal E$, and this allows us to use a \emph{family} of $\mathcal E$s, depending on $t$, 
and hence a family of estimates as \eqref{eq:lowenest}. This feature allows us to tune the choice of the parameter $t$, depending on $\hbar$ and $\alpha$, to get the best possible estimate on $\vAnear$ in \eqref{eq:decLF1}.  
       
\paragraph{The component $\vhigh$.}
    
Comparing \eqref{eq:res} and \eqref{eq:decHF}, and recalling that in the nontrapping case \eqref{eq:res} holds with $M=0$, we see that $\vhigh$ satisfies a bound 
that is better, by at least one power of $\hsc$, than the bound satisfied by $u$;
this is the analogue of the property (ii) in \S\ref{sec:idea1} of the results of 
\cite{MeSa:10, MeSa:11, EsMe:12, MePaSa:13}, 
and is a consequence of the semiclassical ellipticity of $P_\hsc -1$ on high-frequencies (as discussed in \S\ref{sec:idea2}). The regularity of $\vhigh$ depends on the domain of the operator but not on any other features of the black box (in particular, not on the regularity estimate \eqref{eq:lowenest}).

\paragraph{The component $\vlow$.}

$\vlow$ is in the domain of arbitrary powers of the operator ($\vlow \in \mathcal D_\hsc^{\sharp,\infty}$) and so is smooth in an abstract sense. $\vlow$ is split further into two parts: $\vAnear$ and $\vAfar$, with 
$\vAnear$ regular near the black box and negligible away from it, and
$ \vAfar$ entire away from the black box and negligible near it;
Figure \ref{fig:line2} illustrates this set up (with ``$\vAnear$ analytic'' replaced by ``$\vAnear$ regular'').
Comparing \eqref{eq:res} and \eqref{eq:decLF1}/\eqref{eq:decLF3}, we see that, in the regions where they are not negligible,
$\vAnear$ and $\vAfar$
satisfy bounds with the same $\hsc$-dependence as $u$, but with improved regularity. These properties are the analogue of the property (i) in \S\ref{sec:idea1} of the results of 
\cite{MeSa:10}, \cite{MeSa:11}, \cite{EsMe:12}, \cite{MePaSa:13}. In particular, the regularity of $\ulow$ depends on the regularity inside the black box (from \eqref{eq:lowenest}), and, for the exterior Dirichlet problem with analytic obstacle and coefficients analytic in a neighbourhood of the obstacle, $\ulow$ is analytic.

\paragraph{The boundary conditions satisfied by each component.}

On \emph{both} $\Gamma_\tr$ \emph{and} on any boundaries in the
interior of the black box, each of the main components $\vhigh$,
$\vAfar$, and $\vAnear$ \emph{either} satisfies the same boundary
condition as the PML solution $v$ \emph{or} is negligible in a
neighbourhood of that boundary. Indeed, both $\vhigh$ and
$\vAfar \in \cD(\Otr)$, and thus satisfy the same boundary
conditions as $v$ in both the black box and on $\Gamma_\tr$. The
component $\vAnear\in \cD^{\sharp,\infty}_\hsc$, and thus satisfies
the same boundary condition(s) (if any) as the PML solution $v$ in the
black box; furthermore, by \eqref{eq:decLF2}, $\vAnear$ is negligible
near $\Gamma_\tr$.  This discussion was all for the case $\rho\neq 1$
in \eqref{eq:lowenest} (where $\vlow$ is split into $\vAfar$ and
$\vAnear$).  When $\rho=1$ in \eqref{eq:lowenest},
$\vlow\in \cD^{\sharp,\infty}_\hsc$, and thus satisfies the same
boundary condition(s) (if any) as the PML solution $v$ in the black
box, and is negligible itself in a neighbourhood of $\Gamma_\tr$ by
\eqref{eq:decLF5a} and the fact that $R_{\tr} > R_1(1+\e)$.

These facts about the boundary conditions are important when using the decomposition of Theorem \ref{thm:LSW4} (obtained from the general decomposition in Theorem \ref{thm:mainbb}) in proving Theorem \ref{thm:FEM1} about the $hp$-FEM. 
Indeed, Lemma \ref{lem:Schatz} reduces proving quasioptimality of the Galerkin solution to determining how well $v$ is approximated by the sequence of finite-element spaces $(V_N)_{N=0}^\infty$, with each $V_N \subset \cD(\Otr)$ (i.e., the spaces have the boundary conditions for $v$ ``built in'').  Via the decomposition $v=\vhigh + \vlow$, we then seek to determine how well $\vhigh$ and $\vlow$ are approximated in these spaces -- hence why we care about the boundary conditions.

\paragraph{The error term $\vres$.}
The reason the negligible error term $\vres$ appears in the decomposition \eqref{eq:maindec} is so that  $\vhigh$ satisfies the zero Dirichlet boundary condition on $\Gamma_\tr$, the importance of which is highlighted above. Note that if we did not care about
$\vhigh$ satisfying this boundary condition, we could include $\vres$ in $\vhigh$.

\paragraph{Comparison with the analogous result for the (non-truncated) Helmholtz solution in \cite[Theorem A]{LSW4}.}

By design, the assumptions of Theorem \ref{thm:mainbb} are exactly the same as the assumptions in the analogue of Theorem \ref{thm:mainbb} for the non-truncated Helmholtz problem, i.e., \cite[Theorem A]{LSW4}. The conclusions of Theorem \ref{thm:mainbb} are essentially the same as those of \cite[Theorem A]{LSW4}, except for the fact that 
the decomposition has the residual term $\vres$; as discussed in the previous paragraph, the reason for this is that we want $\vhigh$ to satisfy the zero Dirichlet boundary condition on $\Gamma_\tr$ (which is not present for the non-truncated Helmholtz problem).

\section{Proof of Theorem \ref{thm:mainbb}}\label{sec:blackboxresult}

The decomposition (\ref{eq:maindec}) is defined in \S \ref{subsec:abdec}
(and illustrated schematically in Figures \ref{fig:split} and \ref{fig:split_uL}). The estimates (\ref{eq:decHF}) and  (\ref{eq:decLF1})--(\ref{eq:decLF5})  are proved in \S \ref{subsec:high} and
 \ref{subsec:low} respectively.

In this proof, we shorten the notation $\residual$ to $\residualD$ to keep expressions compact. 

\subsection{The decomposition} \label{subsec:abdec}

\paragraph{Definition of the frequency cut-offs.}
Let
$\psi \in C^\infty_{\rm comp}(\mathbb R;[0,1])$ be such that 
$\supp \,\psi \subset [-2,2]$ and $\psi\equiv 1$ on $[-1,1]$.
For $\mu,\mu'>0$, let 
\begin{equation} \label{eq:psimu}
\psi_\mu := \psi\left(\frac{\cdot}{\mu}\right) \quad\tand\quad \psi_{\mu'} := \psi\left(\frac{\cdot}{\mu'}\right).
\end{equation}
We now assume that $\mu>2$ and choose $\mu'$ as a function of $\mu$ so that 
\beq\label{eq:psi_prop}
(1- \psi_{\mu'}) (1- \psi_\mu) = (1-\psi_{\mu} )
\quad\tand\quad 1 \notin \supp(1-\psi_{\mu'});
\eeq
these two conditions are ensured if $1\leq \mu' \leq \mu/2$ (hence the assumption that $\mu>2$). 

\paragraph{Choice of the parameter $\mu$.}

We now impose additional conditions on $\mu$. By \eqref{eq:Qnew}, there exists $\mu_0$ such that if $\mu\geq \mu_0$, then
\beq\label{eq:John1}
\big\{ (x,\xi) \, :\, |q_\hsc(x,\xi)|\geq \mu\big\}=\big\{ (x,\xi) \, :\, q_\hsc(x,\xi)\geq \mu\big\}.
\eeq
We then choose $\mu \geq \max\{\mu_0,\mu_1\}$, where $\mu_1$ is given by the following lemma.

\ble[Semiclassical ellipticity of $Q_\hsc$ and $\widetilde{Q}_{\hsc,\theta}$ for $\mu$ large enough]\label{lem:mu_elliptic}
Given $\epsilon>0$, there exists $\mu_1>2$ and $c_{\rm ell}>0$ such that if $\mu\geq \mu_1$ then the following hold.

(i) If $q_\hsc(x,\xi)\geq \mu$, then 
\begin{equation} \label{eq:newdefmu}
\langle \xi \rangle^{-2}(q_\hsc(x, \xi) -1) \geq c_{\rm ell} >0
\eeq
(i.e., $Q_\hsc- 1$ is semiclassically elliptic in this region of phase space).

(ii) If $\epsilon\leq \theta\leq \pi/2-\epsilon$, $x\in B_{R_1(1+3\delta)}\setminus B_{R_0}$, and $q_{\hsc}(x,\xi)\geq\mu$, then
\begin{equation*}
\langle \xi \rangle^{-2}\big|\widetilde{q}_{\hsc,\theta}(x, \xi) -1\big| \geq c_{\rm ell} >0.
\eeqs
(i.e., $\widetilde{Q}_{\hsc,\theta}- 1$ defined by  \eqref{eq:widetildeQ} is semiclassically elliptic in this region of phase space). 
\ele

\bpf
In each part we show that there exists a $\mu_1$ such that the conclusion holds, and set the final constant $\mu$ to be the maximum of the two.

(i) By the lower bound in (\ref{eq:Qnew}), there exists $\widetilde{\mu} > 1$ and $c_{\rm ell}>0$ such that
\beqs
|\xi| \geq \widetilde{\mu} \quad\text{ implies that } \langle \xi \rangle^{-2}(q_\hsc(x, \xi) -1) \geq c_{\rm ell} >0.
\eeqs
The lower bound \eqref{eq:Qnew} also ensures that there exists $\mu > 1$ such that $q_\hsc(x,\xi) \geq \mu$ implies that 
$|\xi| \geq \widetilde{\mu}$, and thus 
 \eqref{eq:newdefmu} holds. 

(ii) Recall from \S\ref{sec:scaled} that 
$\widetilde{Q}_{\hsc,\theta}= Q_\hsc$ on $B_{R_{\rm scat}}$ and
$\widetilde{Q}_{\hsc,\theta}= -\hsc^2 \Delta_\theta$ on $B_{R_1(1+3\delta)}\setminus B_{R_{\rm scat}}$. Therefore, by \cite[Theorem 4.32]{DyZw:19}, given $\epsilon>0$, there exist $C_1,C_2>0$ such that
\beqs
C_1 |\xi|^2 \leq \big|\widetilde{q}_{\hsc,\theta}(x,\xi)\big| \leq C_2 |\xi|^2
\eeqs
for all $x\in B_{R_1(1+3\delta)}\setminus B_{R_{\rm scat}}$, for all $\xi$, and for $\epsilon \leq \theta\leq \pi/2-\epsilon$.
The result then follows in a similar way to the proof of Part (i).
\epf

\

 Let 
\beq\label{eq:Lambda}
\Lambda := 2 \mu
\quad\text{ so that } \quad
\supp\,  \psi_{\mu} \subset [-\Lambda, \Lambda].
\end{equation}
Note that both $\mu$ and $\Lambda$ only depend on $q_\hsc$ and $\{\widetilde{q}_{\hsc,\theta}\}_{\epsilon \leq \theta\leq \pi/2-\epsilon}$.

\paragraph{The frequency cut-offs.}

We define, using the Borel functional calculus for $\refop$ (Theorem \ref{thm:fundfc}), 
\begin{equation}\label{eq:PiL}
\Pilow := \psi_{\mu}(\refop),
\end{equation}
and additionally
\beq\label{eq:Pihigh}
\Pihigh :=  (1 - \psi_{\mu})(\refop) = I-\Pilow \quad\tand\quad \Pihigh' := (1 - \psi_{\mu'})(\refop).
\eeq
By the first property in \eqref{eq:psi_prop} and the fact the Borel functional calculus is an algebra homomorphism (Part \ref{it:fc3} of Theorem \ref{thm:fundfc}), 
\begin{equation} \label{eq:PiPi}
\Pihigh' \Pihigh = \Pihigh.
\end{equation}
By Part  \ref{it:fc5} of Theorem \ref{thm:fundfc}, the operators $\Pilow, \Pihigh,$ and $\Pihigh'$ are bounded on $\Hilb^{\sharp}$, with
\begin{equation} \label{eq:boundPi}
\Vert \Pilow\Vert_{\mathcal L(\mathcal H^{\sharp})}, \; \Vert \Pihigh\Vert_{\mathcal L(\mathcal H^{\sharp})}, \; \Vert \Pihigh'\Vert_{\mathcal L(\mathcal H^{\sharp})} \leq 1,
\end{equation}
and they commute with $\refop$ by Part \ref{it:fc3} of Theorem \ref{thm:fundfc}.

By the definition of $\psi_\mu$ (\ref{eq:psimu}), $\langle t\rangle^m \psi_\mu(t)$ is a bounded function for all $m$, and thus $\Pilow: \cD^\sharp \to \cD_\hsc^{\sharp,\infty}$. Then, $\Pihigh:= I- \Pilow: \cD^\sharp \to \cD^\sharp$.

\paragraph{Definition of the decomposition.}

Let $v$ be the solution of \eqref{eq:pde}.
Given $\epsilon>0$, fix $\delta >0$ so that $R_1(1 + 4 \delta) < R_1(1 + \epsilon)$; the condition that $R_1(1 +\epsilon)< R_{\tr}$ implies that $R_1(1 + 4\delta)< R_\tr$ (which is what we assumed in \S\ref{sec:scaled}).
Let $\varphi_{\rm tr} \in C_{\rm comp}^\infty(\Rea^d;[0,1])$ be such that $\varphi_{\rm tr} \equiv 1$ on $B_{R_1(1 + \delta)}$ and $\supp \,\varphi_{\rm tr} \subset B_{R_1(1 + 2\delta)}$. 
After writing
\beqs
v = \varphi_\tr v + (1- \varphi_\tr) v,
\eeqs
we then treat $ \varphi_\tr v$ as an element of $\cD^\sharp$ and let
\beq\label{eq:vHigh_vLow}
\vH := \Pihigh
(\varphi_{\rm tr} v) \in \cD^\sharp, \quad \vL := 
\Pilow
(\varphi_{\rm tr} v) \in \cD^{\sharp,\infty}_\hsc, \quad \vP := (1-\varphi_{\rm tr}) v \in \cD(\Otr),
\eeq
so that
\beq\label{eq:decomp1}
v=( \vH + \vL) + \vP.
\eeq

\bre
The parentheses in \eqref{eq:decomp1} are present because, 
strictly speaking, one cannot add \emph{either} $\vL$ and $\vP$ \emph{or} $\vH$ and $\vP$ individually, since
$\vL,\vH\in \cD^\sharp$ are functions on the torus and $\vP \in \cD(\Otr)$ is a function on $\Otr$. However, by construction, $\vH+ \vL$ is identically zero on $\overline{(\Otr)^c}$, and hence can be thought of as an element of $\cD(\Otr)$ by restriction. Similar sums, e.g., \eqref{eq:LEG1}, arise below, but we omit the parentheses.
\ere

We show below that, given $\e>0$, there exist $\hsc_1>0$, $C>0$ 
such that, for all $R_{\tr}>R_1(1+\e)$, $B_{R_{\tr}}\subset  \Omega_{\tr}\Subset \mathbb{R}^d$ with Lipschitz boundary, $\e<\theta<\pi/2-\e$, 
all $g\in \cH(\Otr)$, 
and all $\hsc \in \subsetH \cap (0,\hsc_1]$, 
\beq\label{eq:HF_eq}
\Vert \vH \Vert_{\mathcal H(\Otr)} + \Vert P_{\hsc,\theta}
 \vH \Vert_{\mathcal H(\Otr)} \leq C \Vert g \Vert_{\mathcal H(\Otr)},
\eeq
and
\beq\label{eq:PML_bound}
\Vert \vP \Vert_{\mathcal H(\Otr)} + \Vert P_{\hsc,\theta} \vP \Vert_{\mathcal H(\Otr)} \leq C \Vert g \Vert_{\mathcal H(\Otr)},
\eeq
When $\rho=1$ in the assumption \eqref{eq:lowenest}, we show that 
\beq\label{eq:rho=1}
\vL = \vlow + \residualD\varphi_\tr v,
\eeq
with $\vlow\in \cD^{\sharp,\infty}_\hsc$ satisfying \eqref{eq:decLF5} and \eqref{eq:decLF5a}. Otherwise, 
we show that 
\beq\label{eq:decomp2}
\vL = \vAnear + \vAfar + \residualD \varphi_\tr v,
\eeq
where $\vAnear$ and $\vAfar$ satisfy \eqref{eq:decLF1}-\eqref{eq:decLF4}, $\vAnear \in \cD_\hsc^{\sharp,\infty}$, and $\vAfar \in \cD(\Otr)$.

The idea now is to let 
$\vhigh$ equal $\vH + \vP$, and then the decomposition \eqref{eq:maindec} would hold by \eqref{eq:decomp1} and \eqref{eq:rho=1}/\eqref{eq:decomp2}.
However, we want $\vhigh$ to be in $\cD(\Otr)$, which is not guaranteed since,
although $\vP\in \cD(\Otr)$ (as noted above), $\vH$ need not be in $\cD(\Otr)$.
We therefore let $\widetilde\varphi_\tr\in C^\infty_{\rm comp}(\Rea^d;[0,1])$ be such that $\widetilde\varphi_\tr \equiv 1$ on a neighbourhood of $\supp \,\varphi_{\rm tr}$ and such that $\supp \,\widetilde \varphi_{\rm tr} \subset B_{R_1 + 3\delta}$. 
Then, by the definitions of $\vH$ \eqref{eq:vHigh_vLow} and $\Pihigh$ \eqref{eq:Pihigh} and Lemma \ref{lem:funcloc1}, 
\beq\label{eq:decomp3}
\vH = \widetilde\varphi_\tr\vH + \residualD \varphi_\tr v.
\eeq
We then set 
\beqs
\vhigh:= \widetilde\varphi_\tr\vH + \vP,
\eeqs
so that, by \eqref{eq:decomp1}, \eqref{eq:decomp2}, and \eqref{eq:decomp3},
\beq\label{eq:LEG1}
v = \vhigh + \vlow + \vres 
\quad\text{ where } \quad 
\vres= \residualD \varphi_\tr v.
\eeq

The bound \eqref{eq:vresidual} on $\vres$ (which completes the proof) follows from the result of \cite{GLS2} (recapped in Theorem \ref{t:scaledResolve} below) that $v$ inherits the polynomial bound on the resolvent enjoyed by $u$ \eqref{eq:res}.

This decomposition strategy is
summed-up in Figure \ref{fig:split}; with an overview of the decomposition of the low-frequency component $\vL$ in Figure \ref{fig:split_uL}.

\paragraph{Organisation of the rest of the proof.}

In \S\ref{subsec:PML} we prove the bound \eqref{eq:PML_bound} on $\vP$. 
In \S \ref{subsec:high} we prove the bound (\ref{eq:HF_eq}) on $\vH$.
In \S\ref{subsec:low} we prove that the decomposition (\ref{eq:decomp2}) holds, with $\vAnear$ and $\vAfar$ satisfying \eqref{eq:decLF1}-\eqref{eq:decLF4}.

In the rest of the proof we assume that $\hsc \in \subsetH$ and we omit the quantifiers and the explicit statement that the bounds hold uniformly for 
$R_{\tr}>R_1(1+\e)$ and $\e<\theta<\pi/2-\e$. We use the notation $\lesssim$ in bounds to indicate that the omitted constant is independent of $\hsc$.

\begin{figure}
\scalebox{0.85}{
%\begin{adjustwidth}{0em}{-18em}
\hspace{-5em}
 \begin{tikzpicture}%
  [>=stealth,
   shorten >=1pt,
   align = center,
   node distance=3cm and 4.5cm,
   on grid
  ]
\node (1)  {$v$};
\node (2) [below=of 1] {\begin{minipage}{0.2\textwidth}
            \begin{gather*} 
            \varphi_{\tr} v \\ 
            \text{\small considered as an element} \\
            \text{\small of the reference torus} 
            \end{gather*}
        \end{minipage}
};
\node (21) [right=of 2] {\begin{minipage}{0.2\textwidth}
            \begin{gather*} 
            \vP:= (1-\varphi_{\tr}) v  
            \end{gather*}
        \end{minipage}};
\node (31) [below=of 2] {\begin{minipage}{0.2\textwidth}
            \begin{gather*} 
            \vL \\ 
            \text{\small low-frequency part} 
            \end{gather*}
        \end{minipage}};
\node (32) [right=of 31] {\begin{minipage}{0.2\textwidth}
            \begin{gather*} 
            \vH \\ 
            \text{\small high-frequency part} 
            \end{gather*}
        \end{minipage}};
\node (A) [right=of 32]
   {\,\,$\widetilde{\varphi}_\tr \vH$\,\,}; %\\ 
\node (B) [right=of A] {\,\,$\vhigh$};
\node (42) [below=of 31] {
\begin{minipage}{0.2\textwidth}
            \begin{gather*} 
            \vAfar \\ 
           \text{\small negligible near $B_{R_0}$} \\
            \text{\small entire away from $B_{R_0}$} 
            \end{gather*}
        \end{minipage}};
\node (41) [left=of 42] {\begin{minipage}{0.2\textwidth}
            \begin{gather*} 
            \vAnear \\ 
            \text{\small regular near $B_{R_0}$} \\
            \text{\small negligible away from $B_{R_0}$} 
            \end{gather*}
        \end{minipage}};
\node (43) [right=of 42] {\begin{minipage}{0.2\textwidth}
            \begin{gather*} 
\residualD \varphi_\tr v\\ 
            \text{small part}
            \end{gather*}
        \end{minipage}};
\node (44) [right=of 43] {\begin{minipage}{0.2\textwidth}
            \begin{gather*} 
\residualD \varphi_\tr v\\ 
            \text{small part}
            \end{gather*}
        \end{minipage}};
%\node (44) [right=of 43] {$\;$};
\node (51) [below=of 42] {$\vlow$};
%\node (52) [below=of 44] {$\uhigh$};
\node (54) [below=of 44] {$\vres$};

\path[->]
(1) edge node[left] {$\varphi_\tr \in C^\infty_{\rm comp}(B_{R_1(1+2\delta)})$} (2)
(1) edge node[right] {\,\,$1-\varphi_\tr$} (21)
(2) edge node[left] {$\Pilow$} (31)
    edge node[right] {$\hspace{0.2cm}\Pihigh$} (32)   
(31) edge node {} (42)
(31) edge node {} (41)
(31)  edge node {} (43) 
(32) edge node[right] {\,\,$1-\widetilde{\varphi}_\tr$} (44)
(32) edge node[above] {\vspace{-2ex}$\widetilde{\varphi}_\tr$} (A) % this one not working
(A) edge node {} (B) % this one not working 
(21) edge node {} (B)
(41) edge node {} (51)
(42) edge node {} (51)
(43) edge node {} (54)
(44) edge node {} (54)
;
\end{tikzpicture}
%\end{adjustwidth}
}
\caption{Decomposition of the PML solution described in \S\ref{subsec:abdec}
(when $\rho\neq 1$ in \eqref{eq:lowenest})
} \label{fig:split}
\end{figure}
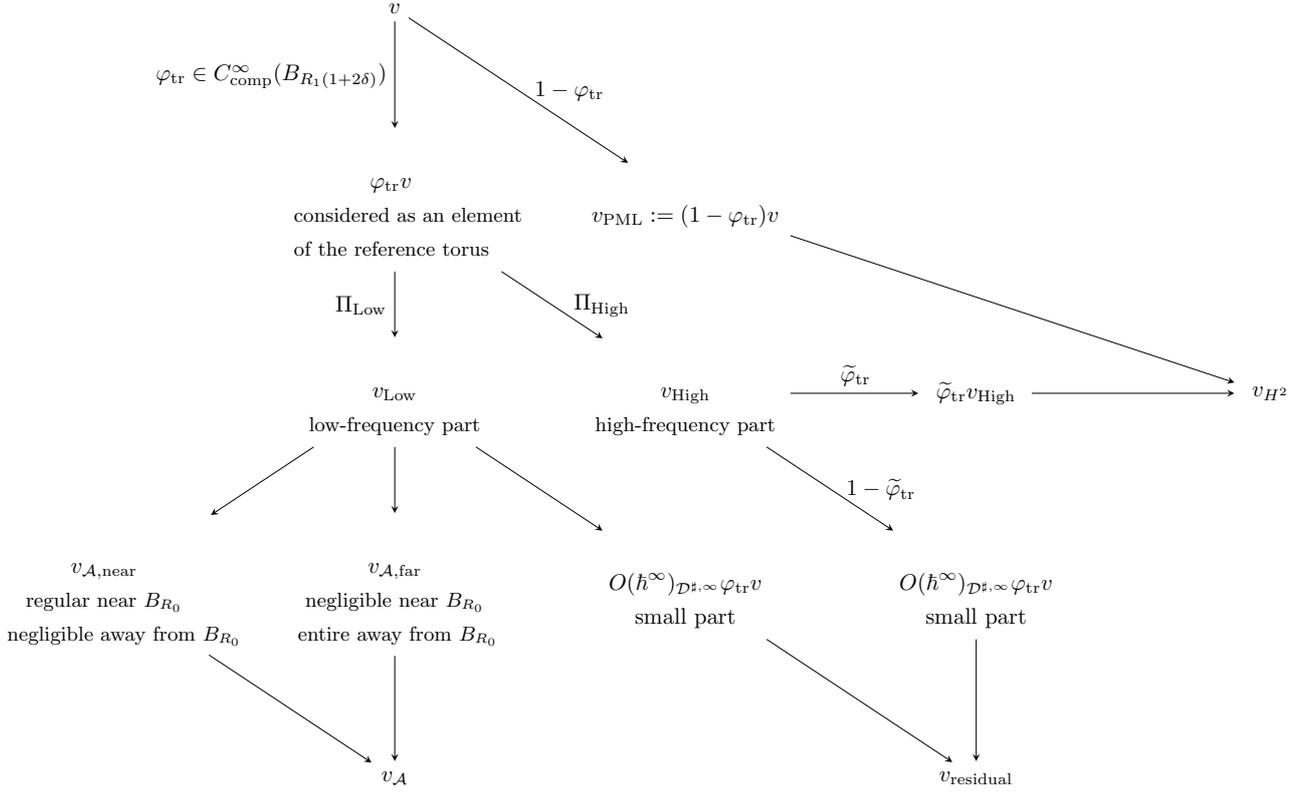

\subsection{The component near the PML boundary}\label{subsec:PML}

In this subsection we prove that the bound \eqref{eq:PML_bound} on $\vP$ holds.
We first recap results from \cite{GLS2} about PML truncation.

\subsubsection{Recap of three results from \cite{GLS2}}

The first result is a special case of the result from \cite[Theorem 1.6]{GLS2} that the solution operator of the PML problem ``inherits" the $\hsc$-dependence of the solution operator of the original (nontruncated) Helmholtz problem.

\begin{theorem}\mythmname{Simplified version of \cite[Theorem 1.6]{GLS2}}
\label{t:scaledResolve}
Suppose Point 1 in Theorem \ref{thm:mainbb} holds; i.e., the solution operator of the black-box problem is polynomially bounded for $\hsc \in \subsetH$. 
Given $\e>0$, there exist $C,\hsc_0>0$ such that the following holds. 
For all $R_{\tr}>R_1(1+\e)$, $B_{R_{\tr}}\subset  \Omega_{\tr}\Subset \mathbb{R}^d$ with Lipschitz boundary, $\e<\theta<\pi/2-\e$, 
all $g\in \cH$ with $\supp \,g\subset \Omega_{\tr}$, and all 
$\hsc \in \subsetH\cap [0,\hsc_0]$, the solution $v$ to
\beqs
(P_{\hsc,\theta} -I)v = g \tin \Otr \quad\tand\quad v= 0 \ton \Gamma_\tr
\eeqs
(i.e., \eqref{eq:pde}) exists, is unique, and satisfies
\beq\label{eq:res_PML}
\N{v}_{\cH(\Otr)} + \N{P_{\hsc,\theta}v}_{\cH(\Otr)} \leq C \hsc^{-1-M} \N{g}_{\cH(\Otr)}.
\eeq
\end{theorem}

The next result is an elliptic estimate on the PML solution near the boundary (proved using the structure of $-\Delta_\theta$ in the scaling region).

\begin{lemma}\mythmname{Estimate on the PML solution near the boundary \cite[Lemma 4.4]{GLS2}}
\label{l:nearBoundary}
For any $\eps>0$, there exists $\hbar_0>0$ and $C>0$ so that for any $\eps<\theta<\pi/2-\eps$, $R_{\tr}>R_1(1+\eps)$, $B_{R_1}\Subset\Omega_{\tr}\subset \mathbb{R}^d$ with Lipschitz boundary, if $v$ is supported in 
$\Omega_{\tr}\setminus B_{R_1+\eps}$ and $v=0$ on $\Gamma_{\tr}$, then, for all $0<\hbar \leq \hbar_0$,
\begin{align}\label{eq:boundary1}
 &\Vert v \Vert_{H^1_\hbar(\Omega_{\tr})} \leq C \Vert  (P_{\hsc,\theta} -I)v \Vert_{L^2(\Omega_{\tr})}.
\end{align}
\end{lemma}

The final result is a Carleman estimate describing how solutions of $(-\hbar^2\Delta_\theta-1)v=f$ propagate in the scaling region.

\begin{lemma}\mythmname{Simplified version of \cite[Lemma 4.2, Equation 4.6]{GLS2}}
\label{l:carleman}
Given $\eps>0$ there exist $C_j>0, j=1,2,3,$ and $\hsc_0>0$ such that, for all $\eps\leq \theta\leq \pi/2-\eps$ 
and $0<\hbar<\hbar_0$,
\begin{multline}
\label{e:rightFromLeftNoRightBoundary}
\|v\|_{H_\hbar^1(\Otr\setminus B_{R_1+\eps})}
\leq C_1 
\|(-\hbar^2\Delta_\theta-1)v\|_{L^2(\Otr\setminus B_{R_1})}
+C_2\exp(-C_3 \hsc^{-1})
\|v\|_{H_{\hbar}^1(B_{R_1+\eps}\setminus B_{R_1})}.
\end{multline}
\end{lemma}

\subsubsection{Proof of the bound \eqref{eq:PML_bound} on $\vP$}

Since $\vP:= (1-\varphi_\tr)v$, 
\begin{align}\label{eq:Jeff1}
(P_{\hsc,\theta} -I)\vP 
&= (P_{\hsc,\theta} -I)(1-\varphi_\tr) v 
= (1-\varphi_\tr)g 
+ [P_{\hsc,\theta}, \varphi_\tr]v,
\end{align}
and the fact that $\varphi_\tr\equiv 1$ on $B_{R_1+\delta}$ implies that $\supp\,\vP \subset \Otr\setminus B_{R_1+\delta}$.
Thus, applying Lemma \ref{l:nearBoundary} with $\eps=\min\{\epsilon,\delta\}$, we see that the bound \eqref{eq:boundary1} implies that 
\begin{align*}
\|\vP\|_{H^1_\hsc(\Otr\setminus B_{R_1+\delta})} \lesssim \big\| (P_{\hsc,\theta} -I)\vP\big\|_{L^2(\Otr)} 
\lesssim \|g\|_{\cH(\Otr)} + \big\| [ P_{\hsc,\theta}, \varphi_\tr]v\big\|_{L^2(\Otr)}.
\end{align*}
Now, by direct computation and the fact that $\supp\nabla\varphi_\tr \subset B_{R_1+ 2\delta}\setminus B_{R_1+\delta}$,
\beq\label{eq:Jeff2}
\big\| [ P_{\hsc,\theta}, \varphi_\tr]v\big\|_{L^2(\Otr)} \lesssim \hsc\| v\|_{H^1_\hsc( B_{R_1+ 2\delta}\setminus B_{R_1+\delta})},
\eeq
so that 
\begin{align}\label{eq:Jeff3}
\|\vP\|_{H^1_\hsc(\Otr\setminus B_{R_1+\delta})}
\lesssim \|g\|_{\cH(\Otr)}+\hsc\| v\|_{H^1_\hsc( B_{R_1+ 2\delta}\setminus B_{R_1+\delta})}.
\end{align}
Using \eqref{eq:Jeff1} again, we have 
\beqs
\|P_{\hsc,\theta}\vP\|_{\cH(\Otr)} \lesssim \|g\|_{\cH(\Otr)} + \big\| [ P_{\hsc,\theta}, \varphi_\tr]v\big\|_{L^2(\Otr)} + \|\vP\|_{L^2(\Otr)},
\eeqs
and combining this with \eqref{eq:Jeff2} and \eqref{eq:Jeff3} (and recalling that $\supp \,\vP \subset  \Otr\setminus B_{R_1+\delta}$) we find that 
\beq\label{eq:Jeff4}
\|\vP\|_{\cH(\Otr)} + 
\|P_{\hsc,\theta}\vP\|_{\cH(\Otr)} \lesssim \|g\|_{\cH(\Otr)} + 
\hsc\| v\|_{H^1_\hsc( B_{R_1+ 2\delta}\setminus B_{R_1+\delta})}.
\eeq
Our plan is to use the Carleman estimate \eqref{e:rightFromLeftNoRightBoundary} to bound this last term in terms of $\|g\|_{\cH(\Otr)}$.
We first claim that \eqref{eq:res_PML} implies that
\beq\label{eq:Jeff5}
\| v\|_{H^1_\hsc( B_{R_1+\delta}\setminus B_{R_1})}\lesssim \hsc^{-1-M}\|g\|_{\cH(\Otr)};
\eeq
indeed, this follows by the combination of (i) the fact that $P_{\hsc,\theta}= -\hsc^2 \Delta_\theta$ for $R\geq R_1$, (ii) the fact that $-\Delta_\theta$ is elliptic (by, e.g., \cite[Theorem 4.32]{DyZw:19}), (iii) elliptic regularity (to obtain control of the $H^2_\hsc$ norm of $v$), and then (iv) interpolation (to obtain control of the $H^1_\hsc$ norm of $v$).
Then, the combination of \eqref{e:rightFromLeftNoRightBoundary} (with $\eps =\min\{\epsilon,\delta\}$) and \eqref{eq:Jeff5} implies that 
\begin{align*}
\| v\|_{H^1_\hsc( B_{R_1+ 2\delta}\setminus B_{R_1+\delta})} \lesssim \Big(1 + \exp(-C_3 \hsc^{-1}) \hsc^{-1-M} \Big)\|g\|_{\cH(\Otr)}.
\end{align*}
Combining this last inequality with \eqref{eq:Jeff4} and reducing $\hsc_0$ if necessary, the result \eqref{eq:PML_bound} follows.

\subsection{Proof of the bound (\ref{eq:HF_eq}) on $\vH$ (the high-frequency component)} \label{subsec:high}

\paragraph{Decomposing into parts that are ``near to'' or ``far from'' the black box.}

 Let $\varphi_0$, $\widetilde \varphi_0 \in C^\infty_{\rm comp}(\Rea^d;[0,1])$ be such that $\varphi_0 \equiv 1$ near $B_{R_0}$ and $\widetilde \varphi_0 \equiv 1$ in a neighbourhood of  
$\supp \,\varphi_0$, with $\supp \,\varphi_0 \subset \supp  \,\widetilde \varphi_0 \subset B_{R_1(1 - \delta)}$, so that, in particular, 
\beq\label{eq:support_prop1}
\refop= P_\hsc=P_{\hsc,\theta} \quad\text{  on the supports of } \varphi_0 \tand \widetilde \varphi_0.
\eeq
In addition, let $\varphi_1 := 1 - \varphi_0$ and let $\widetilde \varphi_1 \in C^\infty(\Rea^d; [0,1])$ be supported away from the black box $B_{R_0}$ and such that $\widetilde \varphi_1 \equiv 1$ in a neighbourhood of  $\supp \,\varphi_1$.
Finally, let $\widetilde\varphi_\tr\in C^\infty_{\rm comp}(\Rea^d;[0,1])$ be as in \S\ref{subsec:abdec}; i.e., equal to one on the support of $\varphi_{\rm tr}$ and so that $\supp \,\widetilde \varphi_{\rm tr} \subset B_{R_1(1 + 3\delta)}$; see Figure \ref{fig:cutoff_high}.
(Observe then that a tilde denotes a function with larger support than the corresponding function without the tilde.)

\begin{figure}
\begin{center}
    \includegraphics[width=\textwidth]{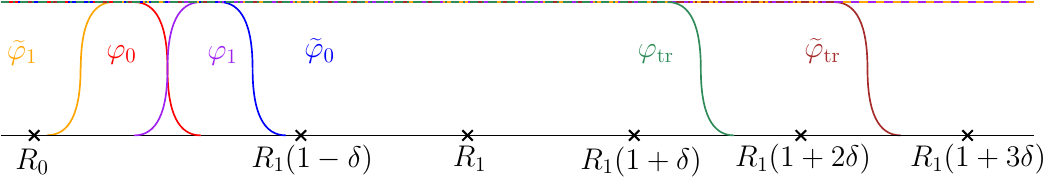}
  \end{center}
    \caption{The cut-off functions $\varphi_0, \widetilde{\varphi}_0, \varphi_1, \widetilde{\varphi}_1, \varphi_\tr$, and $\widetilde{\varphi}_\tr$ described at the start of \S\ref{subsec:high}.}
  \label{fig:cutoff_high}
\end{figure}

These definitions imply the following support properties
\beq\label{eq:support_prop}
\supp (1-\widetilde{\varphi}_{\rm tr})\cap \supp \,\varphi_{\rm tr}=\emptyset,\,\,
\supp (1-\widetilde{\varphi}_{0})\cap \supp \,\varphi_{0}=\emptyset,\,\,\tand\,\,
\supp (1-\widetilde{\varphi}_{1})\cap \supp \,\varphi_{1}=\emptyset.
\eeq
Starting from the definition $\vH:= \Pihigh \varphi_\tr v$ \eqref{eq:vHigh_vLow}, 
using that $\varphi_0+ \varphi_1=1$, the first and third support properties in \eqref{eq:support_prop}, Lemma \ref{lem:funcloc1}, and that $\varphi_0\varphi_\tr= \varphi_0$, we obtain that
\begin{align}\nonumber
\vH = \Pihigh \varphi_0 \varphi_\tr v + \Pihigh \varphi_1 \varphi_\tr v
&= \Pihigh \varphi_0 \varphi_\tr v +\widetilde\varphi_1 \Pihigh \varphi_1 \varphi_\tr v + \residualD \varphi_\tr v\\ \nonumber
&= \Pihigh \varphi_0 \varphi_\tr v +\widetilde\varphi_\tr \widetilde\varphi_1 \Pihigh \varphi_1 \varphi_\tr v + \residualD \widetilde\varphi_\tr v\\ \nonumber
&= \Pihigh \varphi_0 v +\widetilde\varphi_\tr\widetilde\varphi_1 \Pihigh \varphi_1 \varphi_\tr v + \residualD \widetilde\varphi_\tr v\\
&=: \vHnear + \vHfar + \residualD \widetilde\varphi_\tr v.
\label{eq:vH_decomp1}
\end{align}

\bre[The decomposition of $\vH$]
This decomposition of $\vH$ into ``near'' and ``far'' components is different from the non-truncated case in \cite{LSW4}. The reason we do it is we want the function to which $\Pihigh$ is applied in $\vHnear$ to be supported away from the scaling region (i.e., supported where $\refop= P_{\hsc,\theta}$)
-- see \eqref{eq:eqPiH} below. The component $\vHfar$ can then be dealt with via Lemma \ref{lem:funcloc2} (since it involves cut-offs supported away from the black box) and semiclassical ellipticity; see Step 4 below.
\ere

\paragraph{Overview of the rest of the proof of \eqref{eq:HF_eq}.}

We proceed in four steps; Steps 1-3 obtain the bound 
\begin{equation} \label{eq:vHnear_bound}
\Vert \vHnear \Vert_{\mathcal H^{\sharp} } + \big\| 
P_{\hsc,\theta}
\vHnear\big\|_{\mathcal H^{\sharp} }  \lesssim \Vert g \Vert_{\mathcal H(\Otr)},
\end{equation}
on $\vHnear$ and are the analogues of Steps 1-3 in \cite[\S3.2]{LSW4} that deal with $u_{\rm High}$ (although Step 1 is more involved because of the presence of the two operators $\refop$ and $P_{\hsc,\theta}$ as opposed to just $\refop$). Step 4 obtains the bound 
\begin{equation} \label{eq:vHfar_bound}
\Vert \vHfar \Vert_{\mathcal H^{\sharp} } + \big\| P_{\hsc,\theta}\vHfar\big\|_{\mathcal H^{\sharp} }  \lesssim \Vert g \Vert_{\mathcal H(\Otr)},
\end{equation}
on $\vHfar$ using ideas from Steps 2 and 3 (in a simplified setting). 

\paragraph{Step 1:~An abstract argument in $\cH^\sharp$ to bound $\vHnear$.}
Since $\Pihigh$ commutes with $\refop$ (by Part \ref{it:fc3} of Theorem \ref{thm:fundfc}) and $\refop=P_{\hsc,\theta}$ on $\supp \,\varphi_0\subset B_{R_0}$,
\begin{align} \nonumber
(\refop-I)\Pihigh({\varphi}_0 v) &= \Pihigh (\refop-I)({\varphi}_0 v) \\ &\hspace{-1cm}=  \Pihigh (P_{\hsc,\theta} -I)(\varphi_0 v) = \Pihigh \varphi_0 g+\Pihigh [P_{\hsc,\theta}, \varphi_0 ] v
= \Pihigh \varphi_0 g+\Pihigh [\refop,\varphi_0 ] v.\label{eq:eqPiH}
\end{align}
(Note that, strictly speaking, we should be writing the commutator 
$[P_{\hsc,\theta},\varphi_0 ]$ as $[P_{\hsc,\theta}, M_{\varphi_0} ]$, where 
multiplication is defined in the black-box setting by \eqref{eq:mult}; however, we abuse this notation slightly for simplicity.)
For $\lambda \in \mathbb R$, let
$$
f(\lambda) := (\lambda - 1)^{-1}(1 - \psi_{\mu'})(\lambda),
$$ 
and observe that $f\in C_0(\mathbb R)$ (defined by \eqref{eq:C0})  by the second property in \eqref{eq:psi_prop}.
Using (\ref{eq:PiPi}), the fact that the Borel calculus is an algebra
homomorphism (Part \ref{it:fc3} of Theorem \ref{thm:fundfc}), and
finally (\ref{eq:eqPiH}), we get 
\beq\label{eq:PiHcom}
\Pihigh (\varphi_0 v) = \Pihigh' \Pihigh (\varphi_0 v)
= f(\refop)(\refop-I) \Pihigh (\varphi_0 v) = f(\refop) \big(  \Pihigh \varphi_0 g+\Pihigh [\refop, \varphi_0] u \big).
\eeq
Since $f \in C_0(\mathbb R)$, $f(\refop)$ is uniformly bounded from $\Hilb^{\sharp}\to \Hilb^{\sharp}$ by Part \ref{it:fc5} of Theorem \ref{thm:fundfc}. 
Combining this fact with  (\ref{eq:PiHcom}), we obtain
$$
\norm{\Pihigh (\varphi_0 v)}_{\Hilb^{\sharp}} \lesssim  \norm{\Pihigh (\varphi_0 g)}_{\Hilb^{\sharp}}+ \big\|\Pihigh [\refop,\varphi_0] v\big\|_{\Hilb^{\sharp}}.
$$
Writing $\refop\Pihigh  = \Pihigh  + (\refop - I)\Pihigh $ and using (\ref{eq:eqPiH}) again, we obtain
$$
\norm{\Pihigh (\varphi_0 v)}_{\Hilb^{\sharp}} + \big\|\refop \Pihigh (\varphi_0 v)\big\|_{\Hilb^{\sharp}}  \lesssim  \norm{\Pihigh (\varphi_0 g)}_{\Hilb^{\sharp}}+ \big\|\Pihigh [\refop,\varphi_0] v\big\|_{\Hilb^{\sharp}}.
$$
Hence, by (\ref{eq:boundPi})
\begin{align}
\norm{\vHnear}_{\Hilb^{\sharp}}+ \big\|\refop \vHnear\big\|_{\Hilb^{\sharp}}  &\lesssim  \norm{\varphi g}_{\Hilb^{\sharp}}+ \big\|\Pihigh [\refop,\varphi_0] v\big\|_{\Hilb^{\sharp}} \nonumber  \\
& \lesssim \norm{ g}_{\cH(\Otr)}+ \big\|\Pihigh [\refop,\varphi_0] v\big\|_{\Hilb^{\sharp}}. \label{eq:high1}
\end{align}
We now seek to convert the $\|\refop \vHnear\|_{\Hilb^{\sharp}}$ on the left-hand side of this last bound into $\|P_{\hsc,\theta} \vHnear\|_{\Hilb^{\sharp}}$ using that $\refop=P_{\hsc,\theta}$ on $\supp \, {\varphi}_0$ and pseudolocality of the functional calculus.
With $\widetilde{\varphi}_0\in C^\infty_{\rm comp}(\Rea^d;[0,1])$ defined as above,
the definition of $\vHnear$ \eqref{eq:vH_decomp1}, the second support property in \eqref{eq:support_prop}, and Lemma \ref{lem:funcloc1} then imply that 
\begin{align*}
\vHnear = \widetilde{\varphi}_0 \Pihightheta (\varphi_0 v) + \residualD {\varphi}_0 v.
\end{align*}
By \eqref{eq:support_prop1} and a further use of Lemma \ref{lem:funcloc1},
\begin{align*}
P_{\hsc,\theta}\vHnear &= P_{\hsc,\theta} \widetilde{\varphi}_0 \Pihightheta (\varphi_0 v) + P_{\hsc,\theta} \residualD \widetilde{\varphi}_0 v\\
&= \refop \widetilde{\varphi}_0 \Pihightheta (\varphi_0 v) + P_{\hsc,\theta} \residualD {\varphi}_0 v\\
&
= \refop \Pihightheta (\varphi_0 v) + \refop \residualD {\varphi}_0 v+P_{\hsc,\theta} \residualD {\varphi}_0 v\\
&= \refop \vHnear + \refop \residualD {\varphi}_0 v+P_{\hsc,\theta} \residualD {\varphi}_0 v.
\end{align*}
Therefore, by the resolvent estimate \eqref{eq:res_PML} and the bound \eqref{eq:high1}, 
\begin{align}\nonumber
\norm{\vHnear}_{\Hilb^{\sharp}}+ \big\|P_{\hsc,\theta} \vHnear\big\|_{\Hilb^{\sharp}}
&\lesssim \norm{\vHnear}_{\Hilb^{\sharp}}+ \big\|\refop \vHnear\big\|_{\Hilb^{\sharp}} + \|g\|_{\cH(\Otr)}\\
& \lesssim \norm{ g}_{\cH(\Otr)}+ \big\|\Pihigh [\refop,\varphi_0] v\big\|_{\Hilb^{\sharp}}. \label{eq:high1a}
\end{align}

\paragraph{Step 2:~Viewing $\Pihigh [\refop,\varphi_0]$ as a semiclassical pseudodifferential operator on $\mathbb T^d_{R_{\sharp}}$.}

To prove \eqref{eq:vHnear_bound} from \eqref{eq:high1a}, it therefore remains to bound the commutator term $\Pihigh [\refop,\varphi_0] u$. Since $[\refop,\varphi_0]$ is supported away from $B_{R_0}$, we 
can write the high-frequency cut-off in terms of a semiclassical pseudodifferential operator thanks to Lemma \ref{lem:funcloc2}. 

Recall that  $\varphi_0$ is compactly supported in $B_{R_1}$ and equal to one near $B_{R_0}$, 
Let $\phi \in C^\infty_{\rm comp}(\mathbb R^d;[0,1])$ be supported in $B_{R_1}$, equal to zero near $B_{R_0}$, and such that
\begin{equation} \label{eq:suppphi}
\phi \equiv 1 \,\,\text{ near } \,\,\supp \nabla \varphi_0.
\end{equation}
Then, since $\refop=Q_\hsc$ on $\supp \,\varphi_0$,
\begin{equation} \label{eq:comsupp}
[\refop, \varphi_0] =[Q_\hsc, \varphi_0]=[Q_\hsc, \varphi_0]\phi =\phi[Q_\hsc, \varphi_0]= \phi[Q_\hsc, \varphi_0]\phi.
\end{equation}

Let $\chi \in C_{\rm comp}^\infty(\mathbb R^d)$ be supported in $B_{R_{1}}$, equal to zero near $B_{R_0}$, and equal to one near $\supp \phi$. Using (\ref{eq:comsupp}) and Lemma \ref{lem:funcloc1} with $\psi_1 =1-\chi$ and $\psi_2=\chi\phi=\phi$, we obtain that
\begin{align} 
\Pihigh [\refop,\varphi_0] =\Pihigh \phi[Q_\hsc,\varphi_0] \phi
&= \chi \Pihigh \chi \phi [Q_\hsc,\varphi_0] \phi  + \residualD \nonumber \\
 &= \chi \Pihigh \chi [Q_\hsc,\varphi_0] \phi  + \residualD.
 \label{eq:PiHproper}
\end{align}
Lemma \ref{lem:funcloc2} with $f(\refop) = \psi_\mu(\refop) = \Pilow$ implies that $\Pilow^\Psi:= \psi_\mu(Q_\hsc) \in \Psi^{-\infty}_\hsc(\mathbb T^d_{R_\sharp})$ satisfies
$$
\chi \Pilow \chi = \chi \Pilow^\Psi \chi  + \residualD.
$$ 
Hence, taking $\Pihigh^{\Psi} := I - \Pilow^{\Psi}=(1-\psi_\mu)(Q_\hsc) \in \Psi^0_\hsc(\mathbb T^d_{R_{\sharp}})$,
\begin{equation} \label{eq:PiHP}
\chi \Pihigh \chi = \chi \Pihigh^{\Psi}\chi  + \residualD
\end{equation}
i.e., modulo negligible terms, $\chi \Pihigh \chi $ is a high-frequency cut-off defined from the semiclassical pseudodifferential calculus. 
We here emphasise  that, since $\chi$ is supported in $B_{R_{1}}$ and vanishes near $B_{R_0}$, $\chi\Pihigh^{\Psi}\chi $ can be seen as
an element of \emph{both}  $\mathcal L(\mathcal H ^{\sharp})$ \emph{and} $\Psi^0_\hsc(\mathbb T^d_{R_{\sharp}})$.

\ble
With 
$\Pilow^{\Psi}:= \psi_\mu(Q_\hsc)$ and 
$\Pihigh^{\Psi} := (1-\psi_\mu)(Q_\hsc)$,
\begin{equation} \label{eq:PiLPWF}
\WFh  \Pilow^\Psi \subset  q_\hsc^{-1}\big(\supp \,\psi_\mu\big) = \{|q_\hsc| \leq 2 \mu\}
\eeq
and 
\beq\label{eq:PiHPWF}
\WFh  \Pihigh^\Psi \subset  q_\hsc^{-1}\big(\supp  (1- \psi_\mu)\big) = \{|q_\hsc| \geq \mu\}.
\end{equation}
\ele

\bpf[Reference for the proof]
See \cite[Lemma 3.1]{LSW4}, where this is proved using Lemma \ref{lem:funcloc2}.
\epf

\

Now, by (\ref{eq:PiHproper}) and (\ref{eq:PiHP}), for any $N$ and any $m$,
$$
\big\Vert \Pihigh [\refop,\varphi_0] v \big\Vert_{\mathcal H^{\sharp}} 
\leq \big\Vert\chi\Pihigh^{\Psi} \chi [Q_\hsc,\varphi_0] \phi v \big\Vert_{\mathcal H^{\sharp}} + C_{N,m} \hsc^N \big\Vert  [Q_\hsc,\varphi_0]   \phi v \big\Vert_{\mathcal D_\hsc^{\sharp, -m}} +  C'_{N} \hsc^N \big\Vert \widetilde \phi v \big\Vert_{\mathcal H^{\sharp}},
$$
with $\widetilde \phi$ compactly supported in $B_{R_{1}} \backslash B_{R_0}$ and equal to one on $\supp\, \phi$. Taking $m=1$, then $N=M+1$ and using the resolvent estimate (\ref{eq:res_PML}) we get
 \begin{align} \nonumber
\big\Vert \Pihigh [\refop,\varphi_0] v \big\Vert_{\mathcal H^{\sharp}} &
\leq 
 \big\Vert\chi \Pihigh^{\Psi}  \chi [Q_\hsc,\varphi_0] \phi v \big\Vert_{\mathcal H^{\sharp}} + C''_{M+1} \hsc^{M+1} \big\Vert  \widetilde \phi v \big\Vert_{\mathcal H} \nonumber \\ \nonumber
&\lesssim \big\Vert \chi \Pihigh^{\Psi} \chi [Q_\hsc,\varphi_0] \phi v \big\Vert_{\mathcal H^{\sharp}} +  \big\Vert  g \big\Vert_{\mathcal H},\\
& =\big\Vert \chi \Pihigh^{\Psi} \chi [Q_\hsc,\varphi_0] \phi v \big\Vert_{\mathcal H^{\sharp}} +  \big\Vert  g \big\Vert_{\mathcal H}.
 \label{eq:apprPiH}
\end{align}

\paragraph{Step 3:~A semiclassical elliptic estimate in $\mathbb T^d_{R_{\sharp}}$.}

Combining \eqref{eq:high1a} and \eqref{eq:apprPiH}, we see that to prove \eqref{eq:vHnear_bound} we only need to bound 
$ \chi\Pihigh^{\Psi} \chi[Q_\hsc  ,\varphi_0] \phi v$ in $L^2(\mathbb T_{R_{\sharp}}^d)$.
To do this, we use the semiclassical elliptic parametrix construction given by Theorem \ref{thm:para}. 

\begin{lemma}
The operator $Q_\hsc - 1 $ is semiclassically elliptic
on 
$\WFh(\hsc^{-1} \chi\Pihigh^{\Psi}\chi[Q_\hsc ,  \varphi_0])$.
\end{lemma}
\begin{proof}
By (\ref{eq:WFprod}), (\ref{eq:support}),  (\ref{eq:PiHPWF}), and \eqref{eq:John1}, 
$$
\WFh \big(\hsc^{-1} \chi\Pihigh^{\Psi}\chi[Q_\hsc , \varphi_0]\big) \subset \WFh  {\Pihigh^{\Psi}} 
 \subset \{ q_\hsc \geq \mu \}.
$$
But, on $\{ q_\hsc \geq \mu\}$, by definition of $\mu$ (\ref{eq:newdefmu}),
$
\langle \xi \rangle^{-2} (q_\hsc(x,\xi) - 1)\geq c_{\rm ell} >0,
$
and the proof is complete.
\end{proof}

Since $\hsc^{-1} \chi\Pihigh^{\Psi}\chi[Q_\hsc, \varphi_0] \in \Psi^1_\hsc(\mathbb T^d_{R_{\sharp}})$ by Theorem \ref{thm:basicP},
we can therefore apply the elliptic parametrix construction given by Theorem \ref{thm:para} with $A = \hsc^{-1}\chi \Pihigh^{\Psi}\chi[Q_\hsc- 1, \varphi_0]$, $B = Q_\hsc- 1$, and $\ell = 1$, $m= 2$. Hence, there exists $S \in \Psi^{-1}_\hsc(\mathbb T^d_{R_{\sharp}})$ and $R = O(\hsc^\infty)_{\Psi^{-\infty}_\hsc}$ with
\begin{equation} \label{eq:wfQ}
\WFh  S \subset \WFh  \big( \hsc^{-1} \Pihigh^{\Psi}[Q_\hsc, \varphi_0] \big)
\quad\tand\quad
\chi\Pihigh^{\Psi}\chi[Q_\hsc, \varphi_0] =  \hsc S ( Q_\hsc- 1) + R.
\eeq
We apply both sides of this identity to $\phi v$ and then use 
that $\phi$ is equal to zero 
near $B_{R_0}$ and supported in $B_{R_1}$, and thus $Q_\hsc= P_\hsc=P_{\hsc, \theta}$ on $\supp\, \phi$
; the result is that
\begin{align}
\chi\Pihigh^{\Psi}\chi[Q_\hsc, \varphi_0] \phi v =  \hsc S ( Q_\hsc- 1) \phi v + R \phi v \nonumber 
&=  \hsc S \phi ( Q_\hsc- 1) v +\hsc S [ Q_\hsc, \phi] v + R \phi v  \nonumber\\
&=  \hsc S \phi g +\hsc S [ Q_\hsc, \phi] v + R \phi v. \label{eq:parau}
\end{align}
The following lemma combined with (\ref{eq:WFdis}) shows that 
\begin{equation} \label{eq:commneg}
S [ Q_\hsc, \phi] = O(\hsc^\infty)_{\Psi^{-\infty}_\hsc}.
\end{equation}

\begin{lemma}\label{lem:WFintersect}
$
\WFh  S \cap \WFh   [ Q_\hsc, \phi] = \emptyset.
$
\end{lemma}
\begin{proof}
By (\ref{eq:wfQ}) and the definition of $Q_\hsc$ \eqref{eq:Qdef},
\beqs
\WFh  S \subset \WFh  [Q_\hsc, \varphi_0] \subset 
\big\{ (x,\xi) \,:\, x\in 
\supp  \nabla \varphi_0, \,\,
\xi \in\mathbb R^d\big\}.
\eeqs
Similarly, 
$$
\WFh   [ Q_\hsc, \phi] \subset
\big\{ (x,\xi) \,:\, x\in 
\supp  \nabla \phi_0, \,\,
\xi \in\mathbb R^d\big\}.
$$
Now, by (\ref{eq:suppphi}), $\supp  \nabla \varphi_0$ and $\supp  \nabla \phi$ are disjoint, and the result follows.
\end{proof}

\

Therefore, by (\ref{eq:parau}), (\ref{eq:commneg}) and the definition of $O(\hsc^\infty)_{\Psi^{-\infty}_\hsc}$ (\ref{eq:residual}), for any $N$, there exists $C_N, C'_N >0$ such that
\begin{align*}
\Vert \chi\Pihigh^{\Psi}\chi[Q_\hsc, \varphi_0] \phi v \Vert_{L^2(\mathbb T^d_{R_{\sharp}})}
&\lesssim \hsc \Vert S \phi g \Vert_{L^2(\mathbb T^d_{R_{\sharp}})} + C_N \hsc^N \Vert \widetilde \phi v \Vert_{L^2(\mathbb T^d_{R_{\sharp}})} + C'_N \hsc^N \Vert \phi v \Vert_{L^2(\mathbb T^d_{R_{\sharp}})} \\
&\lesssim \hsc \Vert S \phi g \Vert_{L^2(\mathbb T^d_{R_{\sharp}})} + C_N \hsc^N \Vert \widetilde \phi v \Vert_{\mathcal H} + C'_N \hsc^N \Vert \phi v \Vert_{\mathcal H},
\end{align*}
where $\widetilde \phi$ is compactly supported in $B_{R_1} \backslash B_{R_0}$ and equal to one on $\supp\, \phi$.
Taking $N:=M+1$, using the resolvent estimate (\ref{eq:res_PML}),
and then using that $S \in \Psi^{-1}(\mathbb T^d_{R_{\sharp}}) \subset  \Psi^{0}(\mathbb T^d_{R_{\sharp}})$ together with Part (iii) of Theorem \ref{thm:basicP}, 
we  obtain that
\begin{align*}\nonumber
\Vert \chi\Pihigh^{\Psi}\chi[Q_\hsc, \varphi_0] \phi v \Vert_{L^2(\mathbb T^d_{R_{\sharp}})} \lesssim \hsc \Vert S \phi g \Vert_{L^2(\mathbb T^d_{R_{\sharp}})} + \hsc \Vert g \Vert_{\mathcal H(\Otr)} 
 \lesssim \hsc \Vert \phi g \Vert_{L^2(\mathbb T^d_{R_{\sharp}})} + \hsc \Vert g \Vert_{\mathcal H(\Otr)}
\lesssim \hsc \Vert g\Vert_{\cH}.
\end{align*}
Combining this last estimate with (\ref{eq:high1a})  and (\ref{eq:apprPiH}) we obtain the desired bound 
\eqref{eq:vHnear_bound} on $\vHnear$.

\paragraph{Step 4:~Obtaining the bound \eqref{eq:vHfar_bound} on $\vHfar$ using the ideas from Steps 2 and 3.}

We now show that 
\beq\label{eq:Monday2}
\vHfar:=  \widetilde \varphi_{\rm tr} \widetilde \varphi_{1}
\Pihigh
\varphi_1 \varphi_{\rm tr}v 
=\widetilde \varphi_{\rm tr} \widetilde \varphi_{1} 
\Pihigh
\widetilde \varphi_{1} \widetilde \varphi_\tr
\varphi_1 \varphi_{\rm tr}v 
\eeq
 satisfies the bound \eqref{eq:vHfar_bound}.
Since $ \widetilde \varphi_{1} $ is supported away from $B_{R_0}$, 
exactly as in Step 2,
$\Pilow^\Psi := \psi_\mu(Q_\hsc)\in \Psi^{-\infty}_\hsc(\mathbb T^d_{R_\sharp})$
  and $\Pihigh^\Psi := (1-\psi_\mu)(Q_\hsc)\in \Psi^{0}_\hsc(\mathbb T^d_{R_\sharp})$
 satisfy
\beq\label{eq:Monday1}
 \widetilde \varphi_{1} \Pilow  \widetilde \varphi_{1} 
  = \widetilde \varphi_{1} 
  \Pilow^\Psi 
   \widetilde \varphi_{1}   + \residualD
 \quad\tand\quad
  \widetilde \varphi_{1}\Pihigh
  \widetilde \varphi_{1} 
  =
   \widetilde \varphi_{1} 
  \Pihigh^{\Psi}
   \widetilde \varphi_{1}
+\residualD.
\eeq
Now, by \eqref{eq:Monday2}, \eqref{eq:Monday1}, and the facts that $\widetilde{\varphi}_\tr\varphi_\tr = \varphi_\tr$ and $\widetilde \varphi_1\varphi_1 =\varphi_1$,
\beqs
\vHfar = \widetilde \varphi_{\rm tr} \widetilde \varphi_{1} 
\Pihigh^\Psi
\varphi_1 \varphi_{\rm tr}v + \residual \varphi_\tr v.
\eeqs

\begin{lemma}\label{lem:Gaston_late}
The operator $\widetilde{Q}_{\hsc,\theta} - 1 $ is semiclassically elliptic
on 
$\WFh(\widetilde \varphi_{1} \widetilde \varphi_{\rm tr}
  \Pihigh^{\Psi}
 \varphi_{1}  \varphi_{\rm tr})$.
\end{lemma}

\begin{proof}
First recall that $\supp (\varphi_1 \varphi_\tr)\subset \supp (\widetilde \varphi_{1} \widetilde \varphi_{\rm tr}) \subset B_{R_1(1 +3 \delta)}\setminus B_{R_0}$.
Using this property, along with  (\ref{eq:WFprod}), (\ref{eq:support}), (\ref{eq:PiHPWF}), and \eqref{eq:John1}, 
we find that 
\begin{align*}
\WFh(\widetilde \varphi_{1} \widetilde \varphi_{\rm tr}
  \Pihigh^{\Psi}
\varphi_{1} \varphi_{\rm tr})
&\subset \Big\{ (x,\xi) \, :\, x\in B_{R_1(1 +3 \delta)}\setminus B_{R_0}\Big\} \cap 
\WFh   \Pihigh^{\Psi}\\
&\subset \Big\{ (x,\xi) \, :\, x\in B_{R_1(1 +3 \delta)}\setminus B_{R_0}\Big\} \cap 
\big\{ (x,\xi) \,:\, 
q_\hsc(x,\xi)\geq \mu\big\}.
\end{align*}
By Lemma \ref{lem:mu_elliptic}, $\widetilde{Q}_{\hsc,\theta}-1$ is semiclassically elliptic on the set on the right-hand side of the last displayed inclusion, and the proof is complete.
\end{proof}

\

We now apply Theorem \ref{thm:para} with $A = \widetilde \varphi_{1} \widetilde \varphi_{\rm tr}
  \Pihigh^{\Psi}
\varphi_{1} \varphi_{\rm tr}$, $B = \widetilde{Q}_{\hsc,\theta}- 1$, $\ell = 0$, and $m= 2$; observe that the assumptions of Theorem \ref{thm:para} are then satisfied by  Lemma \ref{lem:Gaston_late}.   
      Hence, there exists $\widetilde{S} \in \Psi^{-2}_\hsc(\mathbb T^d_{R_{\sharp}})$ and $\widetilde{R} = O(\hsc^\infty)_{\Psi^{-\infty}_\hsc}$ with
\begin{equation} \label{eq:wfQ2}
\WFh  \widetilde{S} \subset
\WFh(\widetilde \varphi_{1} \widetilde \varphi_{\rm tr}
  \Pihigh^{\Psi}
\varphi_{1} \varphi_{\rm tr})
\quad\tand\quad
\widetilde \varphi_{1} \widetilde \varphi_{\rm tr}
  \Pihigh^{\Psi}
 \varphi_{1} \varphi_{\rm tr}=  \widetilde{S} ( \widetilde{Q}_{\hsc,\theta}- 1) + \widetilde{R}.
\eeq
We now apply the equality in \eqref{eq:wfQ2} to $\widetilde{\chi}v$ where $\widetilde{\chi}\in C^\infty_{\rm comp}(\Rea^d;[0,1])$ is such that $\widetilde{\chi}\equiv 1$ on a neighbourhood of $ \supp ({\varphi}_1 {\varphi}_{\tr})$ and $\supp\, \widetilde{\chi} \subset B_{R_1(1+3\delta)}\setminus B_{R_0}$; thus
\begin{align}\nonumber
\widetilde \varphi_{1} \widetilde \varphi_{\rm tr}
  \Pihigh^{\Psi}
\varphi_{1}  \varphi_{\rm tr}v
&=
\widetilde \varphi_{1} \widetilde \varphi_{\rm tr}
  \Pihigh^{\Psi}
 \varphi_{1}  \varphi_{\rm tr}\widetilde{\chi}v\\
&  =  \widetilde{S} ( \widetilde{Q}_{\hsc,\theta}- 1)\widetilde{\chi}v + \widetilde{R}\widetilde{\chi}v=  \widetilde{S} \widetilde{\chi}( \widetilde{Q}_{\hsc,\theta}- 1)v +
\widetilde{S}[\widetilde{Q}_{\hsc,\theta},\widetilde{\chi}]v + 
 \widetilde{R}\widetilde{\chi}v.\nonumber
\end{align}
By construction $\widetilde{Q}_{\hsc,\theta}= P_{\hsc,\theta}$ on $B_{R_1(1+3\delta)}\setminus B_{R_0}$ (see \eqref{eq:widetildeQ} and \eqref{eq:widetildef}); thus $\widetilde{\chi}( \widetilde{Q}_{\hsc,\theta}- 1)v = \widetilde{\chi} g$ and
\begin{align}
\widetilde \varphi_{1} \widetilde \varphi_{\rm tr}
  \Pihigh^{\Psi}
\varphi_{1}  \varphi_{\rm tr}v
=\widetilde{S} \widetilde{\chi}g +
\widetilde{S}[\widetilde{Q}_{\hsc,\theta},\widetilde{\chi}]v + 
 \widetilde{R}\widetilde{\chi}v.\label{eq:laugh1}
\end{align}
Arguing exactly as in Lemma \ref{lem:WFintersect}, using \eqref{eq:wfQ2}
and the fact that $\supp\,\nabla\widetilde{\chi}\cap\supp ({\varphi}_1 {\varphi}_{\tr})=\emptyset$, we find that 
\beqs
\WFh \widetilde S \cap \WFh[ \widetilde{Q}_{\hsc,\theta}, \widetilde{\chi}] = \emptyset 
\quad\text{ and thus }\quad
\widetilde S[ \widetilde{Q}_{\hsc,\theta}, \widetilde{\chi}]= O(\hsc^\infty)_{\Psi^{-\infty}_\hsc}.
\eeqs
Using this in \eqref{eq:laugh1} and then 
taking the $H^2_\hsc(\mathbb{T}^d_{R_\sharp})$ norm,
using the definitions of $O(\hsc^\infty)_{\Psi^{-\infty}_\hsc}$ and $\residualD$, 
we obtain that, given $N>0$ there exists $C_{N}>0$ such that 
\begin{align*}
\big\|
\widetilde \varphi_{1} \widetilde \varphi_{\rm tr}
  \Pihigh^{\Psi}
\varphi_{1}  \varphi_{\rm tr}v
   \big\|_{H^2_\hsc(\mathbb{T}^d_{R_\sharp})}
    \lesssim \big\| \widetilde{S} \widetilde{\chi} g\big\|_{H^2_\hsc(\mathbb{T}^d_{R_\sharp})} + C_N \hsc^N \big\|\widetilde{\chi}_{\rm alt} v \big\|_{L^2(\mathbb{T}^d_{R_\sharp})}
   \end{align*}
  where $\widetilde{\chi}_{\rm alt}$ is compactly supported in $B_{R_\sharp}\setminus B_{R_0}$ and equal to one on a neighbourhood of $\supp\,\widetilde{\chi}$.
By Part (iii) of Theorem \ref{thm:basicP}, and the fact that $\widetilde{S} \in \Psi^{-2}_\hsc(\mathbb T^d_{R_{\sharp}})$,
\beqs
\big\|
\widetilde \varphi_{1} \widetilde \varphi_{\rm tr}
  \Pihigh^{\Psi}
 \varphi_{1}  \varphi_{\rm tr}v
   \big\|_{H^2_\hsc(\mathbb{T}^d_{R_\sharp})}
     \lesssim  \|g\|_{\cH^\sharp} + C_N \hsc^N \| v\|_{\cH(\Otr)}.
\eeqs
The bound \eqref{eq:vHfar_bound} on 
   $\vHfar   :=\widetilde \varphi_{\rm tr} \widetilde \varphi_{1} 
\Pihigh
\varphi_1 \varphi_{\rm tr}v$ 
   then follows by combining this last inequality with the resolvent estimate \eqref{eq:res_PML}.

\subsection{Proof of the decomposition (\ref{eq:decomp2}) of $\vL$   (the low-frequency component)
and associated bounds on $\vAnear$ and $\vAfar$
} \label{subsec:low}

\subsubsection{Decomposing $\Pilow$ using Assumption 2 in Theorem \ref{thm:mainbb}}

By Assumption 2 in Theorem \ref{thm:mainbb}, there exists $E_\infty = \residualD$ with
\begin{equation} \label{eq:Edec}
\mathcal E(\refop) = E  + E_\infty,
\end{equation}
and the low-frequency estimate (\ref{eq:lowenest}) holds. 
By (\ref{eq:Lambda}) (a consequence of the definition of the constant $\Lambda$ \eqref{eq:Lambda}), $\mathcal E$ is nowhere zero on the support of $\psi_\mu$; therefore  the function $\psi_\mu/\mathcal E$ is well-defined and in $C_0(\mathbb R)$ (defined by \eqref{eq:C0}). The definition
of $\Pilow$ (\ref{eq:PiL}) and Part \ref{it:fc3} of Theorem \ref{thm:fundfc} imply that
\begin{equation} \label{eq:divbyE}
 \Pilow =\psi_\mu(\refop)  = \mathcal E(\refop)\left (\frac{1}{\mathcal E} \psi_\mu\right)(\refop) = E \circ\bigg( \left [\frac{1}{\mathcal E} \psi_\mu\bigg](\refop)\right)+ E_\infty\circ\left(\left [\frac{1}{\mathcal E} \psi_\mu\right](\refop)\right).
\end{equation}
Then, by Part \ref{it:fc5} of Theorem \ref{thm:fundfc} and the fact that $E_\infty = \residualD$,
\begin{equation} \label{eq:bigres}
E_\infty\circ\left(\left [\frac{1}{\mathcal E} \psi_\mu\right](\refop)\right)
  = \residualD.
\end{equation}

\subsubsection{The decomposition \eqref{eq:rho=1} of $\vL$ when $\rho=1$ in \eqref{eq:lowenest}}

We first assume that $\rho = 1$ and establish the decomposition \eqref{eq:rho=1}, together with the bounds (\ref{eq:decLF5}) and \eqref{eq:decLF5a}
on $\ulow$. 
In this case, we let
\beq\label{eq:vlow_def}
\vlow:=  E  \circ\left(\left [\frac{1}{\mathcal E} \psi_\mu\right](P^\sharp_\hsc)\right) \varphi_\tr v,
\eeq
so that  (\ref{eq:rho=1}) holds by \eqref{eq:divbyE} and \eqref{eq:bigres}.
Moreover, since $\vlow$ involves a compactly-supported function of $P^\sharp_\hsc$,
by the reasoning below \eqref{eq:boundPi}, $\vlow \in \mathcal{D}^{\sharp,\infty}_\hsc$.
Then, using (in this order) the low-frequency estimate (\ref{eq:lowenest}), Part \ref{it:fc5} of Theorem \ref{thm:fundfc}, and finally the resolvent estimate (\ref{eq:res_PML}), we get
\begin{align*}
 \Vert D(\alpha) \vlow  \Vert_{\mathcal H^{\sharp}} &= \N{ D(\alpha) 
 E  \circ\left(\left [\frac{1}{\mathcal E} \psi_\mu\right](P^\sharp_\hsc)\right) 
 \varphi_\tr v}_{\mathcal H^{\sharp}}
\leq C_{\mathcal E}(\alpha, \hsc) \N{ \left [\frac{1}{\mathcal E} \psi_\mu\right](P^\sharp_\hsc) \varphi_\tr v }_{\mathcal H^{\sharp} } \\
&\leq C_{\mathcal E}(\alpha, \hsc) \,\sup_{\lambda \in \mathbb R} \left| \frac{1}{\mathcal E(\lambda)} \psi_{\mu}(\lambda) \right| \Vert \varphi_\tr v \Vert_{\mathcal H^{\sharp} }
= C_{\mathcal E}(\alpha, \hsc)\, \sup_{\lambda \in \mathbb R} \left| \frac{1}{\mathcal E(\lambda)} \psi_{\mu}(\lambda) \right| \Vert \varphi_\tr v \Vert_{\mathcal H(\Otr) } \\
&\hspace{5.5cm}\lesssim  C_{\mathcal E}(\alpha, \hsc) \,\sup_{\lambda \in \mathbb R} \left| \frac{1}{\mathcal E(\lambda)} \psi_{\mu}(\lambda) \right| \hsc^{-1-M}\Vert g \Vert_{\mathcal H (\Otr)};
\end{align*}
thus (\ref{eq:decLF5}) holds. 
To establish \eqref{eq:decLF5a}, observe that 
\beq\label{eq:nosleep1}
\|v \|_{\cD^{\sharp,m}_\hsc((B_{R_1(1+\e)})^c)}\leq \big\|(1-\widetilde{\varphi}_\tr) v\big\|_{\cD^{\sharp,m}_\hsc},
\eeq
since $\widetilde\varphi_\tr \equiv 0$ on $(B_{R_1(1+\e)})^c$.
Then by (\ref{eq:Edec}), \eqref{eq:bigres}, Part \ref{it:fc3} of Theorem \ref{thm:fundfc}, pseudo-locality of the functional calculus (Lemma \ref{lem:funcloc1}), and the first support property in \eqref{eq:support_prop},
\begin{align}\nonumber
(1-\widetilde\varphi_\tr)
 E  \circ\left(\left [\frac{1}{\mathcal E} \psi_\mu\right](\refop)\right) 
\varphi_\tr
&= (1-\widetilde\varphi_\tr) \mathcal E(\refop)\left (\frac{1}{\mathcal E} \psi_\mu\right)(\refop) \varphi_\tr + \residualD \\
&= (1-\widetilde\varphi_\tr) \psi_\mu(\refop) \varphi_\tr + \residualD  = \residualD. \nonumber
\end{align}
The bound \eqref{eq:decLF5a} then follows by combining this with \eqref{eq:nosleep1} and the resolvent estimate \eqref{eq:res_PML}.

\bre[The decomposition is independent of $\mathcal{E}$ if $E_\infty=0$]
The last part of Theorem \ref{thm:mainbb} is the claim that when $E_\infty=0$, the decomposition is independent of $E$. To establish this in the case $\rho=1$, observe that   \eqref{eq:vlow_def} and \eqref{eq:divbyE} imply that if $E_\infty=0$, then $\vlow = \vL=\psi_\mu(\refop)\varphi_\tr v$ (which is independent of $\mathcal{E}$).
\ere

\subsubsection{Cut-off functions for the case $\rho\neq 1$}\label{subsubsec:cutoff}

We first define the cut-off functions used to bound $\vL$, displayed in Figure \ref{fig:line_func}.
Whereas the cut-off functions used in the bound on $\vH$ (in \S\ref{subsec:high}) were denoted $\varphi$, $\phi$, and $\chi$ (sometimes with tildes), in this section we use the notation $\rho_j$ and $\gamma_j$, $j=1,2$. Recall that $\rho$ is the cut-off function in the assumption \eqref{eq:lowenest}.

Given $R_0$, $R_1$, and $\rho$,
 let $\RfarA, \RfarB, \RlocB, \RlocA,$ be such that 
$R_0<\RfarA< \RfarB< \RlocB< \RlocA<R_1$ and
  $\rho = 1$ on a neighbourhood of $B_{\RlocA}$.
  
Let 
$\rho_{1}\in C^\infty_{\rm comp}(\mathbb T^d_{R_\sharp};[0,1])$ be such that $\supp  (1 - \rho_{1}) \subset (B_{\RfarB})^c $
 and $\supp \, \rho_{1} \Subset B_{\RlocB} $
Let $\rho_2 \in C^\infty_{\rm comp}(\mathbb T^d_{R_\sharp};[0,1])$ be supported in $B_{\RlocB}$ and such that $\rho_2 \equiv 1$ on $\supp  \,\rho_1$, i.e., 
 \beq\label{eq:support_prop2}
 \supp(1-\rho_2)\cap \supp\,\rho_1= \emptyset.
 \eeq

Let $\gamma_1 \in C^\infty(\mathbb T^d_{R_\sharp};[0,1])$ be such that $\gamma_1\equiv 0 $ on a neighbourhood of $B_{R_0}$, 
such that $\gamma_1 \equiv 1$ on a neighbourhood of $B_{R_1(1+2\delta)}\setminus B_{\RfarA}$, and $\gamma_1 \equiv 0$ on $(B_{R_1(1+3\delta)})^c$.
A key feature of this definition is that 
 \beq\label{eq:support_prop3}
\supp \,(1-\gamma_1)\cap\supp\big((1-\rho_1)\varphi_\tr\big)=\emptyset.
\eeq
Finally, let $\gamma_2 \in C^\infty(\mathbb T^d_{R_\sharp};[0,1])$ be equal to zero on $B_{\RfarB}$ and such that $\gamma_2 \equiv 1$ on $\supp  (1-\rho_1)$; i.e.,
 \beq\label{eq:support_prop4}
\supp \,(1-\gamma_2)\cap\supp\big(1-\rho_1)=\emptyset.
\eeq

\begin{figure}
\begin{center}
    \includegraphics[width=\textwidth]{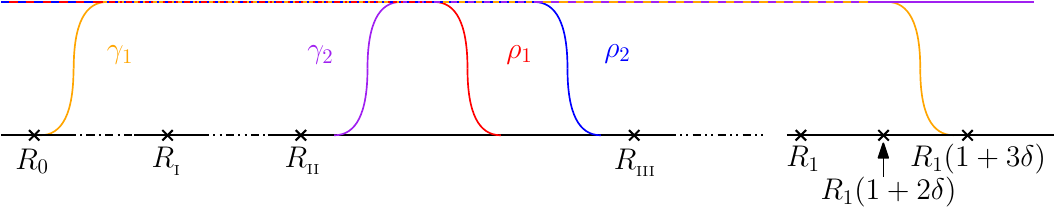}
  \end{center}
    \caption{The cut-off functions $\rho_1, \rho_2, \gamma_1, \gamma_2$ defined at the start of \S\ref{subsec:low}.}
  \label{fig:line_func}
\end{figure}

\subsubsection{Decomposing into parts that are ``near to'' or ``far from'' the black box when $\rho\neq 1$}\label{sec:decomp_vLow}
 
 We split $\vL$ in the following way, using the pseudo-locality of the functional calculus (i.e., Lemma \ref{lem:funcloc1}) and the support properties \eqref{eq:support_prop2} and \eqref{eq:support_prop3},
\begin{align*}\nonumber
\vL := 
\psi_\mu(\refop)
 \varphi_{\rm tr} v
&= \psi_\mu(\refop)
 \rho_1 \varphi_{\rm tr} v + \psi_\mu(\refop)
 (1-\rho_1) \varphi_{\rm tr} v \\
\nonumber
&= 
\psi_\mu(\refop)
 \rho_1 \varphi_{\rm tr} v + \gamma_1 
\psi_\mu(\refop)
 (1-\rho_1) \varphi_{\rm tr} v + \residualD \varphi_{\rm tr}v \\
&=: \vLnear + \vLfar + \residualD \varphi_{\rm tr}v.
\end{align*}
 
We now split $\vLnear$ and $\vLfar$ further, with this decomposition summarised in Figure \ref{fig:split_uL}.
We highlight that the arguments from here on are identical to the corresponding arguments in \cite{LSW4} (in \cite[\S3.3.3-\S3.3.4]{LSW4}).

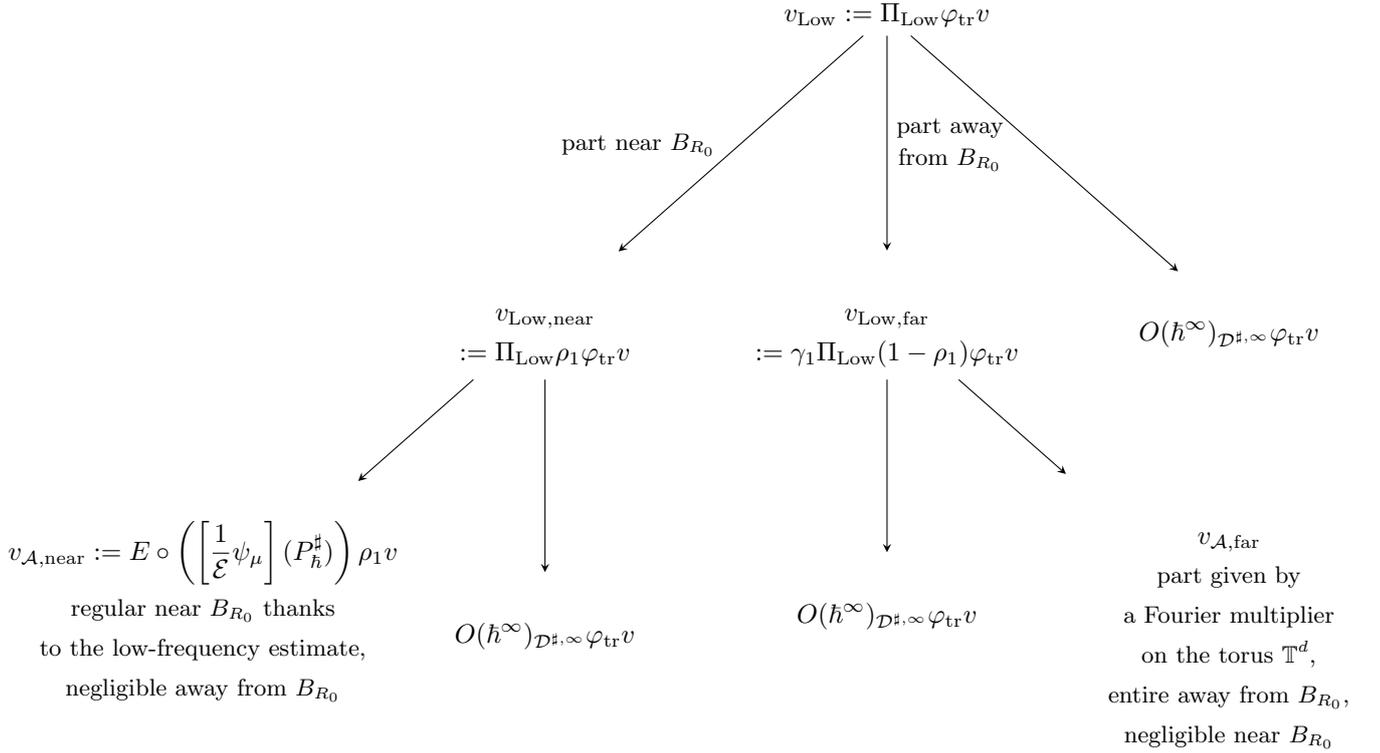
\begin{figure}
\hspace{-18em}
%\begin{adjustwidth*}{}{-18em}
 \begin{tikzpicture}%
  [>=stealth,
   shorten >=1pt,
   align = center,
   node distance=4cm and 4.5cm,
   on grid
  ]
  \node(0) {};
 \node (1)[right=of 0]  {$\vL := \Pilow \varphi_\tr v$};
\node (21) [below=of 0] {\begin{minipage}{\textwidth}
            \begin{gather*} 
\vLnear \\
:= \Pilow \rho_1 \varphi_\tr v
            \end{gather*}
        \end{minipage}};
\node (22) [right=of 21] {\begin{minipage}{\textwidth}
            \begin{gather*} 
            %\hspace{1cm}
            \vLfar  \\
            := \gamma_1\Pilow (1-\rho_1) \varphi_\tr v
            \end{gather*}
        \end{minipage}};
\node (23) [right=of 22] {\begin{minipage}{\textwidth}
            \begin{gather*} 
            %\hspace{1cm}
            \residualD \varphi_\tr v
            \end{gather*}
        \end{minipage}};

\node (32) [below=of 21] {
\begin{minipage}{\textwidth}
            \begin{gather*} 
\residualD \varphi_\tr v
            \end{gather*}
        \end{minipage}};
\node (31) [left=of 32] {\begin{minipage}{\textwidth}
            \begin{gather*} 
            \vAnear :=E \circ\left(\left [\frac{1}{\mathcal E} \psi_\mu\right](\refop)\right)  \rho_1 v \\ 
            \text{\small regular near $B_{R_0}$ thanks} \\
            \text{\small to the low-frequency estimate,} \\
            \text{\small negligible away from $B_{R_0}$} \\
            \end{gather*}
        \end{minipage}};
\node (33) [below=of 22] {\begin{minipage}{\textwidth}
            \begin{gather*} 
\residualD \varphi_\tr v\\   
            \end{gather*}
        \end{minipage}};
\node (34) [right=of 33] {\begin{minipage}{\textwidth}
            \begin{gather*} 
            \vAfar \\ 
            \text{\small part given by} \\
            \text{\small a Fourier multiplier} \\
            \text{\small on the torus $\mathbb T^d$,} \\
            \text{\small entire away from $B_{R_0}$,} \\
             \text{\small negligible near $B_{R_0}$}           
            \end{gather*}
        \end{minipage}};
\path[->]
(1) edge node [left] {\small part near $B_{R_0}$ \;} (21)
    edge node [right]{%\hspace{0.4cm} 
    {\small part away} \\ {\small from $B_{R_0}$}} (22)
(1) edge node {} (23)
(21) edge node {} (31)
    edge node {} (32)
(22) edge node {} (33)
    edge node {} (34)
%(32) edge node {} (4)
%(33) edge node {} (4)
;
\end{tikzpicture}
%\end{adjustwidth*}
\caption{The decomposition of $\vL$ when $\rho\neq 1$, described in \S\ref{sec:decomp_vLow}-\S\ref{sec:334}}  \label{fig:split_uL}
\end{figure}

\subsubsection{The part near the black box $\vLnear$}
By \eqref{eq:divbyE}, and \eqref{eq:bigres}, along with the fact that $\rho_1\varphi_\tr=\rho_1$, 
\begin{align}
\vLnear =\psi_\mu(\refop) 
\rho_1
v 
&= 
E \circ \left(\left[ \frac{1}{\mathcal E} \psi_\mu \right] (\refop)\right) \rho_1 
v
+\residualD \varphi_{\rm tr}v=: \vAnear + \residualD\varphi_\tr v. \label{eq:vAnear}
\end{align}
Since $\vAnear$ involves a compactly-supported function of $\refop$
by the reasoning below \eqref{eq:boundPi} $\vAnear$ is in $\mathcal{D}^{\sharp,\infty}_\hsc$.

\bre[The decomposition is independent of $\mathcal{E}$ if $E_\infty=0$]
The last part of Theorem \ref{thm:mainbb} is the claim that the decomposition is independent of $\cE$ if $E_\infty=0$.
To establish this when $\rho\neq 1$, observe that the only part of the definition of the decomposition where $\cE$ enters is in the decomposition $\vLnear =\vAnear + \residualD v$.
Furthermore, if $E_\infty=0$, then, by \eqref{eq:divbyE} and \eqref{eq:vAnear}, we can define 
$\vAnear:= \psi_\mu(\refop) \rho_1 v$ (which is independent of $\cE$) and have 
$\vLnear= \vAnear$.
\ere

\paragraph{Proof of (\ref{eq:decLF1}) and (\ref{eq:decLF2}) for 
$\vAnear$.} \label{sec:333}

Using (in this order) the definition of $\vAnear$ \eqref{eq:vAnear}, the fact that $\rho=1$ on $B_{\RlocA}$,
the low-frequency estimate (\ref{eq:lowenest}), Part \ref{it:fc5} of Theorem \ref{thm:fundfc}, and finally the resolvent estimate (\ref{eq:res_PML})  we obtain 
\begin{align*}
 \Vert D(\alpha) \vAnear  \Vert_{\mathcal H^{\sharp}(B_{\RlocA})} 
\leq \N{ \rho D(\alpha)
  E  \circ\left(\left [\frac{1}{\mathcal E} \psi_\mu\right](\refop)\right) 
\rho_1 v}_{\mathcal H^{\sharp} } 
&\leq C_{\mathcal E}(\alpha, \hsc) \N{
\left(\left [\frac{1}{\mathcal E} \psi_\mu\right](\refop)\right)
\rho_1 v}_{\mathcal H^{\sharp} }\\
&\hspace{-1cm}\leq C_{\mathcal E}(\alpha, \hsc) \sup_{\lambda \in \mathbb R} \left| \frac{1}{\mathcal E(\lambda)} \psi_{\mu}(\lambda) \right| \Vert \rho_1 v \Vert_{\mathcal H(\Otr) }\\
&\hspace{-1.5cm}\lesssim  C_{\mathcal E}(\alpha, \hsc) \sup_{\lambda \in \mathbb R} \left| \frac{1}{\mathcal E(\lambda)} \psi_{\mu}(\lambda) \right| \hsc^{-1-M}\Vert g \Vert_{\mathcal H(\Otr) };
\end{align*}
thus (\ref{eq:decLF1}) holds, where the $\sup_{\lambda \in \mathbb R}$ becomes $\sup_{\lambda \in [-\Lambda, \Lambda]}$ because of the support property \eqref{eq:Lambda} of $\psi_\mu$.

The proof of \eqref{eq:decLF2} is very similar to the proof of \eqref{eq:decLF5a} above.
Since $\rho_2 \equiv 0$ on $(B_{{\RlocB}})^c$,
\begin{align}
\Vert  \vAnear  \Vert_{\mathcal D^{m, \sharp}((B_{\RlocB})^c)}
&\leq \N{(1-\rho_2)
 E  \circ\left(\left [\frac{1}{\mathcal E} \psi_\mu\right](\refop)\right) 
\rho_1 v }_{\mathcal D^{m, \sharp}}.\label{eq:nosleep2}
\end{align}
By (\ref{eq:Edec}), Part \ref{it:fc3} of Theorem \ref{thm:fundfc}, pseudo-locality of the functional calculus (Lemma \ref{lem:funcloc1}), and the support property \eqref{eq:support_prop2},
\begin{align*}\nonumber
(1-\rho_2) 
 E  \circ\left(\left [\frac{1}{\mathcal E} \psi_\mu\right](\refop)\right) 
\rho_1 &= (1-\rho_2) \mathcal E(\refop)\left (\frac{1}{\mathcal E} \psi_\mu\right)(\refop) \rho_1 + \residualD \\
&= (1-\rho_2) \psi_\mu(\refop) \rho_1 + \residualD  = \residualD. 
\end{align*}
Combining this with \eqref{eq:nosleep2} and then using the resolvent estimate (\ref{eq:res_PML}), we obtain (\ref{eq:decLF2}).

\subsubsection{The term away from the black box $\vLfar$.} \label{sec:334}

We now study
\beq\label{eq:vLfar}
\vLfar:= 
\gamma_1 \Pilow(1-\rho_1) \varphi_{\rm tr} v
\eeq
which is in $\cD_\hsc^{\sharp, \infty}$ by the fact that $\Pilow : \cD^\sharp\to \cD_\hsc^{\sharp, \infty}$ (see \S\ref{subsec:abdec}) and the smoothness and support properties of $\gamma_1$ (see \S\ref{subsubsec:cutoff}).

\paragraph{Step 1:~
expressing $\vLfar$ in terms of $\vAfar$}

Since $\supp (1-\gamma_1)$ and $\supp(1-\rho_1)$ are disjoint (see Figure \ref{fig:line_func}), 
the pseudo-locality of the functional calculus given by Lemma \ref{lem:funcloc1} implies that
\begin{align*}
\gamma_1 \Pilow(1-\rho_1)
= \gamma_1 \Pilow \gamma_1 (1-\rho_1)+  \residualD.
\end{align*}
Therefore, by Lemma \ref{lem:funcloc2} (and exactly as in \S\ref{subsec:high}), $\Pilow^\Psi:= \psi_\mu(Q_\hsc) \in\Psi^{\infty}_\hsc(\mathbb T^d_{R_\sharp})$ with
\begin{equation} \label{eq:A:far:1}
\gamma_1 \Pilow (1-\rho_1) = \gamma_1 \Pilow^\Psi \gamma_1 (1-\rho_1) +  \residualD,
\end{equation}
and, by \eqref{eq:PiLPWF},
$\WFh  \Pilow^\Psi \subset \{ |q_\hsc| \leq 2\mu\}.$ 
Therefore, by \eqref{eq:Qnew},
there exists $\lambda>1$ such that 
\begin{equation} \label{eq:low:a:A}
\WFh  \Pilow^\Psi \subset 
\Big\{ (x,\xi) \,:\, x \in \mathbb T^d_{R_\sharp},\,\, \xi \in B_{\lambda/2}\Big\}.
\end{equation}
Now, let 
$\widetilde{\varphi} \in C^\infty_{\rm comp}(\Rea^d;[0,1])$ be supported in $[-\lambda^2,\lambda^2]$ and equal to one on $[-\lambda^2/4, \lambda^2/4]$. By
(\ref{eq:support})
and (\ref{eq:low:a:A}), $\WFh  \big( 1 - \Optorus(\widetilde{\varphi}(|\xi|^2)) \big) \cap\WFh  \big(\Pilow^\Psi \big)= \emptyset$. Therefore, by (\ref{eq:WFdis}), as operators on the torus,
\begin{equation}\label{eq:E1} 
\Pilow^\Psi = \Optorus(\widetilde{\varphi}(|\xi|^2)) \Pilow^\Psi + \residualP. 
\end{equation}
Since $\gamma_1=0$ on a neighbourhood of $B_{R_0}$, 
by the definitions of $P^\sharp$ \eqref{eq:defref}, $\|\cdot\|_{\mathcal{D}_\hsc^{\sharp,m}}$ \eqref{eq:BB:norms}, and $\|\cdot\|_{H_\hsc^{2m}(\mathbb{T}^d_{R^\sharp})}$ \eqref{eq:Hhnorm}, given $m>0$ there exists $C_j(m)>0, j=1,2,$ such that 
\beq\label{eq:normequiv}
C_1(m) \N{\gamma_1 w}_{\mathcal{D}_\hsc^{\sharp,m}} \leq \N{\gamma_1 w}_{H^{2m}_\hsc(\mathbb{T}^d_{R^\sharp})} \leq C_2(m) 
\N{\gamma_1 w}_{\mathcal{D}_\hsc^{\sharp,m}} \quad \tfa w \in \mathcal{D}_\hsc^{\sharp,m},
\eeq
and thus $\gamma_1 \residualP \gamma_1 = \residualD$. 
Therefore, combining this with \eqref{eq:E1} and (\ref{eq:A:far:1}), we obtain that
\beq\label{eq:PiLrho1}
\gamma_1 \Pilow (1-\rho_1) = \gamma_1 \Optorus(\widetilde{\varphi}(|\xi|^2)) \Pilow^\Psi \gamma_1 (1-\rho_1) + \residualD.
\eeq
We let
\beq\label{eq:uainf}
\vAfar := \gamma_1 \Optorus(\widetilde{\varphi}(|\xi|^2)) \Pilow^\Psi \gamma_1 (1-\rho_1)\varphi_\tr v,
\eeq
so that the combination of \eqref{eq:vLfar}, \eqref{eq:PiLrho1}, and \eqref{eq:uainf} implies that 
\beqs
\vLfar=\vAfar + \residualD \varphi_\tr v.
\eeqs
Observe that $\vAfar\in \cD(\Omega_\tr)$ because of the presence of $\gamma_1$ at the start of the expression (which causes $\vAfar$ to be zero on $\Gamma_\tr$).

\paragraph{Step 2:~proving that $\vAfar$ is regular in $(B_{\RfarA})^c$ (i.e., the bound (\ref{eq:decLF3})).}

By the definition  of $\vAfar$ \eqref{eq:uainf} and the fact that $\gamma_1 = 1$ on $(B_{\RfarA})^c$,
\begin{align} \label{eq:A:far:2}
\Vert \partial^\alpha \vAfar  \Vert_{\mathcal H((B_{\RfarA})^c)} &= \Big\Vert \partial^\alpha\Optorus(\widetilde{\varphi}(|\xi|^2)) \Pilow^\Psi \gamma_1 (1-\rho_1)\varphi_{\rm tr} v \Big\Vert_{\mathcal H((B_{\RfarA})^c)} \nonumber \\
&\leq  \Big\Vert \partial^\alpha\Optorus(\widetilde{\varphi}(|\xi|^2)) \Pilow^\Psi \gamma_1 (1-\rho_1)\varphi_{\rm tr} v \Big\Vert_{L^2(\mathbb T^d_{R_\sharp})}.
\end{align}
We now bound the right-hand side of \eqref{eq:A:far:2}. By Lemma \ref{lem:toruscalculus}, $\Optorus(\widetilde{\varphi}(|\xi|^2))$ is given 
as a Fourier multiplier on the torus (defined by \eqref{eq:A:mult}), i.e.,
\begin{equation} \label{eq:A:far:3}
\Optorus(\widetilde{\varphi}(|\xi|^2)) = \widetilde{\varphi}(-\hsc^2 \Delta).
\end{equation}
Let $w\in L^2(\mathbb T^d_{R_\sharp})$ be arbitrary, and let $\widehat w(j)$ be the Fourier coefficients of $w$. By (\ref{eq:A:mult}),
$$
\widetilde{\varphi}(-\hsc^2 \Delta) w =\sum_{j \in \ZZ^d} \widehat w(j) \widetilde{\varphi}(\hsc^2 \lvert j\rvert ^2 \pi^2/R_\sharp^2) e_j,
$$
where the normalised eigenvectors $e_j$ are defined by (\ref{eq:A:eigen}). Hence, for any multi-index $\alpha$,
\begin{align*}
\partial^\alpha \widetilde{\varphi}(-\hsc^2 \Delta) w &=\sum_{j \in \ZZ^d} \widehat w(j) \widetilde{\varphi}(\hsc^2 \lvert j\rvert ^2 \pi^2/R_\sharp^2) \left( \frac{\ri \pi j}{R_{\sharp}}\right)^\alpha e_j = \sum_{j \in \ZZ^d, \; |j| \leq \frac{\lambda R_{\sharp}}{{\hsc\pi}}} \widehat w(j) \widetilde{\varphi}(\hsc^2 \lvert j\rvert ^2 \pi^2/R_\sharp^2) \left( \frac{\ri \pi j}{R_{\sharp}}\right)^\alpha e_j,
\end{align*}
since $\widetilde{\varphi}$ is supported in $B(0, \lambda^2)$.
Therefore
\begin{align} \label{eq:low:A:3}
\Vert \partial^\alpha \widetilde{\varphi}(-\hsc^2 \Delta) w \Vert_{L^2(\mathbb T^d_{R_{\sharp}})}^2 &=  \sum_{j \in \ZZ^d, \; |j| \leq \frac{\lambda R_{\sharp}}{\hsc\pi}}\left| \widehat w(j) \widetilde{\varphi}(\hsc^2 \lvert j\rvert ^2 \pi^2/R_\sharp^2) \left( \frac{\ri \pi j}{R_{\sharp}}\right)^\alpha \right|^2  \nonumber \\&\leq \lambda^{2|\alpha|} \hsc^{-2|\alpha|}  \sum_{j \in \ZZ^d} | \widehat w(j) | ^2 = \lambda^{2|\alpha|} \hsc^{-2|\alpha|} \Vert w \Vert^2_{L^2(\mathbb T^d_{R_{\sharp}})}.
\end{align}
We now use (\ref{eq:low:A:3}) with 
$$
w:= \Pilow^\Psi \gamma_1 (1-\rho_1)\varphi_{\rm tr} v,
$$
and combine the resulting estimate with (\ref{eq:A:far:2}) and (\ref{eq:A:far:3}). Using the fact that $\Pilow^\Psi \in \Psi^{\infty}(\mathbb T^d_{R_\sharp})$, $\gamma_1 = 0$ on a neighbourhood of $B_{R_0}$, and the resolvent estimate \eqref{eq:res_PML}, we get
\begin{align*}
\Vert \partial^\alpha \vAfar  \Vert_{\mathcal H((B_{\RfarA})^c)} 
&\leq\lambda^{|\alpha|} \hsc^{-|\alpha|} \Vert \Pilow^\Psi \gamma_1 (1-\rho_1)\varphi_{\rm tr} v \Vert_{{L^2}(\mathbb T^d_{R_{\sharp}})} \lesssim \lambda^{|\alpha|} \hsc^{-|\alpha|} \Vert \gamma_1 (1-\rho_1)\varphi_{\rm tr} v \Vert_{{L^2}(\mathbb T^d_{R_{\sharp}})} \\
&= \lambda^{|\alpha|} \hsc^{-|\alpha|} \Vert \gamma_1 (1-\rho_1)\varphi_{\rm tr} v \Vert_{\mathcal H} \leq \lambda^{|\alpha|} \hsc^{-|\alpha|} \hsc^{-1-M} \Vert g \Vert_{\mathcal H(\Otr)};
\end{align*}
hence (\ref{eq:decLF3}) holds.

\paragraph{Step 3:~proving that $\vAfar$ is negligible in $B_{\RfarB}$ (i.e., the bound (\ref{eq:decLF4})).}

It therefore remains  to show (\ref{eq:decLF4}). 

By (\ref{eq:WFprod}), (\ref{eq:support}), and the support property \eqref{eq:support_prop4},
\beqs
\WFh  \Big( (1-\gamma_2) \Optorus(\widetilde{\varphi}(|\xi|^2)) \Pilow^\Psi \Big) \cap {\WFh  (1-\rho_1)} = \emptyset.
\eeqs
Then, by (\ref{eq:WFdis}),
$$
(1-\gamma_2) \Optorus(\widetilde{\varphi}(|\xi|^2)) \Pilow^\Psi (1-\rho_1) = O(\hsc^\infty)_{\Psi^{-\infty}_\hsc}
$$
as a pseudo-differential operator on the torus.
Multiplying by $\gamma_1$ on the right and on the left, and then using
the fact that $\gamma_1=0$ on $B_{R_0}$ and the norm equivalence
\eqref{eq:normequiv}, we find
\begin{equation} \label{eq:A:far:fin}
(1-\gamma_2) \gamma_1 \Optorus(\widetilde{\varphi}(|\xi|^2)) \Pilow^\Psi \gamma_1 (1-\rho_1) =   \residualD
\end{equation}
as an element of $\mathcal L(\mathcal H^\sharp)$. 
On the other hand, since $\gamma_2 = 0$ on a neighbourhood of  $B_{\RfarB}$,
$$
\Vert \vAfar \Vert_{\mathcal D^{\sharp, m}_\hsc (B_{\RfarB})} = \Vert (1-\gamma_2)\vAfar \Vert_{\mathcal D^{\sharp, m}_\hsc (B_{\RfarB})}.
$$
Then (\ref{eq:decLF4}) follows from combining this last equation with the definition of $\ulow^\infty$ \eqref{eq:uainf},  (\ref{eq:A:far:fin}),  and the resolvent estimate (\ref{eq:res_PML}). 

The proof of Theorem \ref{thm:mainbb} is now complete.

\section{Proofs of Theorems \ref{thm:LSW3} and \ref{thm:LSW4}}\label{sec:LSW34}

These proofs follow very closely the proofs of \cite[Theorem D]{LSW4} and \cite[Theorem B]{LSW4}, i.e., the analogous decompositions for outgoing Helmholtz solutions; this is because (as highlighted after Theorem \ref{thm:mainbb}) the assumptions of Theorem \ref{thm:mainbb} are (by design) the same as the assumptions of the abstract decomposition result in \cite[Theorem A]{LSW4}.
For completeness, we sketch here the ideas behind these proofs.

\subsection{Set-up common to both proofs}\label{sec:common}

Let $\hsc := k^{-1}$ and define $\mathcal H$ and $P_{\hsc}$ as in Lemma \ref{lem:obstacle} with $\obstacle_- = \emptyset$. 
By Lemma \ref{lem:obstacle}, $P_{\hsc}$ is a semiclassical black-box operator on $\mathcal H$. The reference operator is given by $P^\sharp_\hsc = -\hsc^{2} c_{\rm scat}^2 \nabla \cdot (A_{\rm scat}\nabla )$.
Let
\beq\label{eq:H}
\subsetH:= \big\{\hsc : \hsc=k^{-1} \text{ with } k\in K\big\}.
\eeq
The assumption that the solution operator is polynomially bounded (in the sense of Definition \ref{def:poly_bound}) 
means that 
the bound \eqref{eq:res} holds with $\subsetH$ given by \eqref{eq:H}; i.e., the assumption in Point 1 of Theorem \ref{thm:mainbb} is satisfied.
Define $P_{\hsc,\theta}$ by \eqref{e:defP}. In this notation, the PML problem \eqref{eq:PML} becomes $(P_{\hsc,\theta} -I)v = \hsc^2 g$.

\subsection{Sketch proof of Theorem \ref{thm:LSW3}}\label{sec:LSW3proof}

We now construct $\mathcal E$ and $E $ satisfying the assumptions in Point 2 of Theorem \ref{thm:mainbb} under Assumption \ref{ass:LSW3}.
Let $\Lambda>0$ be as in Theorem \ref{thm:mainbb}, and let $\mathcal E \in C^\infty_{\rm comp}(\mathbb R)$ be such that $\mathcal E = 1$ in $[-\Lambda, \Lambda]$, and $\mathcal E = 0$ outside
$[-2\Lambda, 2\Lambda]$. The results of Helffer-Robert  \cite{HeRo:83} 
imply that
$\mathcal E (P^\sharp_\hsc) = \mathcal E(-\hsc^{2} c_{\rm scat}^2 \nabla \cdot \big(A_{\rm scat}\nabla))$ is a pseudo-differential operator on $\mathbb T^d_{R_\sharp}$ (see the discussion in \S\ref{sec:idea2} under the paragraph ``Ingredient 5''). Then, arguing as in Step 1 in \S\ref{sec:334}, we obtain that there exists $\Lambda_0>0$ such that 
\begin{equation*} 
 \mathcal E(\refop)
   = \Optorus(\widetilde{\varphi}(|\xi|^2))   \mathcal E
   (\refop)
   + O(\hsc^\infty)_{\Psi^{-\infty}_\hsc}.
\end{equation*}
with $\widetilde{\varphi} \in C^\infty_{\rm comp}(\Rea^d;[0,1])$ supported in $B(0, \Lambda_0^2)$ and equal to one on $B(0,\Lambda_0^2/4)$.
By Lemma \ref{lem:toruscalculus},
\beqs
 \mathcal E(\refop)
   = \widetilde{\varphi}(-\hsc^2\Delta)  \mathcal E
   (\refop)
    + O(\hsc^\infty)_{\Psi^{-\infty}_\hsc},
\eeqs
so that 
\beq\label{eq:LSW3E0}
\tif \quad E  :=  \widetilde{\varphi}(-\hsc^2\Delta)  \mathcal E(\refop)
\quad\text{ then }\quad \mathcal E(P^\sharp_\hsc) = E  + \residual.
\eeq
We now need to show that an estimate of the form \eqref{eq:lowenest} is satisfied. Since $\widetilde{\varphi}$ is compactly supported in $B(0, \Lambda_0^2)$, the definition of $E $ \eqref{eq:LSW3E0} and the same argument used to show the bound \eqref{eq:low:A:3} imply that 
\beqs
\Vert \partial^\alpha E  v \Vert_{L^2(\mathbb T^d_{R_\sharp})} \leq  \Lambda_0^{|\alpha|} \hsc^{-|\alpha|} \Vert  \mathcal E
(\refop v)
\Vert_{L^2(\mathbb T^d_{R_\sharp})} 
\eeqs
for all $v \in L^2(\mathbb T^d_{R_\sharp})$ and multi-indices $\alpha$.
Then, since $\mathcal E(\refop)
\in \Psi^{-\infty}_\hsc(\mathbb T^d_{R_\sharp})$, Part (iii) of Theorem \ref{thm:basicP} implies that there exists $C>0$ such that 
\beqs
\Vert \partial^\alpha E  v \Vert_{L^2(\mathbb T^d_{R_\sharp})} \leq C \Lambda_0^{|\alpha|} \hsc^{-|\alpha|} \Vert v \Vert_{L^2(\mathbb T^d_{R_\sharp})}  
\eeqs
for all $v \in L^2(\mathbb T^d_{R_\sharp})$ and multi-indices $\alpha$.
Therefore, the assumption in Point 2 of Theorem \ref{thm:mainbb} is satisfied with $D(\alpha):=\partial^\alpha$, $C_{\mathcal E}(\alpha, \hsc) := C\Lambda_0^{|\alpha|} \hsc^{-|\alpha|} $ and $\rho = 1$. 

The bound \eqref{eq:decLFLSW3} on $\vlow$ follows immediately from 
\eqref{eq:decLF5}.
The bound \eqref{eq:decHFLSW3} on $\vhigh$ follows from \eqref{eq:decHF} after using 
(i) Green's identity and Lemma \ref{lem:strong_elliptic} to obtain a bound on the $H^1$ semi-norm, and then (ii) Lemma \ref{lem:strong_elliptic}  and $H^2$ regularity 
by \cite[Theorem 4.18]{Mc:00}. 

\subsection{Sketch proof of Theorem \ref{thm:LSW4}}\label{sec:LSW4proof}

Theorem \ref{thm:LSW4} is based on the following result, which is Theorem \ref{thm:mainbb} specialised to the case when the regularity estimate inside the black box comes from a heat flow estimate.

\begin{corollary} \label{cor:heat}
Let $P_{\hsc}$ be a semiclassical black-box operator on $\mathcal H$ satisfying the polynomial resolvent estimate 
 (\ref{eq:res}) in $\subsetH \subset (0, \hsc_0]$. 
Assume further that (i)  $P^{\sharp}_{\hsc}\geq a(\hsc) >0$ for some $a(\hsc) >0$, and (ii) 
for some $\alpha$-family of black-box differentiation operators $(D(\alpha))_{\alpha \in \frak A}$ (Definition \ref{def:BBdiff}), there exists $\rho\in C^\infty(\mathbb T^d_{R_\sharp})$ equal to one near $B_{R_0}$ such that, for some family of subsets $I(\hbar, \alpha) \subset [0, +\infty)$, the following localised heat-flow estimate holds, 
\begin{equation} \label{eq:heatflow}
\N{\rho D(\alpha) \re^{-t P^\sharp_{\hsc}}}_{\mathcal H^{\sharp} \rightarrow \mathcal H^{\sharp} } \leq C(\alpha, t, \hbar ) \quad \tfa  \alpha \in \frak A , \;t\in I(\hbar, \alpha), \; \hsc \in \subsetH.
\end{equation}

Given $\e>0$, there exist $\hsc_1>0$, $C_j>0,\, j=1,2,3,$ and $\lambda>1$ such that 
for all $R_{\tr}>(1+\e)R_1$, $B_{R_{\tr}}\subset  \Omega_{\tr}\Subset \mathbb{R}^d$ with Lipschitz boundary, $\e<\theta<\pi/2-\e$, 
all $g\in \cH(\Otr)$, 
and all $\hsc \in \subsetH \cap (0,\hsc_1]$, the following holds.
The solution $v \in \mathcal D(\Otr)$ to 
\beqs
(P_{\hsc,\theta} -I)v = g \,\,\ton \Otr\quad \tand \quad v= 0 \,\,\ton \Gamma_\tr
\eeqs
exists and is unique and there exists $\vhigh\in \cD(\Otr),\vlow\in \cD^{\sharp,\infty}_\hsc$ and $\vres\in \cD^{\sharp,\infty}_\hsc$ such that
\beqs
v= \vhigh + \vlow + \vres
\eeqs
and $\vhigh,\vlow,$ and $\vres$ satisfy the following properties.
The component $\vhigh\in \cD(\Otr)$ 
satisfies \eqref{eq:decHF}.
There exist
$\RfarA, \RfarB, \RlocB, \RlocA$ with
$R_0<\RfarA<\RfarB<\RlocB<\RlocA<R_1$ 
such that $\vlow\in \cD^{\sharp,\infty}_\hsc$ decomposes as
\beqs
\vlow = \vAnear + \vAfar,
\eeqs
where $ \vAnear\in \mathcal{D}^\sharp$ is regular near the black box and negligible away from it, in the sense that
\begin{equation} \label{eq:decLF1heat}
\Vert D(\alpha) \vAnear \Vert_{\mathcal H^{\sharp}(B_{\RlocA}) } \leq C_2 
\left(\inf_{t\in I(\hsc,\alpha)} C(\alpha,\hsc,t) \re^{\Lambda t}\right)
\hsc^{-1-M} \Vert g \Vert_{\mathcal H(\Otr)} \,\, \tfa  \hsc \in \subsetH \cap (0,\hsc_1], \alpha \in \mathfrak A,
\end{equation}
and, for any $N,m>0$ there exists $C_{N,m}>0$ (independent of $\theta$) such that \eqref{eq:decLF2} holds
and  $ \vAfar\in \cD(\Otr)$ is entire away from the black box and negligible near it, in the sense that \eqref{eq:decLF3} holds
and, for any $N,m>0$ there exists $C_{N,m}>0$ (independent of $\theta$) such that \eqref{eq:decLF4} holds.
Finally, $\vres\in \cD^{\sharp,\infty}_\hsc$ is negligible in the sense that for any $N,m>0$ there exists $C_{N,m}>0$ (independent of $\theta$) such that \eqref{eq:vresidual} holds.
\end{corollary}

The proof of Corollary \ref{cor:heat} is identical to the proof of \cite[Corollary 4.1]{LSW4}; since the proof is so short, however, we include it for completeness.

\

\begin{proof}[Proof of Corollary \ref{cor:heat}]
For $\alpha \in \mathfrak A$ and $\hbar \in \subsetH$, let $t\in I(\hbar, \alpha)$, and $\mathcal E_t(\lambda) := \re^{-t|\lambda|}$. Since $P^{\sharp}_{\hsc}\geq a(\hsc) >0$, $\operatorname{Sp}P^{\sharp}_{\hsc} \subset [a(\hsc), \infty)$. Therefore, by 
Parts  \ref{it:fc6} and  \ref{it:fc5} of Theorem \ref{thm:fundfc}, $\re^{-t P^\sharp_{\hsc}} = \mathcal E_t(P^\sharp_{\hsc})$. Such an $\mathcal E_t$ is in $C_0(\mathbb R)$, never vanishes, and satisfies (\ref{eq:lowenest}) with $E_t := \mathcal E_t(P^\sharp_{\hsc})$ and $C_{\mathcal E_t}(\alpha, \hsc) := C(\alpha, \hbar, t)$
by   (\ref{eq:heatflow}). From Theorem \ref{thm:mainbb}, we therefore obtain the above decomposition $\vlow, \vAnear,  \vAfar,  \vhigh$.
Since $\mathcal E_t(P^\sharp_{\hsc}) = E_t$ (i.e., $E_\infty=0$), by the final part of Theorem \ref{thm:mainbb}, the decomposition is constructed independently of $\mathcal E_t$, and hence independently of $t$. The result then follows, with the infimum in $t$ in (\ref{eq:decLF1heat}) coming from (\ref{eq:decLF1}) and the fact that this estimate in valid for any $t\in I(\hbar, \alpha)$.
\end{proof}

\

Theorem \ref{thm:LSW4} is proved using Corollary \ref{cor:heat} with the following heat-flow estimate as \eqref{eq:heatflow}.

\begin{theorem}\mythmname{Heat equation estimate from \cite{EsMoZh:17}}\label{thm:heat}
Suppose that 
Assumption \ref{ass:LSW4} holds with $A_{\rm scat}$ and $c_{\rm scat}$ analytic in 
$B_{R_*}$ for some $R_0<R_*<R_{\rm scat}$.
Let $P^\sharp_\hsc$ denote the associated black-box reference operator on the torus (as described in \S\ref{subsec:bb}).

Given $\rho \in C^\infty_{\rm comp}(\Rea^d;[0,1])$ with $\supp\,\rho \subset B_{R_*}$,
there exists $C>0$ such that 
for all $t \in (0, 1]$ and for all $\EMZ\in[0,1]$
\beq\label{eq:EMZheat}
\norm{\rho\pa^\alpha \re^{t \hsc^{-2}P_\hbar^\sharp} }_{L^2\to L^2}
 \leq 
\exp(t^{-\EMZ})
  \abs{\alpha}!\,  C^{\abs{\alpha}}t^{(\EMZ-1)|\alpha|/2}.
\eeq
  \end{theorem}

\bpf[References for the proof of Theorem \ref{thm:heat}]
Since the operator $\re^{t\hsc^{-2} P_\hbar^\sharp}$
is just the variable coefficient heat
operator for time $t$, the estimate \eqref{eq:EMZheat} can be
extracted from the heat equation bounds in \cite[Theorem 1.1 and Lemma
2.7]{EsMoZh:17}; see \cite[Proof of Theorem 4.3]{LSW4} for more
detail.
\epf

\

We therefore apply Corollary \ref{cor:heat} to the specific set up in \S\ref{sec:common}, noting that 
the heat-flow estimate (\ref{eq:heatflow}) is then satisfied with $D(\alpha) := \partial^\alpha$,
$$
C(\alpha, \hbar, t) := 
 \exp\big((\hsc^2t)^{-\EMZ}\big)
  \abs{\alpha}!\,  C^{\abs{\alpha}}  \big(\hsc^2 t)^{(\EMZ-1)|\alpha|/2},
 \quad\tand\quad I(\hbar, \alpha) := (0, \hbar ^{-2}]
$$
(the heat-flow given by the functional calculus, appearing in (\ref{eq:heatflow}), is indeed the solution of the heat equation; see, e.g., \cite[Theorem  VIII.7]{ReeSim72}).

To obtain Theorem \ref{thm:LSW4} from Corollary \ref{cor:heat}, we then only need to show that (i) $\vhigh$ satisfies \eqref{eq:decHFDir}, and (ii) $\vAnear$ satisfies \eqref{eq:decLFDir1}.
The proof of (i) is identical to the proof that $\vhigh$ in Theorem \ref{thm:LSW3} satisfies \eqref{eq:decHFLSW3}. 
For (ii), we carefully choose $t$ and $\tau$ as functions of $|\alpha|$ and $\hsc$ to obtain  \eqref{eq:decLFDir1}; 
for the details, see \cite[\S4.1]{LSW4}.

\appendix 

\section{Semiclassical pseudodifferential operators on the torus} \label{app:sct}

Recall that for $R_\sharp >0$,
$
 \mathbb T^d_{R_{\sharp}} := {\mathbb{R}^d}/{(2R_\sharp\mathbb{Z})^d}.
 $
This appendix reviews the material about semiclassical pseudodifferential operators on $\mathbb T^d_{R_{\sharp}}$
used in \S \ref{subsec:high}-\S\ref{subsec:low}, and appearing in Lemma  \ref{lem:funcloc2}, with our default references being  \cite{Zw:12} and \cite[Appendix E]{DyZw:19}. 

\paragraph{Semiclassical Sobolev spaces.} 
We consider functions or distributions on the torus as periodic
functions or distributions on $\mathbb R^d.$  To eliminate
confusion between Fourier series and integrals, for $f \in
L^2(\torus_{R_\sharp}^d)$  we define the Fourier coefficients for $j \in \ZZ^d$
\begin{equation} \label{eq:A:eigen}
\widehat f(j) := \int_{\torus_{R_\sharp}^d} f(x) \overline{e_j}(x) \, \rd x, \quad\text{ where }\quad
e_j(x) = (2R_\sharp)^{-d/2} \exp{\big(\ri\pi j\cdot  x/R_\sharp\big)}.
\end{equation}
The Fourier inversion formula 
and the action of the operator $(\hsc D)^\alpha$ on the torus are then, respectively, 
$$
f = \sum_{j \in \ZZ^d} \widehat f(j) e_j
\quad\tand\quad (\hsc D)^\alpha f  =\sum_{j \in \ZZ^d} (\hsc\pi j/R_\sharp)^\alpha \widehat f(j) e_j.
$$
We work on the spaces defined by the boundedness of these operators, namely
\beq
H_\hsc ^ m (\mathbb T_{R_\sharp}^d):= \Big\{ u\in L^2(\mathbb T_{R_\sharp}^d),
\; \langle j \rangle^m \widehat f(j) \in \ell^2(\ZZ^d) \Big\},
\qquad
\label{eq:Hhnorm}
\Vert u \Vert_{H_\hsc^m(\mathbb T_{R_\sharp}^d)} ^2 :=\sum
\langle \hsc j\rangle^{2m}
 \lvert \widehat f(j)\rvert^2,
\eeq
where $\langle j \rangle := (1+ |j|^2)^{1/2}$; see \cite[\S8.3]{Zw:12}, \cite[\S E.1.8]{DyZw:19}. 
In this appendix, we abbreviate $H_\hsc ^ m (\mathbb T_{R_\sharp}^d)$ to $H_\hsc ^ m$ and $L^2(\mathbb T_{R_\sharp}^d)$ to $L^2$.

Since  these spaces are defined for positive integer $m$ by boundedness of $(hD)^\alpha$ with
$\lvert \alpha \rvert =m$ (and can be extended to $m \in \RR$ by
interpolation and duality), they agree with localized versions of the
corresponding spaces on $\mathbb R^d$ defined by the semiclassical Fourier transform 
$$
\mathcal F_{\hsc}u(\xi) := \int_{\mathbb R^d} \exp\big( -\ri x \cdot \xi/\hsc\big)
u(x) \, \rd x
\quad\tand\quad
\Vert u \Vert_{H_\hsc^m(\mathbb R^d)} ^2 := (2\pi \hsc)^{-d} \int_{\mathbb R^d} \langle \xi \rangle^m  |\mathcal F_\hsc u(\xi)|^2 \, \rd \xi.
$$

\paragraph{Phase space.}
The set of all possible positions $x$ and momenta (i.e.~Fourier variables) $\xi$ is denoted by $T^*\mathbb T_{R_\sharp}^d$; this is known informally as ``phase space". Strictly, $T^*\mathbb T_{R_\sharp}^d :=\mathbb T_{R_\sharp}^d\times (\mathbb R^d)^*$, but 
for our purposes, we can consider $T^*\mathbb T_{R_\sharp}^d$ as $\{(x,\xi)
: \bx\in \mathbb T^d_{R_\sharp}, \xi\in\Rea^d\}$.  We also use the analogous
notation for $T^* \mathbb R^d$ where appropriate.

To deal uniformly near fiber-infinity with the behavior of
functions on phase space, we also consider the \emph{radial
  compactification} in the fibers of this space,
$
\Tbar^* \torus^d_{R_\sharp}:= \mathbb T^d \times B^d,
$
where $B^d$ denotes the closed unit ball, considered as the closure of the
image of $\mathbb R^d$ under the radial compactification map 
$\RC: \xi \mapsto \xi/(1+\langle \xi
\rangle);$
see \cite[\S E.1.3]{DyZw:19}.
Near the boundary of the
ball, $\lvert \xi\rvert^{-1}\circ \RC^{-1}$ is a smooth function, vanishing to
first order at the boundary, with $(\lvert \xi\rvert^{-1}\circ \RC^{-1}, \widehat\xi\circ\RC^{-1})$
thus furnishing local coordinates on the ball near its boundary.  The boundary of the
ball should be considered as a sphere at infinity consisting of all
possible \emph{directions} of the momentum variable.  Where
appropriate (e.g., in dealing with finite values of $\xi$ only), we abuse notation by dropping the composition with $\RC$ from our
notation and simply identifying $\mathbb R^d$ with the interior of $B^d$.

\paragraph{Symbols, quantisation, and semiclassical pseudodifferential operators.} 

A symbol on $\mathbb R^d$ is a function on $T^*\mathbb{R}^d$ that is also allowed to depend on $\hsc$, and thus can be considered as an $\hsc$-dependent family of functions.
Such a family $a=(a_\hsc)_{0<\hsc\leq\hsc_0}$, with $a_\hsc \in C^\infty({\mathbb R^d})$, 
is a \emph{symbol
of order $m$} on the $\mathbb R^d$, written as $a\in S^m(\mathbb R^d)$, if for any multi-indices $\alpha, \beta$
\beqs
| \partial_x^\alpha \partial^\beta_\xi a(x,\xi) | \leq C_{\alpha, \beta} \langle \xi \rangle^{m -|\beta|} \quad\tfa (x,\xi) \in T^*\mathbb R^d \text{ and for all } 0<\hsc\leq \hsc_0,
\eeqs
where $C_{\alpha, \beta}$ does not depend on $\hsc$; see \cite[p.~207]{Zw:12}, \cite[\S E.1.2]{DyZw:19}. 

For $a \in S^m(\mathbb R^d)$, we define the \emph{semiclassical quantisation} of $a$ on $\mathbb R^d$, denoted by $\operatorname{Op}_{\hsc}(a)$
\beq \label{Oph}
\big(\operatorname{Op}_{\hsc}(a) v\big)(x) := (2\pi \hsc)^{-d} \int_{\mathbb R^d} \int_{ \Rea^d} 
\exp\big(\ri (x-y)\cdot\xi/\hsc\big)\,
a(x,\xi) v(y) \,\rd y  \rd \xi;
\eeq
\cite[\S4.1]{Zw:12} \cite[Page 543]{DyZw:19}. The integral in
\eqref{Oph} need not converge, and can be understood \emph{either} as
an oscillatory integral in the sense of \cite[\S3.6]{Zw:12},
\cite[\S7.8]{Ho:83}, \emph{or} as an iterated integral, with the $y$
integration performed first; see \cite[Page 543]{DyZw:19}.  It can be
shown that for any symbol $a,$ $\Op_\hsc(a)$ preserves Schwartz
functions, and extends by duality to act on tempered distributions
\cite[\S4.4]{Zw:12}

We use below that if $a=a(\xi)$
depends only on $\xi$, then
$
\Op_{\hsc}(a) = \F_\hsc^{-1} M_a \F_{\hsc},
$
where $M_a$ denotes multiplication by $a$; i.e., in this case $\Op_{\hsc}(a)$
is just a Fourier multiplier on $\mathbb R^d.$

We now return to considering the torus:~if $a(x,\xi) \in S^m(\mathbb R^d)$ and
is periodic, and if $v$ is a distribution on the torus, we can view
$v$ as a periodic (hence, tempered) distribution on $\mathbb R^d,$ and
define
$$
\big(\operatorname{Op}^{\torus^d_{R_\sharp}}_{\hsc}(a) v\big)=\big(\operatorname{Op}_{\hsc}(a) v\big),
$$
since the right side is again periodic; for details see, e.g., \cite[\S 5.3.1]{Zw:12}.  

If $A$ can be written in the form above, i.\,e.\ $A = \operatorname{Op}^{\torus^d_{R_\sharp}}_{\hsc}(a)$ with $a\in S^m$, we say that $A$ is a \emph{semiclassical pseudodifferential operator of order $m$} on the torus and
we write $A \in \Psi_{\hsc}^m(\mathbb T_{R_\sharp}^d)$; furthermore that we often abbreviate $ \Psi_{\hsc}^m(\mathbb T_{R_\sharp}^d)$ to $\Psi_{\hsc}^m$ in this Appendix. We use the notation $a \in \hsc^l S^m$  if $\hsc^{-l} a \in S^m$; similarly 
$A \in \hsc^l \Psi_\hsc^m$ if 
$\hsc^{-l}A \in \Psi_\hsc^m$.

\begin{theorem}\mythmname{Composition and mapping properties of
semiclassical pseudodifferential operators \cite[Theorem 8.10]{Zw:12}, \cite[Proposition E.17 and Proposition E.19]{DyZw:19}}\label{thm:basicP} If $A\in \Psi_{\hsc}^{m_1}$ and $B  \in \Psi_{\hsc}^{m_2}$, then

(i) $AB \in \Psi_{\hsc}^{m_1+m_2}$,

(ii) $[A,B] \in \hsc\Psi_{\hsc}^{m_1+m_2-1}$,

(iii) For any $s \in \mathbb R$, $A$ is bounded uniformly in $\hsc$ as an operator from $H_\hsc^s$ to $H_\hsc^{s-m_1}$.
\end{theorem}

\paragraph{Residual class.} 
We say that $A =O(\hsc^\infty)_{\Psi^{-\infty}_\hsc}$ if, for any $s>0$ and $N\geq 1$, there exists $C_{s,N}>0$ such that
\beq \label{eq:residual}
\Vert A \Vert_{H_\hsc^{-s} \rightarrow H_\hsc^{s}} \leq C_{N,s} \hsc^N;
\eeq
i.e.~$A\in \Psi_\hsc^{-\infty}$ and furthermore all of its operator norms are bounded by any algebraic power of $\hsc$.

\paragraph{Principal symbol $\sigma_{\hsc}$.}
Let the quotient space $ S^m/\hsc S^{m-1}$ be defined by identifying elements 
of  $S^m$ that differ only by an element of $\hsc S^{m-1}$. 
For any $m$, there is a linear, surjective map
$
\sigma^m_{\hsc}:\Psi_\hsc ^m \to S^m/\hsc S^{m-1},
$
called the \emph{principal symbol map}, 
such that, for $a\in S^m$,
\beq\label{eq:symbolone}
\sigma_\hsc^m\big(\Op^{\torus^d_{R_\sharp}}_\hsc(a)\big) = a \quad\text{ mod } \hsc S^{m-1};
\eeq
see \cite[Page 213]{Zw:12}, \cite[Proposition E.14]{DyZw:19} (observe that \eqref{eq:symbolone} implies that 
$\operatorname{ker}(\sigma^m_{\hsc}) = \hsc\Psi_\hsc ^{m-1}$).

When applying the map $\sigma^m_{\hsc}$ to 
elements of $\Psi^m_\hsc$, we denote it by $\sigma_{\hsc}$ (i.e.~we omit the $m$ dependence) and we use $\sigma_{\hsc}(A)$ to denote one of the representatives
in $S^m$ (with the results we use then independent of the choice of representative).

\paragraph{Operator wavefront set $\WFh$.}  
We say that $(x_0,\zeta_0) \in {\Tbar}^*\mathbb T_{R_\sharp}^d$ is \emph{not}
in the \emph{semiclassical operator wavefront set} of
$A = \Optorus(a) \in \Psi_{\hsc}^m$, denoted by
$\WFh A$, if there exists a neighbourhood $U$ of
$(x_0,\zeta_0)$ such that for all multi-indices $\alpha, \beta$ and all
$N\geq 1$ there exists $C_{\alpha,\beta,U,N}>0$ (independent of
$\hsc$) such that, for all $0<\hsc\leq \hsc_0$,
\beq\label{eq:microsupport} |\partial_x^\alpha \partial_\xi^\beta
a(x,\xi)| \leq C_{\alpha, \beta, U, N} \hsc^N \langle \xi \rangle^{-N}\quad\tfa (x,\RC(\xi))\in U.
\eeq
For $\zeta_0=\RC(\xi_0)$ in the interior of $B^d,$ the factor $\ang{\xi}^{-N}$ is moot, and the definition
merely says that
outside its semiclassical operator wavefront set an operator
is the quantization of a symbol that vanishes faster than any algebraic power of $\hsc$; see \cite[Page
194]{Zw:12}, \cite[Definition E.27]{DyZw:19}.  For $\zeta_0\in \partial
B^d=S^{d-1},$ by contrast, the definition says that the symbol decays rapidly
in a conic neighborhood of the direction $\zeta_0,$ in addition to decaying in $\hsc.$

Properties of the
semiclassical operator wavefront set that we use in \S
\ref{subsec:high} and \S\ref{subsec:low} are 
\begin{equation} \label{eq:WFempt}
\WFh  A = \emptyset \quad\text{ if and only if }\quad A = O(\hsc^\infty)_{\Psi^{-\infty}_\hsc},
\end{equation}
(see  \cite[E.2.3]{DyZw:19}),
\beq \label{eq:WFprod} \WFh
(AB) \subset \WFh A \cap
\WFh B, \eeq (see \cite[\S 8.4]{Zw:12},
\cite[E.2.5]{DyZw:19}), 
\beq \label{eq:WFdis} \WFh
(A) \cap \WFh (B) = \emptyset \quad\text{ implies that } \quad AB =
O(\hsc^\infty)_{\Psi^{-\infty}_\hsc}, 
\eeq
 (as a consequence of \eqref{eq:WFempt} and \eqref{eq:WFprod}), and \beq\label{eq:support}
\WFh\big( \Op_\hsc(a)\big) \subset \supp \,a \eeq
(since
$(\supp \, a)^c \subset (\WFh( \Op_\hsc(a)))^c$ by
\eqref{eq:microsupport}).

\paragraph{Ellipticity.} 
We  say that $B\in \Psi_\hsc^m$ is \emph{elliptic} at $(x_0,
  \zeta_0) \in  \Tbar^*\mathbb T_{R_\sharp}^d$ if there exists a
  neighborhood $U$ of $(x_0, \zeta_0)$ and $c>0$, independent of $\hsc$, such that 
\beqs
\langle \xi \rangle^{-m} \big|\sigma_\hsc(B)(x,\xi)\big| \geq c \quad \tfa (x,\RC(\xi))\in U  \text{ and for all } 0<\hsc\leq \hsc_0.
\eeqs

A key feature of elliptic operators is that they are microlocally
invertible; this is reflected in the following result.
\newcounter{fnnumber}
\begin{theorem}\mythmname{Elliptic parametrix \cite[Proposition E.32]{DyZw:19}} \footnote{We highlight that working in a compact manifold
allows us to dispense with the proper-support assumption appearing in \cite[\S4]{LSW3}, \cite[Proposition E.32, Theorem E.33]{DyZw:19}.\label{fn:comp}} \setcounter{fnnumber}{\thefootnote}
 \label{thm:para} 
Let $A \in \Psi_\hsc^{\ell}(\mathbb T_{R_\sharp}^d)$ and $B \in \Psi_\hsc^{m}(\mathbb T_{R_\sharp}^d)$ be such 
that $B$ is elliptic on $\operatorname{WF_\hsc}(A)$.
Then there exist $S, S' \in \Psi_\hsc^{\ell-m}(\mathbb T_{R_\sharp}^d)$ such that
$$
A = BS + O(\hsc^\infty)_{\Psi^{-\infty}_\hsc} = S'B + O(\hsc^\infty)_{\Psi^{-\infty}_\hsc},
$$ 
with
$
\WFh  S \subset \WFh  A$  and 
$\WFh  S' \subset \WFh  A.
$
\end{theorem}

\paragraph{Functional Calculus.}

The main properties of the functional calculus in the black-box
  context are recalled in
  \S\ref{sec:BBFC}; here we record a simple result that we need about
  functions of the flat Laplacian.
  
For $f$ a Borel function, the operator $f(-\hsc^2 \Lap)$ is defined on
smooth functions on the torus (and indeed on distributions if $f$ has
polynomial growth) by the functional calculus for the flat Laplacian,
i.e., by the Fourier multiplier
\begin{equation} \label{eq:A:mult}
f(-\hsc^2 \Lap) v=\sum_{j \in \ZZ^d} \widehat v(j) f(\hsc^2 \lvert j\rvert ^2 \pi^2/R_{\sharp}^2) e_j.
\end{equation}
The following lemma shows that $f(-\hsc^2 \Lap)$ is precisely
the quantization of $f(\lvert\xi\rvert^2)$; since our quantization
procedure was defined in terms of Fourier transform rather than
Fourier series, this is not obvious a priori.

\begin{lemma}\mythmname{\cite[Lemma A.3]{LSW4}}\label{lem:toruscalculus}
  For $f \in S^m(\RR^1)$ (i.e., $f$ is a function of only one variable),
  $
f(-\hsc^2 \Lap) = \Op_{\hsc} f(\lvert\xi\rvert^2).
  $
  \end{lemma}

\section*{Acknowledgements}

We thank the referee for their constructive comments on the paper. JG acknowledges support from EPSRC grants EP/V001760/1 and EP/V051636/1.
EAS acknowledges support from EPSRC grant EP/1025995/1. 
JW was partially supported
by Simons Foundation grant 631302, NSF grant DMS--2054424, and a
Simons Fellowship.

\footnotesize{
\bibliographystyle{plain}
\bibliography{biblio_combined_sncwadditions}
}

\end{document}